\documentclass[12pt]{article}
\usepackage{amsmath,amsthm,amssymb,mathrsfs}
\usepackage{times}
\usepackage{enumerate}
\usepackage{verbatim,subfig,epsfig,color}

\newtheorem{thm}{Theorem}[section]

\newtheorem{lem}[thm]{Lemma}
\newtheorem{proposition}[thm]{Proposition}

\theoremstyle{definition}
\newtheorem{defin}[thm]{Definition}
\newtheorem{rem}[thm]{Remark}
\newtheorem{exa}[thm]{Example}

\numberwithin{equation}{section}

\DeclareMathOperator{\supp}{supp\,}
\DeclareMathOperator{\ebdiv}{div\,}
\DeclareMathOperator{\Rng}{Rng\,}
\DeclareMathOperator{\loc}{loc}
\DeclareMathOperator{\ebint}{int\,}

\allowdisplaybreaks

\begin{document}

\title{Minimizers of anisotropic surface tensions under gravity: higher dimensions via symmetrization}

\author{Eric Baer
\footnote{Massachusetts Institute of Technology, 77 Massachusetts Ave., Cambridge, MA 02139-4381; ebaer@math.mit.edu}}

\date{}

\maketitle

\begin{abstract}
We consider a variational model describing the shape of liquid drops and crystals under the influence of gravity, resting on a horizontal surface.  Making use of anisotropic symmetrization techniques, we establish existence, convexity and symmetry of minimizers for a class of surface tensions admissible to the symmetrization procedure.  In the case of smooth surface tensions, we obtain uniqueness of minimizers via an ODE characterization.

\end{abstract}

\section{Introduction}
In this work, we consider questions of existence, regularity and uniqueness for a class of variational problems describing the shape of liquid drops and crystals under the influence of gravity and supported by a horizontal surface, from the viewpoint of symmetrization techniques.  More precisely, fixing $N\geq 2$ we consider minimizers of the functional 
\begin{align*}
\mathscr{F}(E)&:=\mathscr{F}_{s}(E)+\mathscr{F}_c(E)+\mathscr{F}_{p}(E)
\end{align*}
among sets of finite perimeter $E\subset\mathbb{R}^N$ with 
\begin{align*}
E\subset \{x=(x_1,x_2,\cdots,x_N)\in\mathbb{R}^N:x_N>0\},
\end{align*}
and satisfying the volume constraint $|E|=m$ for some fixed $m>0$, where the terms $\mathscr{F}_s$, $\mathscr{F}_c$ and $\mathscr{F}_p$ are functionals respectively representing the internal surface tension of the drop or crystal, the contact energy between the shape and the supporting surface, and the gravitational potential energy.

For the purposes of our study, we will assume that the surface energy $\mathscr{F}_s$ takes the form
\begin{align*}
\mathscr{F}_{s}(E)&:=\int_{\partial^*E\cap\{x:x_N>0\}} f(\nu_E(x))d\mathcal{H}^{N-1}(x)
\end{align*}
for a given convex function $f:\mathbb{R}^N\rightarrow\mathbb{R}^+$ which is positively $1$-homogeneous (i.e. $f(\lambda x)=\lambda f(x)$ for all $\lambda>0, x\in\mathbb{R}^N$), and satisfies $f(x)>0$ for $|x|>0$.  Here, $\partial^*E$ denotes the reduced boundary of $E$ and for each $x\in\partial^*E$, $\nu_E(x)$ refers to the (measure theoretic) unit outer normal to $E$ at $x$.  We shall also impose some additional symmetry (Definition \ref{def_symm}) and admissibility (Definition \ref{def_adm}) constraints on the function $f$ which are adapted to our particular approach; as will be discussed shortly, these enable the application of a suitable codimension $N-1$ Steiner symmetrization procedure (see for instance Figure \ref{figsymm}).

Moreover, letting $\omega\in (-f(e_N),f(-e_N))$ be given, we shall assume that the contact energy takes the form
\begin{align*}
\mathscr{F}_{c}(E)&:=\omega\mathcal{H}^{N-1}(\partial^*E\cap \{x:x_N=0\}),
\end{align*}
while the gravitational potential is
\begin{align*}
\mathscr{F}_{p}(E)&:=\int_{E} x_Ndx,\quad E\in \mathcal{F}_m,
\end{align*}
where $\mathcal{F}_m$ denotes the collection of sets of finite perimeter 
\begin{align*}
E\subset \{x=(x',x_N)\in\mathbb{R}^N=\mathbb{R}^{N-1}\times\mathbb{R}:x_N>0\}
\end{align*}
which satisfy the volume constraint $|E|=m$ (see ($\ref{definition-Fm}$ below).

With this choice of energy functionals, the problem of finding minimizers to $\mathscr{F}$ is known as the {\it sessile drop} problem, and may be seen as an attempt to understand the balance between the energies involved.  The condition $\omega\in (-f(e_N),f(-e_N))$ is a natural requirement for the existence of minimizers, ensuring that it is not energetically preferred for minimizers to ``spread out'' into an infinitesimally thin sheet or to separate from the supporting surface (these two cases may be seen as endpoints of the contact angles depicted in Figures $1$ and $2$).

\begin{figure}
\centering
\def\svgwidth{0.72\columnwidth}

\begingroup
  \makeatletter
  \providecommand\color[2][]{%
    \errmessage{(Inkscape) Color is used for the text in Inkscape, but the package 'color.sty' is not loaded}
    \renewcommand\color[2][]{}%
  }
  \providecommand\transparent[1]{%
    \errmessage{(Inkscape) Transparency is used (non-zero) for the text in Inkscape, but the package 'transparent.sty' is not loaded}
    \renewcommand\transparent[1]{}%
  }
  \providecommand\rotatebox[2]{#2}
  \ifx\svgwidth\undefined
    \setlength{\unitlength}{841.88974609pt}
  \else
    \setlength{\unitlength}{\svgwidth}
  \fi
  \global\let\svgwidth\undefined
  \makeatother
  \begin{picture}(1,0.70707072)%
    \put(0,0){\includegraphics[width=\unitlength]{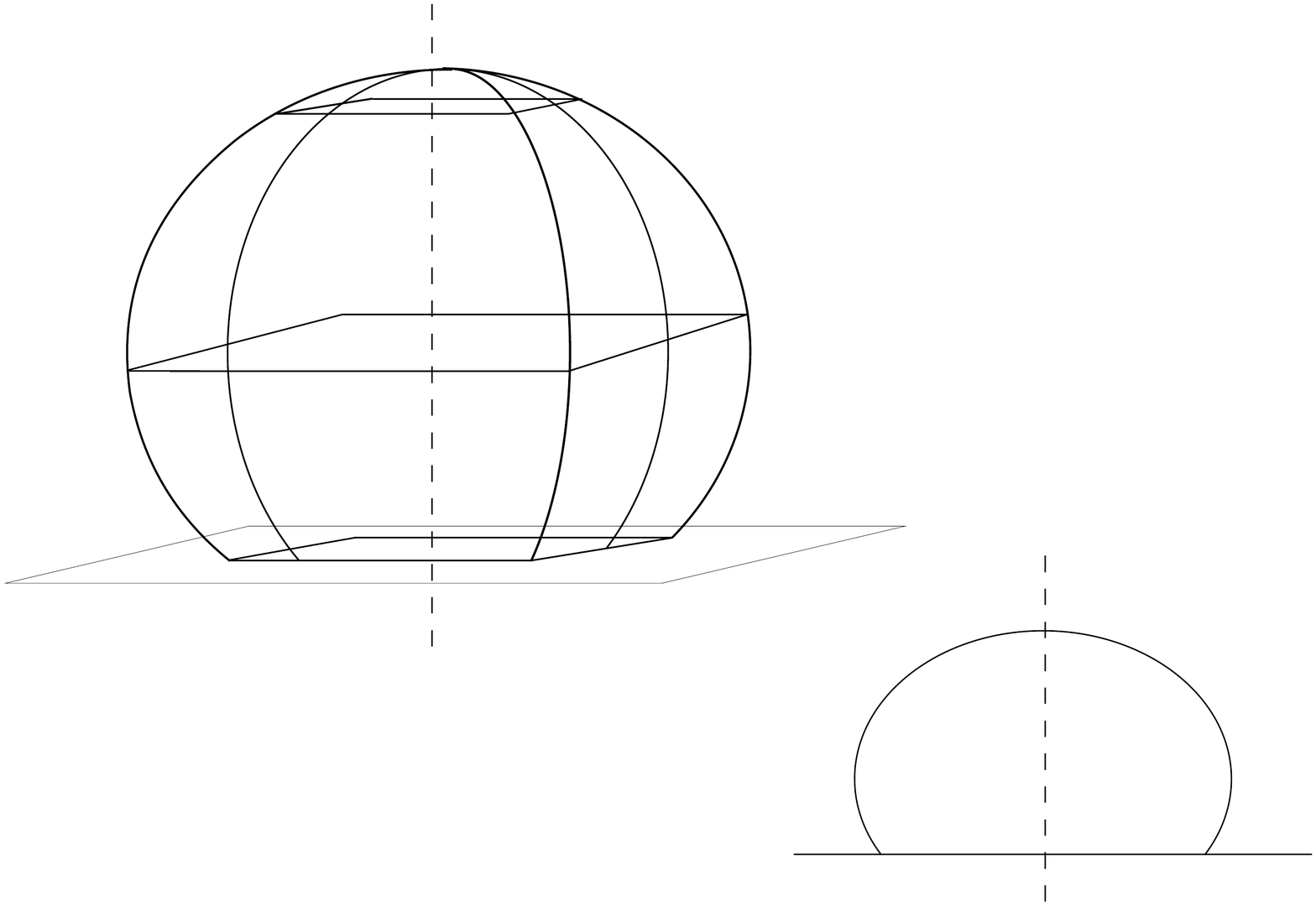}}%
    \put(0.32060343,0.65848492){\color[rgb]{0,0,0}\makebox(0,0)[lb]{\smash{$t$}}}%
    \put(0.75639376,0.28604739){\color[rgb]{0,0,0}\makebox(0,0)[lb]{\smash{$t$}}}%
    \put(0.82358611,0.05567367){\color[rgb]{0,0,0}\makebox(0,0)[lb]{\smash{$\alpha(t)$}}}%
    \put(0.54329807,0.47802551){\color[rgb]{0,0,0}\makebox(0,0)[lb]{\smash{$E_t=\alpha(t)K_h$}}}%
  \end{picture}%
\endgroup
\caption{Example of a minimizer of $\mathscr{F}$ with $\omega>0$, together with its associated profile.  The minimizer is axially symmetric (see Theorem $\ref{thm_min}$), and is therefore characterized by the curve $\{(\alpha(t),t):t>0\}\subset\mathbb{R}^2\}$, whose even reflection across the axis $x_1=0$ is depicted at right.\label{fig-mina}}
\end{figure}

The surface energy defined above is a variant of the Wulff functional
\begin{align}
E\mapsto \int_{\partial^*E} f(\nu_E(x))d\mathcal{H}^{N-1}(x),\label{wulff}
\end{align}
which is a common mathematical model for the shape of crystalline structures driven by surface tension.  When the weight function $f$ is given by $f(\nu)=|\nu|$, the functional ($\ref{wulff}$) reduces to the perimeter of $E$, while the functional $\mathscr{F}_{s}(E)$ reduces to the relative perimeter of $E$ with respect to the half space $\mathbb{R}^N_+$.  This setting is referred to as the {\it isotropic} case -- the dependence on the direction of the outer normal disappears.  On the other hand, crystalline shapes are modelled by piecewise linear weights $f$ (in which case the Wulff shape $K_f$ is polygonal).

In the isotropic setting, the study of minimizers of $\mathscr{F}$ is a classical problem and has been studied by a wide variety of authors.  In particular, we note the works of Gonzalez \cite{Gonzalez,Gonzalez2} in which symmetrization techniques are used to establish existence, symmetry and regularity for the isotropic sessile drop problem.  Subsequently, Gonzalez and Tamanini studied the convexity of minimizers \cite{GonzalezTamanini}.  Concerning stationary points for the functional $\mathscr{F}$, Wente \cite{Wente1,Wente2} established symmetry and stability results for such surfaces, while Finn \cite{Finn1,FinnBook} established uniqueness results for the symmetric sessile drop.  These works have formed the basis of a rich literature in the subject; see, for instance \cite{GonzalezMassariTamanini,Nikolov,CaffarelliMellet,CaffarelliMellet2,KosioPalmer,MelletNolen,Treinen,ElcratTreinen,ElcratKimTreinen}.

In the {\it anisotropic} case, when $f$ is no longer constant on $S^{N-1}$, the study of the shape of sessile drops is considerably more subtle.  The unique minimizer of ($\ref{wulff}$) with respect to a volume constraint of the form $|E|=m$ is well-known to be a convex set known as the Wulff shape \cite{Taylor,Wulff}; in particular, this set may be written as
\begin{align*}
K_f:=\bigcap_{\nu\in S^{N-1}} \{x\in\mathbb{R}^{N}:x\cdot \nu<f(\nu)\}.
\end{align*}
When the contact energy is taken into account (i.e. we consider minimizers of $\mathscr{F}_s+\mathscr{F}_c$), the minimizer to the resulting variational problem is a truncated Wulff shape; this is the result of the classical Winterbottom construction \cite{Winterbottom} (see also \cite{summertop}).

Much of the prior work in the anisotropic case concerns the planar case $N=2$; in particular, we point out the work of McCann \cite{McCann} (and the work of Okikiolu cited therein), in which for a general class of surface tensions and potentials (not necessarily restricted to the gravitational case considered here) minimizers are shown to consist of a countable union of connected components, each of which is convex and a minimizer among convex sets of the same mass.  Moreover, in the case of the half space with gravitational potential, Avron, Taylor and Zia have established convexity and uniqueness of minimizers \cite{AZT}.  However, in the higher dimensional setting $N>2$, much less is known.  We make particular mention of the recent work of Figalli and Maggi \cite{FM}, where (again for general potentials) the authors show that minimizers of sufficiently small mass are convex and uniformly close to the Wulff shape.

Motivated by the utility of symmetrization techniques in the isotropic case, the goal of the present work is to approach the problem with the aim of applying recent developments in anisotropic symmetrization, that is, notions of symmetrization adapted to the functional $\mathscr{F}_s$ (see for instance \cite{JS}, as well as Definition $\ref{def_symm}$ and Theorem \ref{thm_symm_f} below). 

\begin{figure}
\centering
\subfloat[(a)][]{\def\svgwidth{0.42\columnwidth}
\def\svgwidth{0.42\columnwidth}

\begingroup
  \makeatletter
  \providecommand\color[2][]{%
    \errmessage{(Inkscape) Color is used for the text in Inkscape, but the package 'color.sty' is not loaded}
    \renewcommand\color[2][]{}%
  }
  \providecommand\transparent[1]{%
    \errmessage{(Inkscape) Transparency is used (non-zero) for the text in Inkscape, but the package 'transparent.sty' is not loaded}
    \renewcommand\transparent[1]{}%
  }
  \providecommand\rotatebox[2]{#2}
  \ifx\svgwidth\undefined
    \setlength{\unitlength}{841.88974609pt}
  \else
    \setlength{\unitlength}{\svgwidth}
  \fi
  \global\let\svgwidth\undefined
  \makeatother
  \begin{picture}(1,0.70707072)%
    \put(0,0){\includegraphics[width=\unitlength]{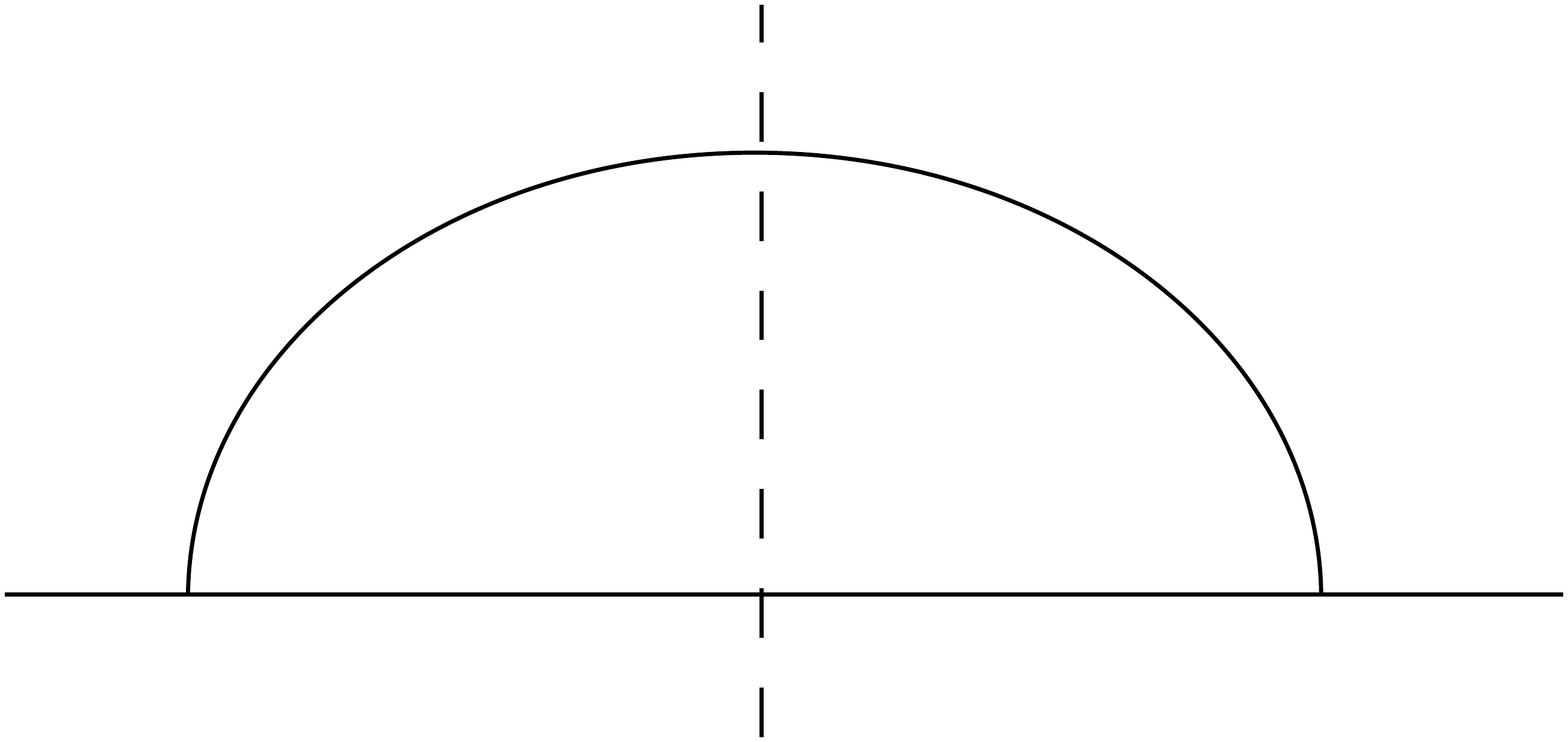}}%
  \end{picture}%
\endgroup}
\subfloat[(b)][]{\def\svgwidth{0.42\columnwidth}
\def\svgwidth{0.42\columnwidth}

\begingroup
  \makeatletter
  \providecommand\color[2][]{%
    \errmessage{(Inkscape) Color is used for the text in Inkscape, but the package 'color.sty' is not loaded}
    \renewcommand\color[2][]{}%
  }
  \providecommand\transparent[1]{%
    \errmessage{(Inkscape) Transparency is used (non-zero) for the text in Inkscape, but the package 'transparent.sty' is not loaded}
    \renewcommand\transparent[1]{}%
  }
  \providecommand\rotatebox[2]{#2}
  \ifx\svgwidth\undefined
    \setlength{\unitlength}{841.88974609pt}
  \else
    \setlength{\unitlength}{\svgwidth}
  \fi
  \global\let\svgwidth\undefined
  \makeatother
  \begin{picture}(1,0.70707072)%
    \put(0,0){\includegraphics[width=\unitlength]{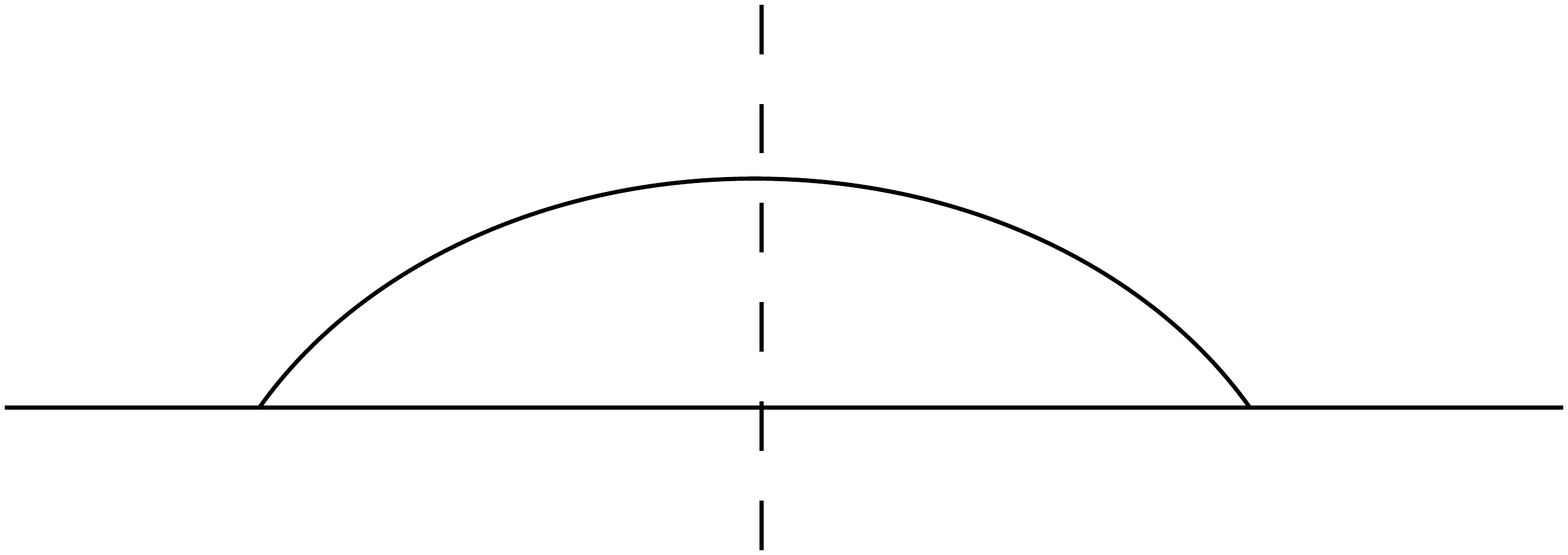}}%
  \end{picture}%
\endgroup}
\caption{Examples of profiles of minimizers of $\mathscr{F}$ with (a) $\omega=0$ and (b) $\omega<0$.  The contact angle of the minimizer with the supporting plane is determined by the parameter $\omega$ (see Theorem $\ref{thm_young}$), while the precise shape of the profile is the consequence of a balance between the Wulff shape associated to the function $f$ and the pull downward of the gravitational potential.\label{fig-minbc}}
\end{figure}

A first difficulty in implementing this approach is the use of a non-standard form of the anisotropic symmetrization, for which the identification of equality cases presents some subtlety.  Moreover, while in the isotropic case the theory of minimal surfaces may be applied to obtain analyticity of minimizers (which is then used to obtain convexity \cite{GonzalezTamanini}), such an approach is not feasible in our setting as a consequence of the fact that non-smooth weights are included in the class of admissible surface energies.  Our arguments therefore proceed in a different manner.  In particular, the main results of our study take the following form:
\begin{itemize}
\item Identify a suitable class of symmetrizable surface tensions and characterize the equality cases in the relevant symmetrization inequality.
\item Use the direct method of the Calculus of Variations to prove existence of minimizers.
\item Prove that symmetric minimizers consist of a single convex connected component.
\item Prove that all minimizers are symmetric (and therefore convex).
\item Characterize the profile of minimizers as the solution of an ODE along with boundary conditions (corresponding to the classical Young's equation describing the contact angle between a liquid drop and a surface).  
\item Under appropriate smoothness hypotheses, show that there exists a unique minimizing shape.
\end{itemize}

We remark that once the symmetrization framework is in place, the key step in this series of arguments is to obtain the convexity of symmetric minimizers.  For this, we use a delicate construction involving fine properties of sets of finite perimeter to show that such minimizers cannot have points of local non-convexity (if such points were present, we could construct a set with smaller energy).  In particular, this argument can be viewed as expressing the balance between the surface energy's preference for convexity (the Wulff shape is convex) with the preference for mass to be ``pulled downward'' by gravitational forces.  After establishing the convexity of symmetric minimizers, the ensuing additional regularity is used as a key tool to show that all minimizers are symmetric.  We refer to the discussion in Section $\ref{sec_statements}$ for the precise statements of our results, and for further descriptions of the techniques involved.

We also point out the recent work of Koiso and Palmer \cite{PK1,PK2,PK3} where anisotropic capillarity problems have been studied for smooth surface tensions, without a gravitational term (and using different techniques).  In particular, we observe that the work \cite{PK3} makes use of symmetrization techniques (in particular, the notion of a Wulff shape having product form is analogous to the notion of symmetrizable functions $f$ given in Definition $\ref{def_symm}$ below).

\subsection*{Outline of the paper}
We now describe the outline of the remainder of the paper.  In Section $2$ below, we recall some background material and establish our notation, while in Section $3$, we give the precise statement of our results.  Sections $4$ through $7$ are then devoted to the proofs of these results.  In particular, in Section $4$ we develop a suitable form of the anisotropic symmetrization inequality (including a careful examination of the case of equality).  Continuing in Section $5$, this symmetrization result is used to establish the existence of minimizers for the variational problem, while in Section $6$ we use a delicate construction involving fine properties of sets of finite perimeter to establish regularity and convexity properties for symmetric minimizers.  Section $7$ is then devoted to the study of general minimizers, using the equality case of the anisotropic symmetrization to show that all minimizers are symmetric, and establishing the uniqueness of minimizers via ODE techniques.

\section{Preliminaries and Notation}

Throughout our discussion we let $N\geq 2$ be fixed, and for each $x\in\mathbb{R}^N$, we write
\begin{align*}
x=(x',x_N)\in\mathbb{R}^{N-1}\times \mathbb{R}.
\end{align*}
We also use the projections $\pi_1:\mathbb{R}^N\rightarrow\mathbb{R}^{N-1}$ and $\pi_N:\mathbb{R}^N\rightarrow\mathbb{R}$ given by $\pi_1(x)=x'$ and $\pi_N(x)=x_N$.

For each $k\geq 1$, we let $\mathcal{L}^k$ denote the $k$-dimensional Lebesgue measure; for $E\subset\mathbb{R}^k$, we will often write $|E|=\mathcal{L}^k(E)$.
 Likewise, for each $d\geq 1$, we will let $\mathcal{H}^d$ denote the $d$-dimensional Hausdorff measure.  Given a measure $\mu$, we let $|\mu|$ denote its total variation and $\supp \mu$ denote its support.  Finally, given a set $E\subset\mathbb{R}^N$ and $t\in\mathbb{R}$, we shall use the notation 
\begin{align}
E_t:=\{x\in\mathbb{R}^{N-1}:(x,t)\in E\}.\label{defslice}
\end{align}

Let $\Omega\subset\mathbb{R}^N$ be a given open set.  Recall that the space of functions of bounded variation is defined by
\begin{align*}
BV(\Omega):=\left\{f\in L^1(\Omega): \sup\left\{\int_{\Omega} f\ebdiv \phi dx:\phi\in C_c^1(\mathbb{R}^N;\mathbb{R}^N),|\phi|\leq 1\right\}<\infty\right\},
\end{align*}
and note that $BV(\Omega)$ is the set of functions $f$ for which the distributional derivative $Df$ is a Radon measure with $|Df|(\Omega)<\infty$.

Let $\mathcal{F}(\Omega)$ denote the collection of sets of finite perimeter in $\Omega$, i.e. sets $E\subset\Omega$ such that $\chi_E\in BV(\Omega)$.  For any set $E\in\mathcal{F}(\mathbb{R}^N)$, we let $\partial^* E$ denote the {\it reduced boundary} of $E$, consisting of the points $x\in \supp |D\chi_E|$ such that the limit
\begin{align*}
\nu_E(x):=-\lim_{\rho\downarrow 0} \frac{D\chi_E(B(x,\rho))}{|D\chi_E|(B(x,\rho))}
\end{align*}
exists in $\mathbb{R}^N$ and satisfies $|\nu_E(x)|=1$.  We shall also use the notation $\mathcal{F}_m$, $m>0$, to denote the collection 
\begin{align}
\mathcal{F}_m:=\{E\in\mathcal{F}(\{x:x_N>0\}):|E|=m\}\label{definition-Fm}
\end{align}
of admissible competitors for minimization of the functional $\mathscr{F}$ with respect to the volume constraint $|E|=m$.

In general, sets of finite perimeter may be quite degenerate.  Nevertheless, we recall the following results, which allow one to slice an arbitrary set $E$ having finite perimeter into sets $E_t$ as in ($\ref{defslice}$).  These results will be an essential tool for our analysis; they originate with the work of Vol'pert \cite{V} (for proofs of these results, we refer the reader to \cite{AFP,BCF,CCF}).
\begin{lem}
\label{lem_goodset}
Let $N\geq 1$ and suppose that $E\subset\mathbb{R}^N$ is a set of finite perimeter.  Then there exists a subset $G(E)$ of $\pi_N(E)\subset\mathbb{R}$ having full measure such that for every $t\in G(E)$,
\begin{enumerate}
\item[(i)] $E_t$ has finite perimeter,
\item[(ii)] $\mathcal{H}^{N-2}(\partial^*(E_t)\Delta (\partial^*E)_t)=0$, and
\item[(iii)] For $\mathcal{H}^{N-2}$-a.e. $x'$ with $(x',t)\in (\partial^*E)_t\cap \partial^*(E)_t$, we have $\pi_1(\nu_E(x',t))\neq 0$.
\end{enumerate}
\end{lem}

A fundamental tool throughout our analysis is the following case of the coarea formula, which will allow us to compute the surface energy by slices.
\begin{lem}
\label{lem_coarea}
Let $N\geq 1$ and suppose that $E\subset\mathbb{R}^N$ is a set of finite perimeter.  Then for every Borel function $g:\mathbb{R}^N\rightarrow [0,+\infty]$,
\begin{align*}
\int_{\partial^*E} g(x)|\pi_1(\nu_E(x))|d\mathcal{H}^{N-1}(x)=\int_{G(E)} \int_{\partial^*E_{t}} g(x',t)d\mathcal{H}^{N-2}(x')dt
\end{align*}
\end{lem}

As a consequence of Lemma $\ref{lem_coarea}$ and basic properties of sets of finite perimeter, one obtains that for any set $A$ of finite perimeter, the map $t\mapsto |A_t|$ belongs to BV.  In particular, we have the following lemma:
\begin{lem}
\label{lem_slice}
Fix $N\geq 1$.  Then for every $A\in\mathcal{F}(\{x\in\mathbb{R}^N:x_N>0\})$ there exists $v_A\in BV([0,\infty);[0,\infty])$ such that 
\begin{enumerate}
\item[(i)] $v_A(t)=|A_t|$ for almost every $t\in [0,\infty)$, and 
\item[(ii)] $v_A$ is differentiable on the set $G(A)$ given by Lemma $\ref{lem_goodset}$, with
\begin{align*}
v_A'(t)=-\int_{(\partial^*A)_{t}} \frac{\pi_N(\nu_A(x',t))}{|\pi_1(\nu_A(x',t))|}d\mathcal{H}^{N-2}(x'), \quad t\in G(A).
\end{align*}
\end{enumerate}
\end{lem}

We use this opportunity to establish some auxiliary notation which will be useful in the sequel.  In particular, for each $A\in\mathcal{F}(\{x\in\mathbb{R}^N:x_N>0\})$ we let $v_A^-$ and $v_A^+$ be the left and right continuous representatives of the function $v_A$ given by Lemma $\ref{lem_slice}$; in particular, we note the identities
\begin{align*}
v_A^-(t)=\lim_{\epsilon\downarrow 0} \overline{v}_A(t-\epsilon),\quad v_A^+(t)=\lim_{\epsilon\downarrow 0} \overline{v}_A(t+\epsilon)
\end{align*}
for all good representatives $\overline{v}$ of $v$ and every $t\in [0,\infty)$ (see \cite[pg. 136]{AFP} for the notion of good representative of a $BV$ function in this context).  Moreover, we will also make use of the quantities 
\begin{align}
r_A^\pm(t):=(v_A^{\pm}(t)/|K_h|)^\frac{1}{N-1}.\label{eqrdef}
\end{align}

Equipped with this notation, we remark that the convexity of $K$ implies that $v_K$ is continuous with $v_K(t)=v_K^-(t)=v_K^+(t)$.  To aid the clarity of our exposition, we will use the notation $r_A(t):=r_A^+(t)$ when the continuity properties of the representative are not relevant.

\section{Statement of main results}
\label{sec_statements}

We now give the precise statements of our results.  As we described in the introduction, our goal is to apply symmetrization techniques to the study of minimizers of $\mathscr{F}$.  Since the Wulff shape may be any open convex set in general, we must restrict ourselves to an appropriate class of surface tension functionals which possess a suitable notion of symmetry.  In particular, we will restrict our study to {\it symmetrizable} sets in the following sense:
\begin{defin}[Symmetrizability]
\label{def_symm}
Let $f:\mathbb{R}^N\rightarrow \mathbb{R}^+$ be convex and positively $1$-homogeneous, with $f(x)>0$ for $|x|>0$.  
We say that $f$ is {\it symmetrizable} if there exist lower-semicontinuous functions $h:\mathbb{R}^{N-1}\rightarrow [0,\infty)$ and $\phi:[0,\infty)\times\mathbb{R}\rightarrow [0,\infty)$ such that 
\begin{itemize}
\item $h$ is positively $1$-homogeneous, convex, and satisfies $h(x')>0$ for $|x'|>0$,
\item $\phi$ is convex,
and
\item the identity $f(x)=\phi(h(x'),x_N)$ holds for every $x=(x',x_N)\in\mathbb{R}^{N-1}\times\mathbb{R}=\mathbb{R}^{N}$.
\end{itemize}
\end{defin}
\begin{rem}
Suppose that $f$ is symmetrizable with $f(x)=\phi(h(x'),x_N)$.  Then $\phi$ is positively $1$-homogeneous on $\Rng(h)\times\mathbb{R}$.  Indeed, for every $(s,t)\in\Rng(h)\times\mathbb{R}$ and $\lambda>0$ we choose $\xi\in\mathbb{R}^{N-1}$ such that $h(\xi)=s$ and therefore obtain
\begin{align*}
\phi(\lambda s,\lambda t)&=\phi(\lambda h(\xi),\lambda t)=\phi(h(\lambda \xi),\lambda t)=f\big(\lambda(\xi,t)\big)\\
&=\lambda f(\xi,t)=\lambda \phi(h(\xi),t)=\lambda \phi(s,t).
\end{align*}
\end{rem}
\begin{exa} If $f(x)=|x|_p$ with $p>1$, then we may take $\phi(a,b)=|(a,b)|_p=(|a|^p+|b|^p)^\frac{1}{p}$ and $h(x)=|x|_p$, so that $\phi$ is convex and lower-semicontinuous.  Moreover, for every $x\in\mathbb{R}^N$, we have
\begin{align*}
\phi(f(\pi_1(x),0),\pi_N(x))&=(|f(\pi_1(x),0)|^p+|\pi_N(x)|^p)^\frac{1}{p}\\
&=(\sum_{i=1}^{N-1} |x_i|^p+|x_N|^p)^\frac{1}{p}=|x|_p=f(x).
\end{align*}
\end{exa}

The notion of symmetrizability established in Definition $\ref{def_symm}$ corresponds to asking that the Wulff shape $K_f$ is axially symmetric with respect to an open convex set $K_h\subset \mathbb{R}^{N-1}$.  Indeed, this is the content of our next lemma, where we show that the set $K_h$ is exactly the Wulff shape corresponding to the function $h$.
\begin{lem}
\label{lemsymmchar}
Let $f$ be symmetrizable with $f(x)=\phi(h(x'),x_N)$, and let $K\subset \mathbb{R}^N$ be the Wulff shape associated to $f$.  Then there exists $\alpha:\mathbb{R}\rightarrow [0,\infty)$ concave on $\{t:\alpha(t)>0\}$ such that
\begin{enumerate}
\item[(i)] $K_t=\emptyset$ for every $t\in\mathbb{R}$ with $\alpha(t)=0$, and
\item[(ii)] $K_t=\alpha(t)K_h$ for every $t\in\mathbb{R}$ with $\alpha(t)\neq 0$,
\end{enumerate}
where $K_t$ is defined as in ($\ref{defslice}$) and where $K_h$ denotes the Wulff shape in $\mathbb{R}^{N-1}$ corresponding to the function $h$.

Conversely, given an arbitrary lower-semicontinuous, positively $1$-homogeneous, convex function $f$ satisfying $f(x)>0$ for $|x|>0$, if there exists an open convex set $K_h\subset\mathbb{R}^{N-1}$ such that the Wulff shape $K$ for $f$ satisfies $(i)$ and $(ii)$, then $f$ is symmetrizable.
\end{lem}

\begin{proof}
We begin with the second statement.  Suppose that $f$ and $K_h$ are given such that $(i)$ and $(ii)$ hold for some $\alpha:\mathbb{R}\rightarrow [0,\infty)$ as above, and let $h:\mathbb{R}^{N-1}\rightarrow [0,\infty)$ be the unique positively $1$-homogeneous convex function with Wulff shape $K_h$ (obtained via the characterization $h(\nu)=\sup \{x\cdot \nu:x\in K_h\}$ for $\nu\in \mathbb{R}^{N-1}$).  Then for every $x=(x',x_N)\in\mathbb{R}^{N-1}\times\mathbb{R}$, we have
\begin{align*}
f(x)&=\sup \bigg\{\alpha(y_N)h(x')+x_Ny_N:y_N\in\pi_N(K)\bigg\},
\end{align*}
Defining $\phi(s,t)=\sup_{y_N\in\pi_N(K)} [s\alpha(y_N)+ty_N]$ for each $s,t\in\mathbb{R}$, we therefore obtain $f(x)=\phi(h(x'),x_N)$ for every $x=(x',x_N)\in\mathbb{R}^{N-1}\times\mathbb{R}$.

Conversely, suppose that $f$ is symmetrizable.  For each $t\in\mathbb{R}$, define
\begin{align*}
\alpha(t)=\inf_{y_N\in\pi_N(K)} \max\{\phi(1,y_N)-ty_N,0\}.
\end{align*}
Straightforward calculations now give the desired concavity for $\alpha$ along with the conditions $(i)$ and $(ii)$.
\end{proof}

We also introduce some further technical restrictions on $f$ which appear in our arguments.  In particular, we define the following notion of admissibility:
\begin{defin}[Admissibility]
\label{def_adm}
Let $f:\mathbb{R}^N\rightarrow \mathbb{R}^+$ be convex and positively $1$-homogeneous, with $f(x)>0$ for $|x|>0$.  
We say $f$ is {\it admissible} if $f$ is symmetrizable in the sense of Definition $\ref{def_symm}$ and, writing
\begin{align*}
f=\phi(h(x'),x_N),
\end{align*}
the function $\phi$ is strictly convex, and $C^1$ in a neighborhood of $(0,\pm 1)$, with $\partial_1 \phi(0,\pm 1)=0$.
\end{defin}

\begin{figure}
\centering
\def\svgwidth{0.72\columnwidth}

\begingroup
  \makeatletter
  \providecommand\color[2][]{%
    \errmessage{(Inkscape) Color is used for the text in Inkscape, but the package 'color.sty' is not loaded}
    \renewcommand\color[2][]{}%
  }
  \providecommand\transparent[1]{%
    \errmessage{(Inkscape) Transparency is used (non-zero) for the text in Inkscape, but the package 'transparent.sty' is not loaded}
    \renewcommand\transparent[1]{}%
  }
  \providecommand\rotatebox[2]{#2}
  \ifx\svgwidth\undefined
    \setlength{\unitlength}{841.88974609pt}
  \else
    \setlength{\unitlength}{\svgwidth}
  \fi
  \global\let\svgwidth\undefined
  \makeatother
  \begin{picture}(1,0.70707072)%
    \put(0,0){\includegraphics[width=\unitlength]{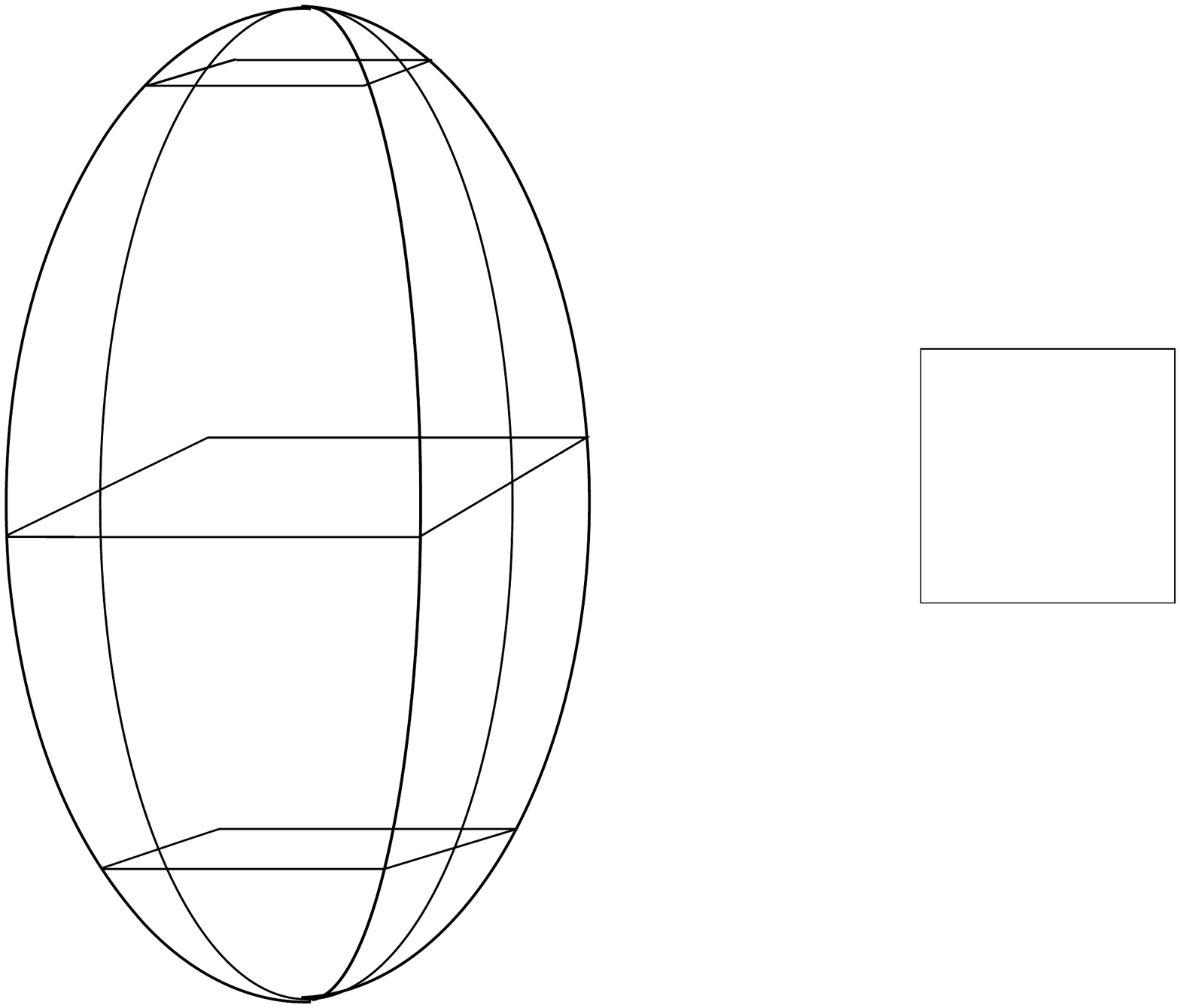}}%
    \put(0.70860993,0.22941581){\color[rgb]{0,0,0}\makebox(0,0)[lb]{\smash{$K_h$}}}%
    \put(0.38552722,0.65159534){\color[rgb]{0,0,0}\makebox(0,0)[lb]{\smash{$K_t=\alpha(t)K_h$
}}}%
  \end{picture}%
\endgroup
\caption{An example of the Wulff shape associated to an admissible function $f:\mathbb{R}^3\rightarrow\mathbb{R}^+$.  The set is axially symmetric with respect to the axis $\{(x',x_3):x'=0\}$ and the set $K_h\subset\mathbb{R}^2$.  The set $K_h$ is the Wulff shape associated to the function $h$.}
\end{figure}

The convexity and smoothness assumptions on $\phi$ in the notion of admissibility are important technical assumptions which are used in our construction of competitors for minimality.  In particular, these assumptions imply that the Wulff shape $K$ associated to $f$ is smooth and ``flat'' at the top and bottom (see also the discussion in Section $\ref{sec_construction}$ below).  Equipped with the notion of admissibility, we are now ready to state the main results of our study.  We begin with the notion of anisotropic symmetrization,
\begin{defin}[Anisotropic symmetrization]
\label{defsymm}
Let $f:\mathbb{R}^N\rightarrow \mathbb{R}^+$ be convex and positively $1$-homogeneous, with $f(x)>0$ for $|x|>0$.  Suppose that $f$ is admissible in the sense of Definition $\ref{def_adm}$, $f(x)=\phi(h(x'),x_N)$, and let $K_h\subset \mathbb{R}^{N-1}$ be the Wulff shape associated to $h$.  For each set of finite perimeter $A\subset\mathbb{R}^N$, we define the anisotropic symmetrization $A^*$ of $A\subset\mathbb{R}^N$ by
\begin{align*}
A^*&=\{(x',t):x'\in \left(v_A(t)/|K_h|\right)^\frac{1}{N-1}K_h, t\in \mathbb{R}\},
\end{align*}
where $v_A$ is as in Lemma $\ref{lem_slice}$.
\end{defin}

Note that Lemma $\ref{lem_slice}$ can be combined with the characterization of functions of bounded variation by sections (c.f. \cite[Remark 3.104]{AFP}) to show that the symmetrization $A\mapsto A^*$ preserves the property of being a set of finite perimeter. 

This symmetrization was introduced in \cite{JS}, and may be seen as the analogue of Steiner symmetrization (with codimension $N-1$) for the convex symmetrization of Alvino, Ferone, Trombetti and Lions in \cite{AVTL}.  In the context of the functional $\mathscr{F}$, our main symmetrization result is the following theorem, which establishes the relationship between the anisotropic symmetrization and the functional $\mathscr{F}$.
\begin{thm}[Symmetrization inequality for $\mathscr{F}$]
\label{thm_symm_f}
Suppose that $f:\mathbb{R}^{N}\rightarrow\mathbb{R}^+$ is admissible, and fix $\omega\in(-f(e_N),f(-e_N))$.  Then for any set of finite perimeter $A\subset \mathbb{R}^{N-1}\times [0,\infty)$, the set $A^*$ has finite perimeter and
\begin{align}
\mathscr{F}(A^*)\leq \mathscr{F}(A).\label{eqsymmgoal}
\end{align}
\end{thm}
A key step in the proof of Theorem $\ref{thm_symm_f}$ is to show that the anisotropic symmetrization leads to a decrease in the surface energy $\mathscr{F}_s$.  We remark that such a result is present in \cite{JS}, where the author proceeds by approximation of $\chi_A$ using functions in $W^{1,1}$.  Since the case of equality will be essential to our identification of minimizers, we pursue a different presentation based on an approach to the isoperimetric problem via Fubini's theorem (see for instance \cite{DeGiorgi}, as well as \cite{BCF,CCF,Fusco,Maggi}).  In particular this approach is based on computing the symmetrization by slices, and has the benefit of allowing easier access to geometric properties of minimizers.

\begin{figure}
\centering
\def\svgwidth{0.72\columnwidth}

\begingroup
  \makeatletter
  \providecommand\color[2][]{%
    \errmessage{(Inkscape) Color is used for the text in Inkscape, but the package 'color.sty' is not loaded}
    \renewcommand\color[2][]{}%
  }
  \providecommand\transparent[1]{%
    \errmessage{(Inkscape) Transparency is used (non-zero) for the text in Inkscape, but the package 'transparent.sty' is not loaded}
    \renewcommand\transparent[1]{}%
  }
  \providecommand\rotatebox[2]{#2}
  \ifx\svgwidth\undefined
    \setlength{\unitlength}{841.88974609pt}
  \else
    \setlength{\unitlength}{\svgwidth}
  \fi
  \global\let\svgwidth\undefined
  \makeatother
  \begin{picture}(1,0.70707072)%
    \put(0,0){\includegraphics[width=\unitlength]{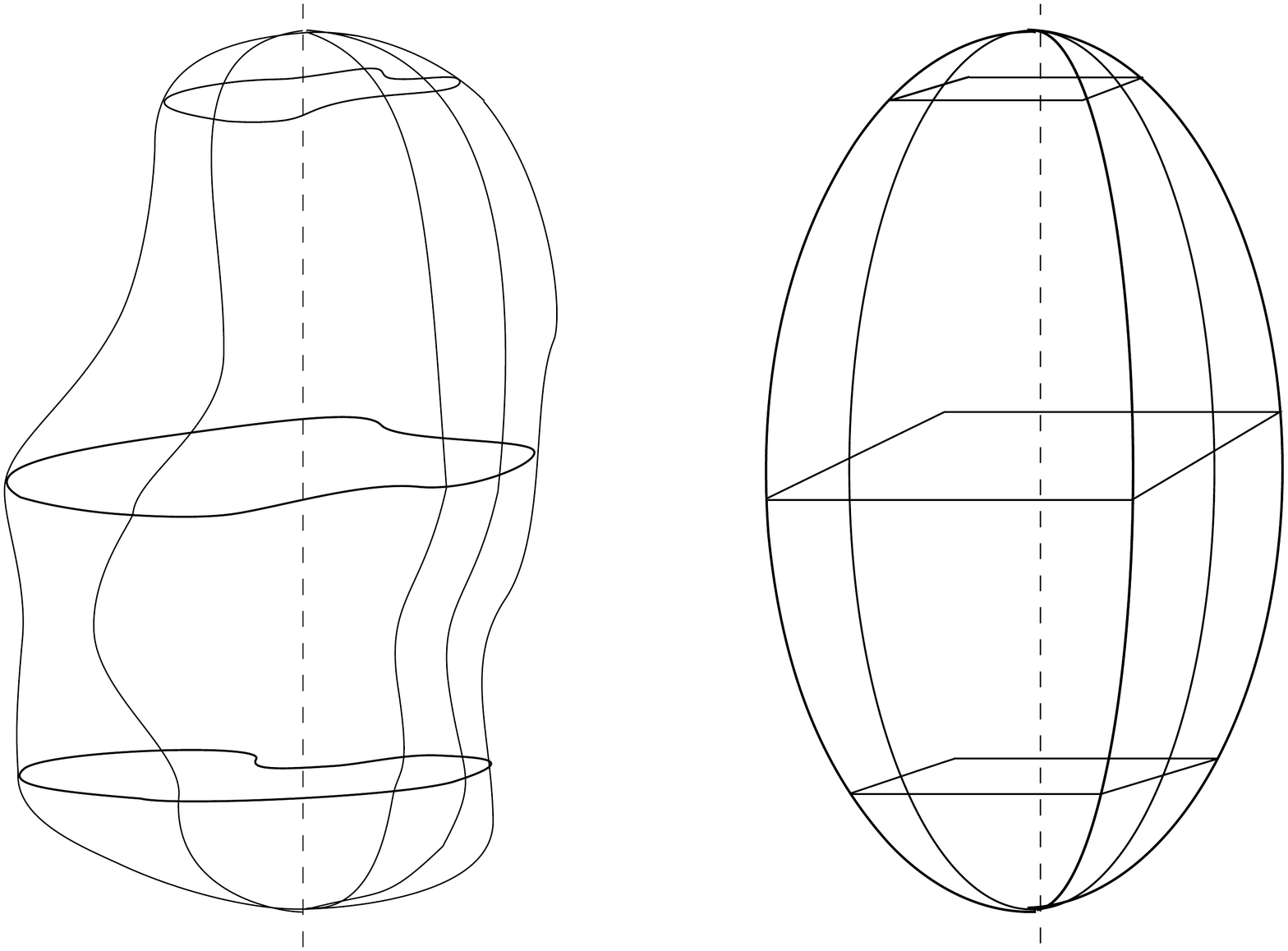}}%
    \put(0.40181713,0.05429958){\color[rgb]{0,0,0}\makebox(0,0)[lb]{\smash{$A$}}}%
    \put(0.91494842,0.05701459){\color[rgb]{0,0,0}\makebox(0,0)[lb]{\smash{$A^*$}}}%
  \end{picture}%
\endgroup
\caption{The symmetrization defined in Definition $\ref{defsymm}$.  Each slice $A_t=\{(x',t):(x',t)\in A\}$ of the original set $A$ is mapped to the rescaled Wulff shape $A_t^*=\{(x',t):x'\in \left(\frac{v_A(t)}{|K_h|}\right)^{1/(N-1)}K_h\}$ which satisfies $|A_t|=|A_t^*|$.}
\label{figsymm}
\end{figure}

\vspace{0.2in} 

With this symmetrization result in hand, we now turn to the main topic of our investigation, the study of minimizers of $\mathscr{F}$.  Our first result concerns the existence of minimizers.
\begin{thm}[Existence of minimizers]
\label{thm_exist}
Fix $m>0$, $\omega\in (-f(e_N),f(-e_N))$ and let $f$ be given admissible with $f(x)=\phi(h(x'),x_N)$.  Then there exists a minimizer $E_0\in \mathcal{F}_m$ for $\mathscr{F}$ among sets in $\mathcal{F}_m$, with $\mathcal{F}_m$ given by (\ref{definition-Fm}).
\end{thm}
The proof of Theorem $\ref{thm_exist}$ is based on the direct method of the calculus of variations and is given in Section $\ref{sec_exist}$.  In particular, choosing a minimizing sequence and invoking the symmetrization result of Theorem $\ref{thm_symm_f}$, we obtain uniform bounds which enable us to use classical compactness theorems for sets of finite perimeter.  Theorem $\ref{thm_exist}$ follows once we establish a suitable lower-semicontinuity result for the functional $\mathscr{F}$.  The arguments involved are closely related to the isotropic case \cite{Gonzalez} where symmetrization results are exploited in a similar manner, although we remark that the anisotropic setting does present some subtlety in establishing the lower-semicontinuity of the contact energy when $\omega<0$.  We refer the reader to Section $\ref{sec_exist}$ for further discussion of this point.

\vspace{0.2in}

The next two results concern regularity properties of minimizers.  In particular, we begin by establishing the convexity of symmetric minimizers; here, and in the remainder of this paper, we use the term symmetric in the context of the symmetrization $E\mapsto E^*$ given by Definition $\ref{defsymm}$.
\begin{thm}[Convexity of symmetric minimizers]
\label{thm_concave}
Fix $m>0$ along with $\omega\in (-f(e_N),f(-e_N))$, and let $f$ be given admissible with $f(x)=\phi(h(x'),x_N)$.  Let $E\in \mathcal{F}_m$ be given such that $E=E^*$.  If $E$ is a minimizer for $\mathscr{F}$, then $\{t:v_E(t)>0\}$ is an interval of the form $(0,T_{\max})$ and $r_E$ is concave on $[0,T_{\max}]$.
\end{thm}
The proof of Theorem $\ref{thm_concave}$ is based on a competitor argument, in which a family of rescaled copies of the Wulff shape for $f$ is used to construct a local competitor at hypothetical points of non-convexity.  By this construction, we show that the functional $\mathscr{F}$ takes on a smaller value for the competing set; we therefore obtain that such points of non-convexity cannot exist.  The argument may be seen as a rather delicate localization of the arguments which establish convexity in the isotropic case \cite{GonzalezTamanini}.  As we remarked above, the proof in \cite{GonzalezTamanini} makes strong use of the regularity theory of minimal surfaces to obtain analyticity of minimizers, which is not available in our setting.  To compensate, our arguments proceed with a more localized construction which carefully exploits fine properties of sets of finite perimeter in place of these a priori regularity results.  We refer to Section $6$ for a more detailed description.

The next step in our analysis is to extend the convexity of Theorem $\ref{thm_concave}$ to arbitrary - that is, not necessarily symmetric - minimizers for $\mathscr{F}$.  In particular, using the equality condition of the symmetrization result Theorem $\ref{thm_symm_f}$, we show that arbitrary minimizers are symmetric and therefore convex (by Theorem $\ref{thm_concave}$).  Our main result takes the form
\begin{thm}[Symmetry and convexity of minimizers]
\label{thm_min}
Fix $m>0$ together with $\omega\in (-f(e_N),f(-e_N))$ and let $f$ be given admissible with $f(x)=\phi(h(x'),x_N)$.  Suppose that $E\in\mathcal{F}_m$ is a minimizer for $\mathscr{F}$.  Then there exists a function $\alpha:[0,\infty)\rightarrow \mathbb{R}^+$ and a constant $\beta_0\in\mathbb{R}^{N-1}$ such that 
\begin{align*}
E_t=\beta_0+\alpha(t)K_h
\end{align*}
for a.e. $t\in [0,T_{\max}]$.  Moreover, $\alpha$ is concave on $[0,T_{\max}]$.
\end{thm}
The proof of Theorem $\ref{thm_min}$ is based on the analysis of symmetric minimizers of Section $4$ combined with a study of the regularity properties of the barycenter of each slice of the minimizer.  These arguments originate in the recent work of Barchiesi, Cagnetti and Fusco \cite{BCF}, where the authors study the equality cases and stability of the classical Steiner symmetrization with arbitrary codimension.

\vspace{0.2in}

We conclude our study by giving an ODE characterization of minimizers along with an associated boundary condition (analogous to the classical Young's law in the isotropic case) in the case that $f$ is sufficiently smooth:
\begin{thm} [ODE characterization of minimizers and anisotropic Young's law]
\label{thm_young}
Suppose that $\phi$ is $C^2$.  Then any minimizer $E$ satisfies the Euler-Lagrange equation
\begin{align}
\nonumber -\frac{d}{dt}\left[(N-1)r_E^{N-2}\partial_2\phi(\Lambda,-(N-1)r'_E)\right]&=(N-2)r_E^{N-3}\phi(\Lambda,-(N-1)r'_E)\\
&\hspace{0.2in}+(N-1)(t+\lambda)r_E^{N-2}\label{el1}
\end{align}
along with the boundary condition
\begin{align}
\label{youngs}-\partial_2\phi(\Lambda,-(N-1)r_E'(0))=\omega,
\end{align}
where
\begin{align*}
\Lambda=\frac{P_h(K_h)}{|K_h|}\quad \textrm{and}\quad P_h(K_h)=\int_{K_h} h(\nu_{K_h}(x'))dx'.
\end{align*}
\end{thm}

We then use this ODE characterization to show the uniqueness of minimizers, basing our analysis on an adaptation of the argument given by Finn for the uniqueness of the classical isotropic sessile drop \cite{Finn1}.  
\begin{thm}[Uniqueness of minimizers]
Suppose that $\phi$ and $h$ are smooth, and let $\omega\in (-f(e_N),f(-e_N))$ be given.  Then for each $m>0$ there exists a unique minimizer $E\in\mathcal{F}_m$.
\label{thm_unique}
\end{thm}

\section{Anisotropic symmetrization.}

This section is devoted to the symmetrization arguments which form the basis for our study.  In particular, our goal is to develop the machinery necessary to establish Theorem $\ref{thm_symm_f}$, with particular attention paid to the case of equality.  Our proof of this theorem is based on an analysis of the slices $A_t$ defined in ($\ref{defslice}$), making use of Lemmas $\ref{lem_goodset}$, $\ref{lem_slice}$ and $\ref{lem_coarea}$.  To facilitate our arguments, we consider the dual function
\begin{align*}
h_*(x):=\sup\{x\cdot y:h(y)=1\},
\end{align*}
which gives the convenient characterization $K_h=\{x:h_*(x)<1\}$ of the Wulff shape for $h$.  We observe that our hypotheses on $h$ imply that $h_*$ is positively $1$-homogeneous and convex; in particular, it is differentiable almost everywhere.  Moreover, we have the following remark, which was observed by Van Schaftingen in \cite{JS} as a key component of the analysis.  We include the argument here for the convenience of the reader.

\begin{rem}(\cite[Theorem $6.2$]{JS})  
\label{rem_hstar}
The identity $h(\nabla h_*(x'))=1$ holds at every point of differentiability for $h_*$.  Indeed, suppose that $h_*$ is differentiable at $x\in \mathbb{R}^{N-1}$.  Then the theory of subdifferentials for convex functions gives 
$\partial h_*(x)=\{\nabla h_*(x)\}$,
where
\begin{align*}
\partial h_*(x):=\{z\in\mathbb{R}^{N-1}:\forall\,y\in\mathbb{R}^{N-1},\quad h_*(y)\geq h_*(x)+z\cdot (y-x)\}.
\end{align*}

On the other hand, writing $h_*(x)=\sup \{\frac{x\cdot y}{h(y)}:|y|=1\}$, we note that the lower-semicontinuity of $h$ and compactness of the unit sphere allow us to choose $y_0\in\mathbb{R}^{N-1}$ such that $|y_0|=1$ and $h_*(x)=x\cdot y_0/h(y_0)$.  Setting $\tilde{y}_0=y_0/h(y_0)$, we then have
$h(\tilde{y}_0)=1$ and $h_*(\tilde{y}_0)=x\cdot y_0$.
This in turn implies 
\begin{align*}
h_*(y)\geq y\cdot \tilde{y}_0=x\cdot \tilde{y}_0+(y-x)\cdot \tilde{y}_0=h_*(\tilde{y}_0)+(y-x)\cdot \tilde{y}_0 \quad \forall\,y\in\mathbb{R}^{N-1},
\end{align*}
so that $\tilde{y}_0$ belongs to the set $\partial h_*(x)$.  Thus, $\nabla h_*(x)=\tilde{y}_0$, which gives $h(\nabla h_*(y))=h(\tilde{y}_0)=1$ as desired.
\end{rem}

With this remark in hand, we now compute the normal to a symmetric set.
\begin{lem}
\label{lem_41}
Suppose $A\in \mathcal{F}_m$ satisfies $A_t=r_A(t)K_h$ for a.e. $t\in [0,\infty)$ with $r_A(t)$ as in $(\ref{eqrdef})$, and define
\begin{align*}
F:=\{(s,t):s<r_A(t)\}.
\end{align*}

Then, for a.e. $t\in [0,\infty)$ and $\mathcal{H}^{N-2}$-a.e. $x'\in \partial^*A_t\cap \{x':\pi_1(\nu_A(x',t))\neq 0\}$, we have
$\pi_1(\nu_F(h_*(x'),t))\nabla h_*(x')\neq 0$, and
\begin{align}
\frac{\nu_A(x',t)}{|\pi_1(\nu_A(x',t)|}&=\frac{1}{|\pi_1(\nu_F(h_*(x'),t))\nabla h_*(x')|}\left(\begin{array}{c}\pi_1(\nu_F(h_*(x'),t))\nabla h_*(x')\\\pi_2(\nu_F(h_*(x'),t))\end{array}\right).\label{eqab1}
\end{align}

Moreover, for a.e. $t\in [0,\infty)$ and $\mathcal{H}^{N-2}$-a.e. $x'\in \partial^*A_t\cap \{x':\pi_1(\nu_A(x',t))=0\}$, we have
\begin{align}
\nu_F(h_*(x'),t)&=\left(\begin{array}{c}0\\\pi_N(\nu_A(x',t))\end{array}\right).\label{eqab2a}
\end{align}
\end{lem}

\begin{proof}
Let $(\rho_n)\subset C_0^\infty(\mathbb{R}^2)$ be a sequence of mollifiers and set $f_n=\rho_n*\chi_F$ for each $n\in\mathbb{N}$.  Moreover, define $a_n:\mathbb{R}^N\rightarrow\mathbb{R}$ by
\begin{align*}
a_n(x)=f_n(h_*(x'),t).
\end{align*}

Fix $\phi\in C_0^1(\mathbb{R}^N;\mathbb{R}^N)$.  By dominated convergence, we have
\begin{align}
\lim_{n\rightarrow\infty} \int_{\mathbb{R}^N} a_n(x)\ebdiv_x \phi(x)dx&=\int_{\mathbb{R}^N} \chi_A(x)\ebdiv_x \phi(x)dx.\label{eqab0}
\end{align}
This quantity is then equal to 
\begin{align}
\nonumber &\int_{\partial^*A} \phi(x)\cdot \nu_A(x)d\mathcal{H}^{N-1}(x)\\
\nonumber &\hspace{0.2in}=\int_0^\infty \int_{\partial^*A_t} \phi(x',t)\cdot \frac{\nu_A(x',t)}{|\pi_1(\nu_A(x',t))|}d\mathcal{H}^{N-2}(x')dt\\
&\hspace{0.4in}+\int_{\partial^*A\cap \{\pi_1(\nu_A(x))=0\}} \pi_N(\phi(x))\pi_N(\nu_A(x))d\mathcal{H}^{N-1}(x),\label{eqab0a}
\end{align}
where we have used the Gauss-Green formula for sets of finite perimeter and Lemma $\ref{lem_coarea}$.
 
On the other hand, for every $n\in \mathbb{N}$, integration by parts and the chain rule give
\begin{align}
\nonumber &\int_{\mathbb{R}^N} a_n(x)\ebdiv_x \phi(x)dx=-\int_{\mathbb{R}^N} \nabla_x [a_n(x)]\cdot \phi(x)dx\\
\nonumber &\hspace{0.2in}=-\int_{\mathbb{R}}\int_{\mathbb{R}^{N-1}} (\partial_1 f_n)(h_*(x'),t)(\nabla h_*)(x')\cdot \pi_1(\phi(x',t))dx'dt\\
&\hspace{1.4in}-\int_{\mathbb{R}}\int_{\mathbb{R}^N}(\partial_2 f_n)(h_*(x'),t)\pi_N(\phi(x',t))dx'dt.\label{eqa2}
\end{align}
Setting 
\begin{align*}
b(y,t):=\left(\begin{array}{cc} \nabla h_*(y)\cdot \pi_1(\phi(y,t))\\ \pi_N(\phi(y,t))\end{array}\right),
\end{align*}
and invoking the co-area formula, the right hand side of ($\ref{eqa2}$) becomes
\begin{align}
\label{eqa3}-\int_{\mathbb{R}}\int_{0}^\infty\int_{h_*^{-1}(\{s\})} \frac{1}{|\nabla h_*(y)|}b(y,t)\cdot \nabla f_n(s,t)d\mathcal{H}^{N-2}(y)dsdt.
\end{align}

We now let $n\rightarrow\infty$.  More precisely, noting that 
\begin{align*}
(s,t)\mapsto \int_{h_*^{-1}(\{s\})} b(y,t)d\mathcal{H}^{N-2}(y)
\end{align*}
is continuous as a map from $\mathbb{R}^2$ to $\mathbb{R}^2$, we use the weak-* convergence of the measure $\nabla f_n(s,t)\mathcal{L}^2(s,t)$ to $D\chi_F(s,t)$ to obtain that ($\ref{eqa3}$) tends to 
\begin{align}
-\int_{\mathbb{R}\times [0,\infty)}\int_{h_*^{-1}(\{s\})} \frac{b(y,t)\cdot \nu_F(h_*(y),t)\chi_{\partial^*F}(h_*(y),t)}{|\nabla h_*(y)|}d\mathcal{H}^{N-2}(y)d\mathcal{H}^{1}(s,t)\label{eqa4}
\end{align}
as $n\rightarrow \infty$.  Then, applying Lemma $\ref{lem_coarea}$ and using the coarea formula, we obtain 
\begin{align}
\lim_{n\rightarrow\infty} \int_{\mathbb{R}^N} a_n(x)\ebdiv_x\phi(x)dx&=(I)+(II),\label{eqab0b}
\end{align}
where
\begin{align*}
(I)&:=-\int_{\mathbb{R}}\int_{\mathbb{R}^{N-1}} \frac{b(x',t)\cdot \nu_F(h_*(x'),t)\chi_{\partial^*F}(h_*(x'),t)}{|\pi_1(\nu_F(h_*(x'),t))\nabla h_*(y)|} d\mathcal{H}^{N-2}(x')dt\\
&=-\int_{\mathbb{R}}\int_{\partial^*A_t} \frac{b(x',t)\cdot \nu_F(h_*(x'),t)}{|\pi_1(\nu_F(h_*(x'),t))\nabla h_*(x')|} d\mathcal{H}^{N-2}(x')dt
\end{align*}
and
\begin{align*}
(II)&:= -\int_{\{(x',t):\pi_1(\nu_F(h_*(x'),t))=0\}} \pi_N(\phi(x',t))\pi_N(\nu_F(h_*(x'),t))\\
&\hspace{2.7in}\chi_{\partial^*F}(h_*(x'),t)d\mathcal{H}^{N-1}(x',t)\\
&= -\int_{\partial^*A\cap \{\pi_1(\nu_A(x',t))=0\}} \pi_N(\phi(x',t))\pi_N(\nu_F(h_*(x'),t))d\mathcal{H}^{N-1}(x',t).
\end{align*}
Combining ($\ref{eqab0}$), ($\ref{eqab0a}$), and ($\ref{eqab0b}$) and recalling that $\phi$ was arbitrary, we obtain ($\ref{eqab1}$) and the equality of the second components in ($\ref{eqab2a}$). 
To obtain the equality for the first components of ($\ref{eqab2a}$), we note that $\pi_1(\nu_A(x',t))=0$ and $|\nu_A(x',t)|=1$ imply $|\pi_N(\nu_F(h_*(x'),t)|=|\pi_N(\nu_A(x',t))|=1$; the result then follows by recalling that $\nu_F(h_*(x'),t)$ is a unit vector.
\end{proof}

In the next lemma, we show that for almost every slice $A_t$ of $A$ the integral over $\partial^*A_t$ appearing in the coarea formula of Lemma $\ref{lem_coarea}$ is reduced by the symmetrization.  The proof is based on an application of the anisotropic isoperimetric (Wulff) inequality, and the following form of Jensen's inequality, which we recall for the convenience of the reader.
\begin{proposition}[Jensen's inequality] 
\label{prop_jensen}
Let $(X,\mu,\mathcal{M})$ be a measure space with $\mu(X)<\infty$, and suppose that $\phi:\mathbb{R}\times \mathbb{R}\rightarrow\mathbb{R}$ is convex and positively $1$-homogeneous.  Then for every $\mu$-measurable $f,g:X\rightarrow [0,\infty)$, we have
\begin{align*}
\phi\left(\int_X f(x)d\mu(x),\int_X g(x)d\mu(x)\right)\leq \int_X \phi\big(f(x),g(x)\big)d\mu(x)
\end{align*}

Moreover, if $\phi$ is strictly convex in either argument then equality holds if and only if $f/g$ is constant $\mu$-a.e. on $X$.
\end{proposition}
\begin{lem}
\label{lem_symm}
Fix $m>0$, $\omega\in (-f(e_N),f(-e_N))$ and let $f$ be given admissible with $f(x)=\phi(h(x'),x_N)$.  Let $A\subset\mathbb{R}^N$ be a given set of finite perimeter with $|A|<\infty$.  Moreover, let $G(A)$, $G(A^*)$ be as in Lemma \ref{lem_goodset}, and set $G=G(A)\cap G(A^*)$.  Then, for every $t\in G$,
\begin{align}
\label{eqclaim1}
\int_{(\partial^*(A^*))_{t}} \frac{f(\nu_{A^*}(x',t))}{|\pi_1(\nu_{A^*}(x',t))|}d\mathcal{H}^{N-2}(x')&\leq \int_{(\partial^*A)_{t}} \frac{f(\nu_A(x',t))}{|\pi_1(\nu_A(x',t))|}d\mathcal{H}^{N-2}(x').
\end{align}
Moreover, if equality holds for some $t\in G$, then $A_{t}$ is equal to $K_h$ up to translation and dilation.
\end{lem}

\begin{proof}
Let $t\in G$ be given.  Note that $A^*$ satisfies the hypotheses of Lemma $\ref{lem_41}$ with $r(t)=(|A_{t}|/|K_h|)^\frac{1}{N-1}$.
To simplify our notation, we set $\nu_1=\pi_1(\nu_F(r(t),t)$ and $\nu_2=\pi_2(\nu_F(r(t),t))$, with $F$ as in Lemma $\ref{lem_41}$.  
Now, using Lemma $\ref{lem_41}$ and Remark $\ref{rem_hstar}$
\begin{align}
\nonumber &\int_{\partial^*A^*_{t}} \frac{f(\nu_{A^*}(x',t))}{|\pi_1(\nu_{A^*}(x',t))|}d\mathcal{H}^{N-2}(x')\\
\nonumber &\hspace{0.4in}=\int_{\partial^*A^*_{t}} \phi\left(h\left(\frac{\pi_1(\nu_{A^*}(x',t))}{|\pi_1(\nu_{A^*}(x',t)|}\right),\frac{\pi_N(\nu_{A^*}(x',t))}{|\pi_1(\nu_{A^*}(x',t))|}\right)d\mathcal{H}^{N-2}(x')\\
\nonumber &\hspace{0.4in}=\int_{\partial^*A^*_{t}} \phi\left(\frac{\nu_1h(\nabla h_*(x'))}{|\nu_1\nabla h_*(x')|},\frac{\nu_2}{|\nu_1\nabla h_*(x')|}\right)d\mathcal{H}^{N-2}(x')\\
\nonumber &\hspace{0.4in}=\phi(\nu_1,\nu_2)\int_{\partial^*A^*_{t}} \frac{1}{|\nu_1\nabla h_*(x')|}d\mathcal{H}^{N-2}(x')\\
\nonumber &\hspace{0.4in}=\phi\bigg(\int_{\partial^*A^*_{t}} \frac{\nu_1}{|\nu_1\nabla h_*(x')|}d\mathcal{H}^{N-2}(x'),\int_{(\partial^*(A^*))_{t}}\frac{\nu_2}{|\nu_1\nabla h_*(x')|}d\mathcal{H}^{N-2}(x')\bigg)\\
&\hspace{0.4in}=\phi\bigg(\int_{\partial^*A^*_{t}} h\bigg(\frac{\pi_1(\nu_{A^*}(x',t))}{|\pi_1(\nu_{A^*}(x',t))|}\bigg)d\mathcal{H}^{N-2}(x'),-v_A'(t)\bigg),\label{eq3p}
\end{align}
where $v_A'$ is as in Lemma $\ref{lem_slice}$.  We remark that to obtain the last equality we have used Remark $\ref{rem_hstar}$ once again.

Applying the anisotropic isoperimetric inequality along with the monotonicity of $x\mapsto \phi(x,y)$, we bound the right hand side of ($\ref{eq3p}$) by 
\begin{align}
\nonumber &\phi\bigg(\int_{\partial^*A_{t}} h\bigg(\frac{\pi_1(\nu_A(x',t))}{|\pi_1(\nu_A(x',t))|}\bigg)d\mathcal{H}^{N-2}(x'),-v_A'(t)\bigg)\\
\nonumber &\hspace{0.4in}=\phi\bigg(\int_{\partial^*A_{t}} h\bigg(\frac{\pi_1(\nu_A(x',t))}{|\pi_1(\nu_A(x',t))|}\bigg)d\mathcal{H}^{N-2}(x'),\\
\nonumber &\hspace{1.4in}\int_{\partial^*A_{t}} \frac{\pi_N(\nu_A(x',t))}{|\pi_1(\nu_A(x',t))|}d\mathcal{H}^{N-2}(x')\bigg)\\
\nonumber &\hspace{0.4in}\leq \int_{\partial^*A_{t}} \phi\bigg(h\bigg(\frac{\pi_1(\nu_A(x',t))}{|\pi_1(\nu_A(x',t))|}\bigg),\frac{\pi_N(\nu_A(x',t))}{|\pi_1(\nu_A(x',t))|}\bigg)d\mathcal{H}^{N-2}(x')\\
\nonumber &\hspace{0.4in}=\int_{\partial^*A_{t}} \frac{f(\nu_A(x',t))}{|\pi_1(\nu_A(x',t))|}d\mathcal{H}^{N-2}(x')
\end{align}
where we have used Proposition $\ref{prop_jensen}$.

It now remains to verify the characterization of the equality case. Suppose that equality holds in ($\ref{eqclaim1}$) for some $t\in G$.  We must then have equality in the anisotropic isoperimetric inequality, so that the strict monotonicity of $x\mapsto \phi(x,y)$ implies
\begin{align*}
\int_{\partial^*A^*_{t}}h\bigg(\frac{\pi_1(\nu_{A^*}(x',t))}{|\pi_1(\nu_{A^*}(x',t))|}\bigg)d\mathcal{H}^{N-2}(x')=\int_{\partial^*A_{t}}h\bigg(\frac{\pi_1(\nu_A(x',t))}{|\pi_1(\nu_A(x',t))|}\bigg)d\mathcal{H}^{N-2}(x').
\end{align*}
The characterization of the Wulff shape $K_h$ as the unique minimizer for the \\anisotropic isoperimetric inequality (see, e.g. \cite{Taylor}, \cite{Fonseca}, \cite{FonsecaMuller}) then implies the desired claim.
\end{proof}

Equipped with this lemma, we turn to the proof of the main symmetrization result, Theorem $\ref{thm_symm_f}$.  The first part of the proof is based on an approximation procedure, reducing considerations to the case of polyhedral sets with no vertical normals.  For similar arguments, see the proofs of Lemma $3.5$ in \cite{CCF} and Lemma $3.3$ in \cite{BCF}.
\begin{proof}[Proof of Theorem $\ref{thm_symm_f}$.]
Let $A\subset\mathbb{R}^N$ be a given set of finite perimeter.  We approximate $A$ by a sequence of polyhedral sets $A_n$ such that
\begin{align*}
|A_n\Delta A|\rightarrow 0,\quad \textrm{and}\quad \mathscr{F}_s(A_n)\rightarrow\mathscr{F}_s(A)
\end{align*} 
as $n\rightarrow \infty$ (see, for instance, \cite[Proposition $4.9$]{Cerf}).  Note that without loss of generality (up to small perturbations of the faces), we may assume that the sets $A_n$ possess no vertical normals.  

For each $n\in\mathbb{N}$, we apply Lemmas $\ref{lem_coarea}$ and $\ref{lem_symm}$ to obtain
\begin{align*}
\mathscr{F}_s(A_n)&=\int_0^\infty \int_{\partial^*(A_n)_t} \frac{f(\nu_{A_n}(x',t))}{|\pi_1(\nu_{A_n}(x',t))|}d\mathcal{H}^{N-2}(x')dt\\
&\geq \int_0^\infty \int_{\partial^*(A_n^*)_t} \frac{f(\nu_{A_n^*}(x',t))}{|\pi_1(\nu_{A_n^*}(x',t))|}d\mathcal{H}^{N-2}(x')dt\\
&=\mathscr{F}_s(A_n^*),
\end{align*}
and hence 
\begin{align}
\mathscr{F}_s(A)=\lim_{n\rightarrow\infty} \mathscr{F}_s(A_n)\geq \limsup_{n\rightarrow\infty} \mathscr{F}_s(A_n^*).\label{eq12}
\end{align}

Next, note that for each $n\in\mathbb{N}$,
\begin{align*}
|A_n^*\Delta A^*|&=\int_{0}^\infty |v_{A_n^*}(t)-v_{A^*}(t)|dt\\
&\leq \int_0^\infty \mathcal{H}^{N-1}((A_n)_t\Delta A_t)dt=|A_n\Delta A|.
\end{align*}
Letting $n\rightarrow 0$, we obtain $|A_n^*\Delta A^*|\rightarrow 0$.  Then, writing
\begin{align*}
&\int_{\partial^*F\cap \{x:x_N>0\}} f(\nu_E(x))d\mathcal{H}^{N-1}(x)\\
&\hspace{0.2in}=\sup \left\{\int_F \ebdiv \phi(x)dx:\phi\in C_c^1(\{x:x_N>0\};K)\right\}
\end{align*}
for $F\in\mathcal{F}(\{x:x_N>0\})$, the Reshetnyak lower-semicontinuity theorems show that $\mathscr{F}_s$ is lower-semicontinuous with respect to $L_{\loc}^1$ convergence.  We therefore obtain
\begin{align*}
\mathscr{F}_s(A^*)\leq \liminf_{n\rightarrow\infty} \mathscr{F}_s(A_n^*).
\end{align*}
Combining this with ($\ref{eq12}$), we obtain
\begin{align}
\mathscr{F}_s(A)\geq \mathscr{F}_s(A^*).\label{eqsymm1}
\end{align}

On the other hand, by Definition $\ref{defsymm}$ we have
\begin{align*}
\mathcal{H}^{N-1}(\partial^*A\cap \{x:x_N=0\})=v_{A}^+(0)=\mathcal{H}^{N-1}(\partial^*A^*\cap \{x:x_N=0\}),
\end{align*}
while an application of Fubini's theorem yields $\mathscr{F}_p(A)=\mathscr{F}_p(A^*)$.  Combining these equalities with ($\ref{eqsymm1}$), the desired inequality ($\ref{eqsymmgoal}$) follows.
\end{proof}

\section{Existence of minimizers.}
\label{sec_exist}

In this section, we use the symmetrization results of Section $3$ to prove Theorem $\ref{thm_exist}$.  As mentioned in the introduction, the proof uses the direct method of the Calculus of Variations, using the results of the previous section to establish compactness and lower-semicontinuity properties for a minimizing sequence.  As mentioned in the introduction, this line of reasoning is inspired by the similar use of symmetrization to prove existence of minimizers for the isotropic problem in \cite{Gonzalez}.

We remark that when $\omega<0$, the lower-semicontinuity does not follow immediately from the classical lower-semicontinuity of the anisotropic perimeter (since the inequality holds in the opposite direction).  In this case, we use a calibration-style argument to compare the contact energy with a portion of the surface energy (see \cite{Gonzalez}, \cite{GonzalezMassariTamanini} and the references cited therein, and in particular the very clear treatment in \cite{MaggiBook}).  Note that the lack of symmetry in the anisotropic setting requires some care in choosing the appropriate vector field; we point out in particular the estimates ($\ref{aa1}$) and ($\ref{aa2}$).

\begin{proof}[Proof of Theorem $\ref{thm_exist}$]
We first obtain a lower bound for the values of $\mathscr{F}$.  Recalling that 
\begin{align}
\label{wulff_f}f(\nu)=\sup\{\nu\cdot x:x\in K\},\quad \nu\in\mathbb{R}^N,
\end{align}
we may choose $v\in \overline{K}$ such that $e_N\cdot v=f(e_N)$.  We then apply the divergence theorem to the constant vector field $x\mapsto v$, obtaining
\begin{align}
\nonumber 0=\int_E\ebdiv vdx&=\int_{\partial^*E\cap \{x:x_N=0\}} -f(e_N)d\mathcal{H}^{N-1}(x)\\
\nonumber &\hspace{0.2in}+\int_{\partial^*E\cap \{x:x_N>0\}} v\cdot \nu_E(x)d\mathcal{H}^{N-1}(x)\\
\nonumber &\leq -f(e_N)\mathcal{H}^{N-1}(\partial^*E\cap \{x:x_N=0\})\\
&\hspace{0.2in}+\int_{\partial^*E\cap \{x:x_N>0\}} f(\nu_E(x))d\mathcal{H}^{N-1}(x)\label{tracebd}
\end{align}
for every $E\in\mathcal{F}_m$, where we have used ($\ref{wulff_f}$) to obtain the last inequality.  Invoking $\omega>-f(e_N)$, we therefore obtain $\mathscr{F}(E)\geq 0$ for all $E\in\mathcal{F}_m$.

Moreover, $\mathscr{F}$ is lower-semicontinuous.  Indeed, it suffices to show the lower-semicontinuity of the functional 
\begin{align*}
E\mapsto \mathscr{F}_0(E):=\mathscr{F}_s(E)+\mathscr{F}_c(E).
\end{align*}
For $\omega\geq 0$, this follows by writing
\begin{align*}
\mathscr{F}_0(E)&=\left(1-\frac{\omega}{f(-e_N)}\right)\mathcal{F}_s(E)\\
&\hspace{0.2in}+\frac{\omega}{f(-e_N)}(f(-e_N)\mathcal{H}^{N-1}(\partial^*E\cap \{x:x_N=0\})+\mathcal{F}_s(E))\\
&=\left(1-\frac{\omega}{f(-e_N)}\right)\mathcal{F}_s(E)+\frac{\omega}{f(-e_N)}\int_{\partial^*E} f(\nu_E(x))d\mathcal{H}^{N-1}(x)
\end{align*}
and using the lower-semicontinuity of $\mathcal{F}_s(E)$ and $\int_{\partial^*E} f(\nu_E(x))d\mathcal{H}^{N-1}(x)$.  Turning to the case $\omega<0$, let $(E_n)$ be a sequence of sets of finite perimeter with $|E_n\Delta E|\rightarrow 0$.  

Choosing $v\in \overline{K}$ such that $v\cdot e_N=f(e_N)$, we obtain
\begin{align}
\nonumber &-\frac{f(e_N)}{\delta}|E_n\cap\{x:0<x_N<\delta\}|\\
\nonumber &\hspace{0.2in}=\int_{E_n}\ebdiv[\max\{1-\frac{x_N}{\delta},0\}v]dx\\
\nonumber &\hspace{0.2in}\leq -f(e_N)\mathcal{H}^{N-1}(\partial^*E_n\cap \{x:x_N=0\})\\
&\hspace{0.4in}+\int_{\partial^*E_n\cap \{x:0<x_N\leq \delta\}}f(\nu_{E_n}(x))d\mathcal{H}^{N-1}(x)\label{aa1}
\end{align}
On the other hand, choosing $v'\in \overline{K}$ such that $-v'\cdot e_N=f(-e_N)$, we find the inequality
\begin{align}
\nonumber &\frac{f(-e_N)}{\delta}|E\cap\{x:0<x_N<\delta\}|\\
\nonumber &\hspace{0.2in}=\int_{E}\ebdiv[\max\{1-\frac{x_N}{\delta},0\}v']dx\\
\nonumber &\hspace{0.2in}\leq f(-e_N)\mathcal{H}^{N-1}(\partial^*E\cap \{x:x_N=0\})\\
&\hspace{0.4in}+\int_{\partial^*E\cap \{x:0<x_N\leq \delta\}}f(\nu_{E}(x))d\mathcal{H}^{N-1}(x)\label{aa2}
\end{align}
We therefore have
\begin{align*}
&\mathcal{H}^{N-1}(\partial^*E_n\cap \{x:x_N=0\})-\mathcal{H}^{N-1}(\partial^*E\cap \{x:x_N=0\})\\
&\hspace{0.2in}\leq \frac{1}{\delta}(|E_n\cap \{x:0<x_N<\delta\}|-|E\cap \{x:0<x_N<\delta\}|)\\
&\hspace{0.4in}+\frac{1}{f(e_N)}\int_{\partial^*E_n\cap \{x:0<x_N\leq \delta\}}f(\nu_{E_n}(x))d\mathcal{H}^{N-1}(x)\\
&\hspace{0.4in}+\frac{1}{f(-e_N)}\int_{\partial^*E\cap \{0<x_N\leq \delta\}}f(\nu_E(x))d\mathcal{H}^{N-1}(x).
\end{align*}
Recalling that we have assumed $\omega<0$ and $\frac{\omega}{f(e_N)}>-1$, we now estimate, for each $n\in\mathbb{N}$,
\begin{align*}
\mathscr{F}_0(E_n)&=\mathscr{F}_0(E)+\mathscr{F}_s(E_n)-\mathscr{F}_s(E)+\omega\mathcal{H}^{N-1}(\partial^*E_n\cap \{x:x_N=0\})\\
&\hspace{0.2in}-\omega\mathcal{H}^{N-1}(\partial^*E\cap \{x:x_N=0\})\\
&\geq \mathscr{F}_0(E)+\int_{\partial^*E_n\cap \{x:x_N>\delta\}}f(\nu_{E_n}(x))d\mathcal{H}^{N-1}(x)-\mathscr{F}_s(E)\\
&\hspace{0.2in}+\frac{\omega}{\delta}(|E_n\cap\{x:0<x_N<\delta\}|-|E\cap \{x:0<x_N<\delta\}|)\\
&\hspace{0.2in}+\frac{\omega}{f(-e_N)}\int_{\partial^*E\cap\{x:0<x_N\leq \delta\}} f(\nu_E(x))d\mathcal{H}^{N-1}(x).
\end{align*}
Letting $n\rightarrow\infty$, we obtain
\begin{align*}
\liminf_{n\rightarrow\infty}\mathscr{F}_0(E_n)&\geq \mathscr{F}_0(E)+\int_{\partial^*E\cap \{x:x_N>\delta\}}f(\nu_E(x))d\mathcal{H}^{N-1}(x)-\mathscr{F}_s(E)\\
&\hspace{0.2in}+\frac{\omega}{f(-e_N)}\int_{\partial^*E\cap\{x:0<x_N\leq \delta\}} f(\nu_E(x))d\mathcal{H}^{N-1}(x),
\end{align*}
where we have used the lower semicontinuity of 
\begin{align*}
E\mapsto \int_{\partial^*E\cap \{x:x_N>\delta\}}f(\nu_E(x))d\mathcal{H}^{N-1}(x).
\end{align*}
Taking $\delta\rightarrow 0$ now yields
\begin{align*}
\liminf_{n\rightarrow\infty}\mathscr{F}_0(E_n)\geq \mathscr{F}_0(E),
\end{align*}
giving the desired lower semicontinuity for $\mathscr{F}_0$.

Choose a minimizing sequence $(E_n)\subset\mathcal{F}_m$ such that $\mathscr{F}(E_n)\rightarrow \mathscr{E}:=\inf_{E\in\mathcal{F}_m} \mathscr{F}(E)$, and let $E_n^*$ be the anisotropic symmetrization of $E_n$ for each $n\geq 1$.  The results of the previous section show that $E_n^*\in \mathcal{F}_m$ and
\begin{align}
\mathscr{F}(E_n^*)\rightarrow \mathscr{E}
\label{entoe}
\end{align}
as $n\rightarrow\infty$.  

By a standard application of compactness results, it suffices to show that for every $\epsilon>0$, there exists $T,R>0$ such that
\begin{align}
|E_n^*\setminus \{x:|x'|<R,0\leq x_N<T\}|<\epsilon\label{egoal1}
\end{align}
for every $n\geq 1$.

To obtain this, we let $\epsilon>0$ be given and consider the choice of $R$ and $T$ individually.  To choose $R$, note that for every $t\geq 0$ the symmetry of $E_n^*$ gives
\begin{align*}
\pi_1(E_n^*\cap \{x:x_N=t\})\subset \{x':|x'|\leq (v_{E_n^*}(t)/|K_h|)^\frac{1}{N-1}R_h\}.
\end{align*}
where $R_h=\inf\{R:K_h\subset B(0,R)\}$.  On the other hand, arguing as in ($\ref{tracebd}$), we obtain
\begin{align}
v_{E_n^*}(t)&\leq \frac{1}{f(e_N)}\int_{(\partial^*E_n^*)\cap \{x:x_N>t\}} f(\nu_{E_n^*}(x)) d\mathcal{H}^{N-1}(x).\label{eqa}
\end{align}
Note that for $\omega>0$ the right hand side of ($\ref{eqa}$) is bounded by $\frac{1}{f(e_N)}\mathscr{F}(E_n^*)$, while for $\omega<0$ the inequality ($\ref{tracebd}$) implies 
\begin{align*}
&\left(1+\frac{\omega}{f(e_N)}\right)\int_{\partial^*E_n^*} f(\nu_{E_n^*}(x))d\mathcal{H}^{N-1}(x)\\
&\hspace{0.2in}\leq \int_{\partial^*E_n^*} f(\nu_{E_n^*}(x))d\mathcal{H}^{N-1}(x)+\omega \mathcal{H}^{N-1}(\partial^*E\cap \{x:x_N=0\})\leq \mathscr{F}(E_n^*),
\end{align*}
so that the right side of $(\ref{eqa})$ is bounded by $\frac{1}{f(e_N)+\omega}\mathscr{F}(E_n^*)$.  Combining these observations, we obtain 
\begin{align}
\pi_1(E_n^*)=\bigcup_{t\geq 0} \pi_1(E_n^*\cap \{x:x_N=t\})\subset \{x':|x'|\leq R\}\label{rbound}
\end{align}
where we have set
\begin{align*}
R:= \left(\frac{\sup_n\mathscr{F}(E_n^*)}{|K_h|}\max\left\{ \frac{1}{f(e_N)} ,\frac{1}{f(e_N)+\omega}\right\}\right)^\frac{1}{N-1}R_h.
\end{align*}
and noted that ($\ref{entoe}$) implies $\sup_n \mathscr{F}(E_n^*)<\infty$.

To choose $T$, we note that the inclusion ($\ref{rbound}$) followed by an invocation of Tchebyshev's inequality implies that for any $n\in\mathbb{N}$ and $T>0$,
\begin{align}
\nonumber |E_n^*\setminus \{x:|x'|<R,0\leq x_N<T\}|&\leq |E_n^*\setminus \{x:0\leq x_N<T\}|\\
&\leq \frac{1}{T}\int_{E_n^*\cap \{x:x_N\geq T\}} x_Ndx.\label{rhs1}
\end{align}
The bound ($\ref{tracebd}$) then allows us to bound the right hand side of ($\ref{rhs1}$) by
\begin{align*}
\frac{1}{T}\sup_n \mathscr{F}(E_n^*).
\end{align*}
Observing once again that $\sup_n \mathscr{F}(E_n^*)<\infty$ as a consequence of $(\ref{entoe})$, the inequality $(\ref{egoal1})$ follows by choosing $T$ sufficiently large.
\end{proof}

\section{Regularity properties of symmetric minimizers.}

In this section, we study the regularity properties of symmetric minimizers for $\mathscr{F}$, establishing Theorem $\ref{thm_concave}$.  We divide our analysis into three steps.  The first step is to introduce a family of rescaled and truncated copies of the Wulff shape $K$ for the function $f$ appearing in the definition of an admissible surface tension.  This family of sets is adapted to allow for the construction of competitors for minimality candidates of the functional $\mathscr{F}$.  The second step in our analysis is then devoted to this construction: by constructing a suitable competitor, we show that any minimizer for $E$ cannot have local points of concavity; this is the content of Lemma $\ref{lem_competitor}$.  

Finally, the third step is to complete the proof of Theorem $\ref{thm_concave}$ by showing that the function $t\mapsto v_E(t)=|E_t|$ is first continuous (Proposition $\ref{prop_cty}$), and then that the minimizer is in fact convex (i.e. the function $r_E(t)=(v_E(t)/|K_h|)^{1/(N-1)}$ is concave on its support).

We remark that the construction given in this section is inspired by the proof of convexity of minimizers for the isotropic case \cite{GonzalezTamanini}, and can be seen as a localization of that construction (in \cite{GonzalezTamanini}, competitors are constructed by replacing the entire top portion of a candidate set, whereas we replace only a section).  We point out that the argument in \cite{GonzalezTamanini} requires rather strong regularity properties (in particular, analyticity) of minimizers which in that setting follow from the regularity theory for minimal surfaces.  In the present setting this is not available in general, and we work instead with fine properties of sets of finite perimeter.

\subsection{A construction involving the Wulff shape}
\label{sec_construction}

Fix $m>0$ and $E\in\mathcal{F}_m$ with $E=E^*$.  Let $f$ be given admissible with $f(x)=\phi(h(x'),x_N)$, and let $K$ be the Wulff shape associated to $f$.  We begin by defining a family of rescaled and truncated copies of $f$:
\begin{defin}
\label{def_top}
For each $t>0$, $\sigma\in (\inf \pi_N(K),\sup\pi_N(K))$, define
\begin{align*}
b_-(E,\sigma,t):=\bigg(\frac{v^-_E(t)}{v_K(\sigma)}\bigg)^{1/(N-1)},\quad b_+(E,\sigma,t):=\bigg(\frac{v^+_E(t)}{v_K(\sigma)}\bigg)^{1/(N-1)}
\end{align*}
and
\begin{align*}
K_{+}(E,\sigma,t)&=te_N+\bigg(-\big[b_-(E,\sigma,t)\sigma\big]e_N+b_-(E,\sigma,t)(K\cap \{x:x_N>\sigma\})\bigg),\\
K_{-}(E,\sigma,t)&=te_N+\bigg(-\big[b_+(E,\sigma,t)\sigma\big]e_N+b_+(E,\sigma,t)(K\cap \{x:x_N<\sigma\})\bigg).
\end{align*}
\end{defin}
The sets $K_+(E,\sigma,t)$ and $K_-(E,\sigma,t)$ are chosen so that $K_+(E,\sigma,t)$ is a rescaled copy of $K\cap \{x:x_N>\sigma\}$, with the dilation chosen so that the measure of the truncated side is equal to $v_E^{-}(t)$, and translated so that the truncated side lies at height $t$, with $K_-(E,\sigma,t)$ a rescaled copy of $K\cap \{x:x_N<\sigma\}$ satisfying an analogous condition.
Moreover, the admissibility of $f$ (see Definition $\ref{def_adm}$) implies 
\begin{align*}
\bigcup_{\inf \pi_N(K)<\sigma<0} K_+(K,\sigma,0)=\mathbb{R}^{N-1}\times (0,\infty),\quad \bigcap_{0<\sigma<\sup \pi_N(K)} K_+(K,\sigma,0)=\emptyset
\end{align*}
and
\begin{align*}
\bigcup_{0<\sigma<\sup \pi_N(K)} K_-(K,\sigma,0)=\mathbb{R}^{N-1}\times (-\infty,0),\quad \bigcap_{\inf \pi_N(K)<\sigma<0} K_-(K,\sigma,0)=\emptyset.
\end{align*}

\begin{figure}
\centering
\def\svgwidth{0.92\columnwidth}

\begingroup
  \makeatletter
  \providecommand\color[2][]{%
    \errmessage{(Inkscape) Color is used for the text in Inkscape, but the package 'color.sty' is not loaded}
    \renewcommand\color[2][]{}%
  }
  \providecommand\transparent[1]{%
    \errmessage{(Inkscape) Transparency is used (non-zero) for the text in Inkscape, but the package 'transparent.sty' is not loaded}
    \renewcommand\transparent[1]{}%
  }
  \providecommand\rotatebox[2]{#2}
  \ifx\svgwidth\undefined
    \setlength{\unitlength}{1641.88808594pt}
  \else
    \setlength{\unitlength}{\svgwidth}
  \fi
  \global\let\svgwidth\undefined
  \makeatother
  \begin{picture}(1,0.36255552)%
    \put(0,0){\includegraphics[width=\unitlength]{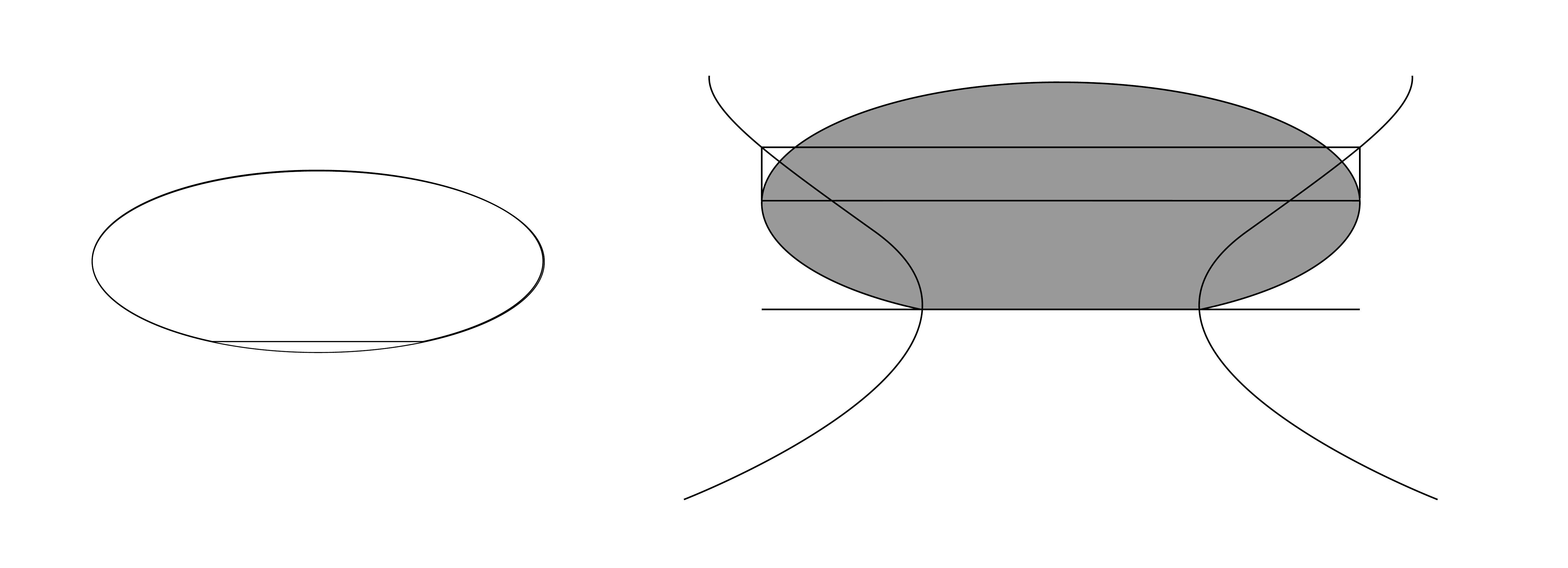}}%
    \put(0.43409279,0.15853576){\color[rgb]{0,0,0}\makebox(0,0)[lb]{\smash{$t_1$}}}%
    \put(0.43409279,0.26402483){\color[rgb]{0,0,0}\makebox(0,0)[lb]{\smash{$t_2$}}}%
    \put(0.43409279,0.22988957){\color[rgb]{0,0,0}\makebox(0,0)[lb]{\smash{$\tau_1$}}}%
    \put(0.57327593,0.03670082){\color[rgb]{0,0,0}\makebox(0,0)[lb]{\smash{$E$}}}%
    \put(0.12167175,0.0356384){\color[rgb]{0,0,0}\makebox(0,0)[lb]{\smash{$K$}}}%
    \put(0.05011652,0.13782043){\color[rgb]{0,0,0}\makebox(0,0)[lb]{\smash{$\sigma_1$}}}%
    \put(0.87262818,0.22181087){\color[rgb]{0,0,0}\makebox(0,0)[lb]{\smash{$K_+(E,\sigma_1,\tau_1)$}}}%
  \end{picture}%
\endgroup
\caption{The set $K_+(E,\sigma_1,t_1)$ of Definition $\ref{def_top}$ with the parameters $\sigma_1$ and $\tau_1$ chosen in Lemma $\ref{lem_construction}$.  Given two heights $t_1<t_2$ we find parameters $\sigma_1$ and $\tau_1$ such that $K_+(E,\sigma_1,t_1)$ - which is a rescaled and translated copy of a truncation of $K$ such that $v_{K_+(E,\sigma_1,t_1)}(t_1)=v^-_E(t_1)$ - can be truncated to the interval $\{(x',t):t_1<t<\tau_1\}$ so that $v_{K_+(E,\sigma_1,t_1)}(\tau_1)=v_E^+(t_2)$ and the truncation has equal measure with $E\cap \{x:t_1<x_N<t_2\}$.  The parameter $\sigma_1$ determines the section of the Wulff shape to be rescaled (depicted at left).\label{fig3}}
\end{figure}
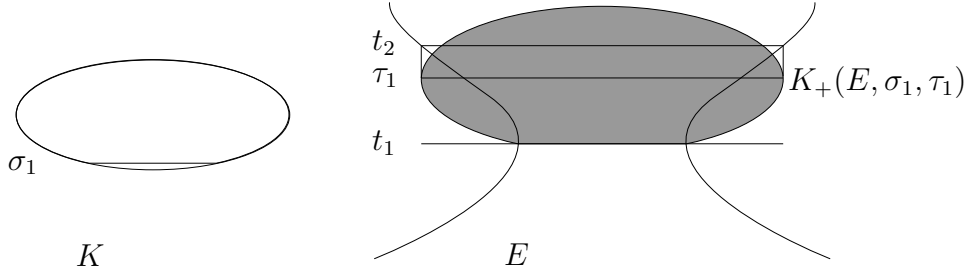

The utility of this construction is based in the following lemma, which shows how to choose the parameter $\sigma$ to construct competitor sets for minimality candidates of the functional $\mathscr{F}$ while respecting the volume constraint.  The idea is that given two heights $t_1<t_2$, one can find parameters $\sigma$ and $\tau$ such that the rescaled set $K_+(E,\sigma,t_1)$ - which has $v_{K_+(E,\sigma,t_1)}(t_1)=v^-_E(t_1)$ by construction, and therefore ``agrees'' with $E$ from below - can be truncated to the interval $\{(x',t):t_1<t<\tau\}$ so that the truncation also agrees with $E\cap \{(x',t):t>t_2\}$ from above (in the sense that $v_{K_+(E,\sigma,t_1)}(\tau)=v_E^+(t_2)$) and has equal measure with $E\cap \{x:t_1<x_N<t_2\}$.  

Once this lemma is established, the competing sets will be constructed in the next section by replacing the set $E\cap \{x:t_1<x_N<t_2\}$ in $E$ by $K_+(E,\sigma,t_1)\cap \{x:t_1<x_N<\tau\}$.  A similar construction is given for $K_-$.  The proof of this lemma is based on continuity and monotonicity properties of the Lebesgue measure, first varying the parameter $\sigma$ to determine a minimal value which allows the volume constraint to be satisfied, and subsequently increasing the value of $\sigma$ while adjusting $\tau$ to enforce the volume constraint.

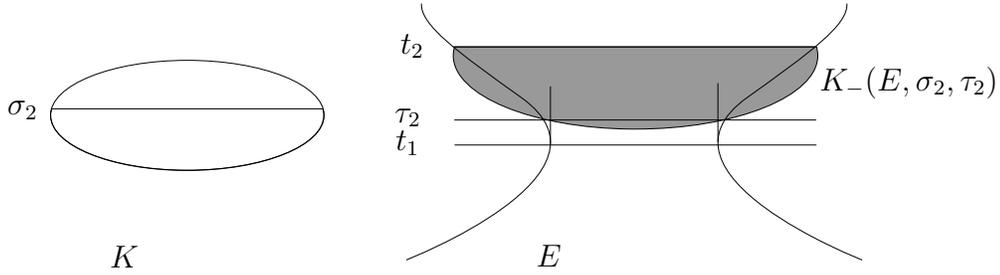
\begin{figure}[h]
\centering
\def\svgwidth{0.92\columnwidth}

\begingroup
  \makeatletter
  \providecommand\color[2][]{%
    \errmessage{(Inkscape) Color is used for the text in Inkscape, but the package 'color.sty' is not loaded}
    \renewcommand\color[2][]{}%
  }
  \providecommand\transparent[1]{%
    \errmessage{(Inkscape) Transparency is used (non-zero) for the text in Inkscape, but the package 'transparent.sty' is not loaded}
    \renewcommand\transparent[1]{}%
  }
  \providecommand\rotatebox[2]{#2}
  \ifx\svgwidth\undefined
    \setlength{\unitlength}{1641.88808594pt}
  \else
    \setlength{\unitlength}{\svgwidth}
  \fi
  \global\let\svgwidth\undefined
  \makeatother
  \begin{picture}(1,0.36255552)%
    \put(0,0){\includegraphics[width=\unitlength]{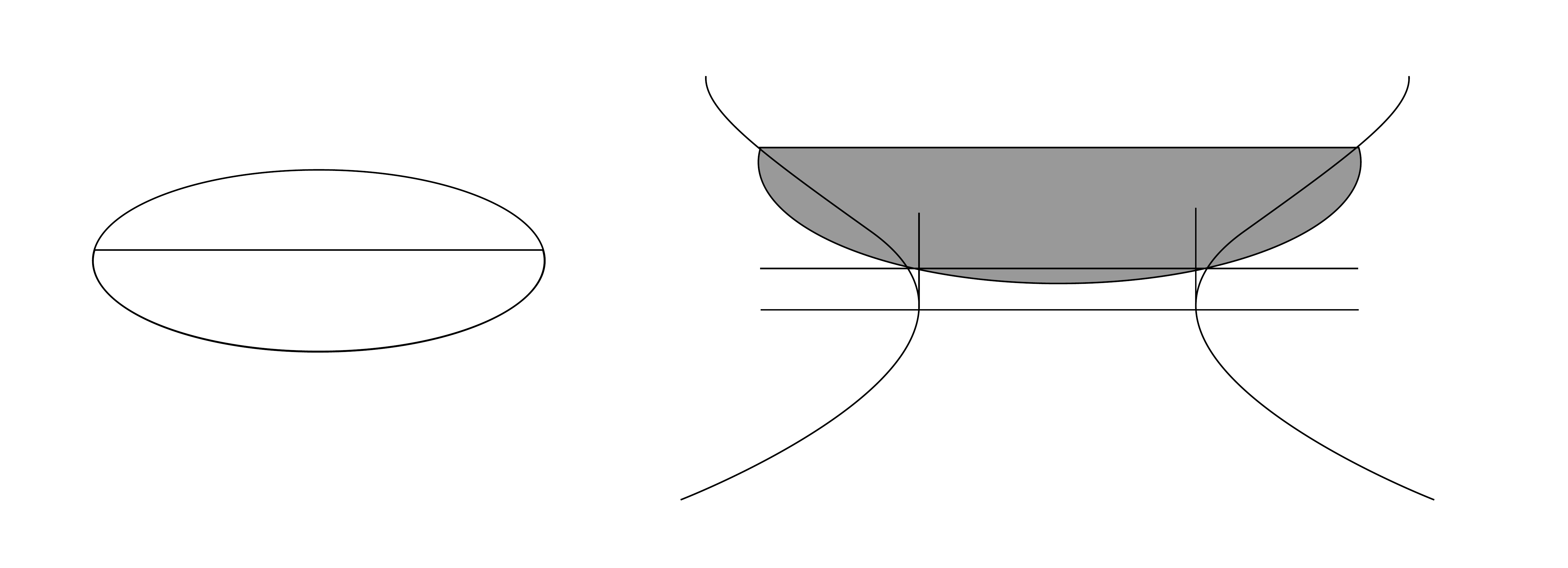}}%
    \put(0.42310188,0.15944828){\color[rgb]{0,0,0}\makebox(0,0)[lb]{\smash{$t_1$}}}%
    \put(0.42797433,0.26145705){\color[rgb]{0,0,0}\makebox(0,0)[lb]{\smash{$t_2$}}}%
    \put(0.42226661,0.18917753){\color[rgb]{0,0,0}\makebox(0,0)[lb]{\smash{$\tau_2$}}}%
    \put(0.57133323,0.03691747){\color[rgb]{0,0,0}\makebox(0,0)[lb]{\smash{$E$}}}%
    \put(0.12194408,0.03607059){\color[rgb]{0,0,0}\makebox(0,0)[lb]{\smash{$K$}}}%
    \put(0.01392125,0.19698576){\color[rgb]{0,0,0}\makebox(0,0)[lb]{\smash{$\sigma_2$}}}%
    \put(0.87007848,0.22274008){\color[rgb]{0,0,0}\makebox(0,0)[lb]{\smash{$K_-(E,\sigma_2,\tau_2)$}}}%
  \end{picture}%
\endgroup
\caption{The set $K_-(E,\sigma_2,t_2)$ of Definition $\ref{def_top}$ with parameters $\sigma_2$ and $\tau_2$ chosen as in Lemma $\ref{lem_construction}$.  The value of $\sigma_2$ determines the section of the Wulff shape to be rescaled, and is depicted at the left.  The construction is analogous to that depicted in Figure $\ref{fig3}$.\label{fig3b}}
\end{figure}
\begin{lem}
\label{lem_construction}
Fix $t_1,t_2\in \pi_N(E)$ such that $t_1<t_2$.  Then there exist $(\sigma_1,\tau_1)\in \pi_N(K)\times (t_1,\infty)$ and $(\sigma_2,\tau_2)\in \pi_N(K)\times (-\infty,t_2)$ such that 
\begin{enumerate}
\item[(i)] $v_E^+(t_2)=v_{K_{+}(E,\sigma_1,t_1)}^+(\tau_1)$, 
\item[(ii)] $|E\cap \{x:t_1<x_N<t_2\}|=|K_{+}(E,\sigma_1,t_1)\cap \{x:t_1<x_N<\tau_1\}|$.
\end{enumerate}
and
\begin{enumerate}
\item[(iii)] $v_E^-(t_1)=v^-_{K_-(E,\sigma_2,t_2)}(\tau_2)$,
\item[(iv)] $|E\cap \{x:t_1<x_N<t_2\}|=|K_{-}(E,\sigma_2,t_2)\cap \{x:\tau_2<x_N<t_2\}|$.
\end{enumerate}
\end{lem}
\begin{proof}
We define $f_1:\ebint\pi_N(K)\rightarrow [0,\infty)$ by
\begin{align*}
f_1(\sigma)=|K_{+}(E,\sigma,t_1)|,
\end{align*}
and note that $f_1$ is continuous and satisfies the limits $f_1(\sigma)\rightarrow \infty$ as $\sigma\rightarrow\inf \pi_N(K)$, $f_1(\sigma)\rightarrow 0$ as $\sigma\rightarrow\sup\pi_N(K)$.  Indeed, this follows from the convexity of $K$ combined with the remarks above and the continuity properties of the Lebesgue measure.

We may therefore choose $\sigma_0\in (\inf \pi_N(K),\sup\pi_N(K))$ such that $f_1(\sigma_0)=|E\cap \{x:t_1<x_N<t_2\}|$.  Then for every $\sigma\leq \sigma_0$, we may define
\begin{align*}
\tau_1(\sigma)=\sup \bigg\{t\in \pi_N(K_{+}(E,\sigma,t_1))&:|K_+(E,\sigma,t_1)\cap \{x:t_1<x_N<t\}|\\
&\hspace{0.2in}\leq |E\cap \{x:t_1<x_N<t_2\}|\bigg\}.
\end{align*}
Note that with this choice of $\tau_1(\sigma)$, we have $|K_+(E,\sigma,t_1)\cap \{x:t_1<x_N<\tau_1(\sigma)\}|=|E\cap \{x:t_1<x_N<t_2\}|$ for $\sigma\leq\sigma_0$.

We now define $f_2:(\inf \pi_N(K),\sigma_0]\rightarrow [0,\infty)$ by
\begin{align*}
f_2(\sigma)=v_{K_+(E,\sigma,t_1)}^+(\tau_1(\sigma)),
\end{align*}
and observe that $f_2$ is continuous, and satisfies $f_2(\sigma_0)=0$ as well as the limit $f_2(\sigma)\rightarrow \infty$ as $\sigma\rightarrow \inf \pi_N(K)$. We may therefore choose $\sigma\in (\inf \pi_N(K),\sigma_0]$ such that $f_2(\sigma)=v_E^+(t_2)$.  This completes the construction, yielding the desired parameters $\sigma_1$ and $\tau_1=\tau_1(\sigma)$.  The construction of $(\sigma_2,\tau_2)$ is analogous.
\end{proof}

In order to compare values of the functional $\mathscr{F}$ at the candidate set $E$ and the competitor set constructed using Lemma $\ref{lem_construction}$, we will use the following comparison lemma, which shows that the surface energy of the set $K_+$ restricted to the interval $t\in (t_1,\tau_1)$ is smaller than the surface energy of $E$ restricted to the interval $t\in(t_1,t_2)$, along with a similar claim for $K_-$.  The proof is based on attaching sets to the top and bottom of both sections and using the anisotropic isoperimetric inequality (see Figure $\ref{fig6}$).
\begin{lem}
\label{lem_compare_s}
Let $f$ be given admissible with $f(x)=\phi(h(x'),x_N)$.  Fix $\sigma\in \ebint \pi_N(K)$ and $t_1,t_2,\tau\in\mathbb{R}$ with $t_1<t_2$, $v_E^+(t_2)=v_{K_+(E,\sigma,t_1)}^+(\tau)$ and $|E\cap \{x:t_1<x_N<t_2\}|=|K_+(E,\sigma,t_1)\cap \{x:t_1<x_N<\tau\}|$.  Then 
\begin{align*}
\int_{\partial^*K_+(E,\sigma,t_1)\cap \{x:t_1<x_N<\tau\}} &f(\nu_{K_+(E,\sigma,t_1)}(x))d\mathcal{H}^{N-1}(x)\\
&\hspace{0.2in}\leq \int_{\partial^*E\cap \{x:t_1\leq x_N\leq t_2\}} f(\nu_E(x))d\mathcal{H}^{N-1}(x).
\end{align*}
Alternatively, if $v_E^-(t_1)=v_{K_{-}(E,\sigma,t_2)}^{-}(\tau)$ and $|E\cap \{x:t_1<x_N<x_2\}|=|K_-(E,\sigma,t_2)\cap \{x:\tau<x_N<t_2\}|$, then
\begin{align*}
\int_{\partial^*K_-(E,\sigma,t_2)\cap \{x:\tau<x_N<t_2\}} &f(\nu_{K_-(E,\sigma,t_2)}(x))d\mathcal{H}^{N-1}(x)\\
&\hspace{0.2in}\leq \int_{\partial^*E\cap \{x:t_1\leq x_N\leq t_2\}} f(\nu_E(x))d\mathcal{H}^{N-1}(x).
\end{align*}
\end{lem}

\begin{proof}
We first consider the case $v_E^+(t_2)=v_{K_+(E,\sigma,t_1)}^+(\tau)$.  For this case, we define $E'\subset\mathbb{R}^N$ by
\begin{align*}
E'&=\bigg( K'\cap \{x:x_N\leq t_1\} \bigg)\cup \bigg( E\cap \{x:t_1<x_N<t_2\} \bigg) \\
&\hspace{0.2in}\cup \bigg( (t_2-\tau)e_N+K'\cap \{x:x_N\geq \tau\}\bigg),
\end{align*}
where
\begin{align*}
K':=t_1e_N+\big(-\big[b_-(E,\sigma,t_1)\sigma\big]e_N+b_-(E,\sigma,t_1)K\big),
\end{align*}
(so that $K_+(E,\sigma,t_1)=K'\cap \{x:x_N>t_1\}$).  We then have $|E'|=|K'|$, so that the anisotropic isoperimetric (Wulff) inequality yields
\begin{align*}
\int_{\partial^*K'} f(\nu_{K'}(x))d\mathcal{H}^{N-1}(x)\leq \int_{\partial^*E'} f(\nu_{E'}(x))d\mathcal{H}^{N-1}(x),
\end{align*}
where we have observed that $K'$ is obtained from the Wulff shape $K$ by an affine transformation.  The desired inequality then follows by noting that the equalities
\begin{align*}
\int_{(\partial^*K')\cap \{x:x_N\leq t_1\}} f(\nu_{K'}(x))d\mathcal{H}^{N-1}(x)=\int_{(\partial^*E')\cap \{x:x_N<t_1\}} f(\nu_{E'}(x))d\mathcal{H}^{N-1}(x)
\end{align*}
and
\begin{align*}
\int_{(\partial^*K')\cap \{x:x_N\geq \tau\}} f(\nu_{K'}(x))d\mathcal{H}^{N-1}(x)=\int_{(\partial^*E')\cap \{x:x_N>t_2\}} f(\nu_{E'}(x))d\mathcal{H}^{N-1}(x)
\end{align*}
follow from the construction of $K'$ and $E'$.

The case $v_E^-(t_1)=v_{K_-(E,\sigma,t_2)}^-(\tau)$ is similar, considering the set
\begin{align*}
E''&=\bigg( K''\cap \{x:x_N\leq \tau\} \bigg) \cup \bigg( (\tau-t_1)e_N+E\cap \{x:t_1<x_N<t_2\} \bigg) \\
&\hspace{0.2in}\cup \bigg( (\tau-t_1)e_N+K''\cap \{x:x_N\geq t_2\} \bigg),
\end{align*}
where
\begin{align*}
K'':=t_2e_N+\big(-\big[b_+(E,\sigma,t_2)\sigma\big]e_N+b_+(E,\sigma,t_2)K\big)
\end{align*}
in place of the set $E'$ above, and repeating the same invocation of the Wulff inequality.
\end{proof}

\begin{figure}[h]
\centering
\def\svgwidth{0.92\columnwidth}

\begingroup
  \makeatletter
  \providecommand\color[2][]{%
    \errmessage{(Inkscape) Color is used for the text in Inkscape, but the package 'color.sty' is not loaded}
    \renewcommand\color[2][]{}%
  }
  \providecommand\transparent[1]{%
    \errmessage{(Inkscape) Transparency is used (non-zero) for the text in Inkscape, but the package 'transparent.sty' is not loaded}
    \renewcommand\transparent[1]{}%
  }
  \providecommand\rotatebox[2]{#2}
  \ifx\svgwidth\undefined
    \setlength{\unitlength}{1600pt}
  \else
    \setlength{\unitlength}{\svgwidth}
  \fi
  \global\let\svgwidth\undefined
  \makeatother
  \begin{picture}(1,0.37204724)%
    \put(0,0){\includegraphics[width=\unitlength]{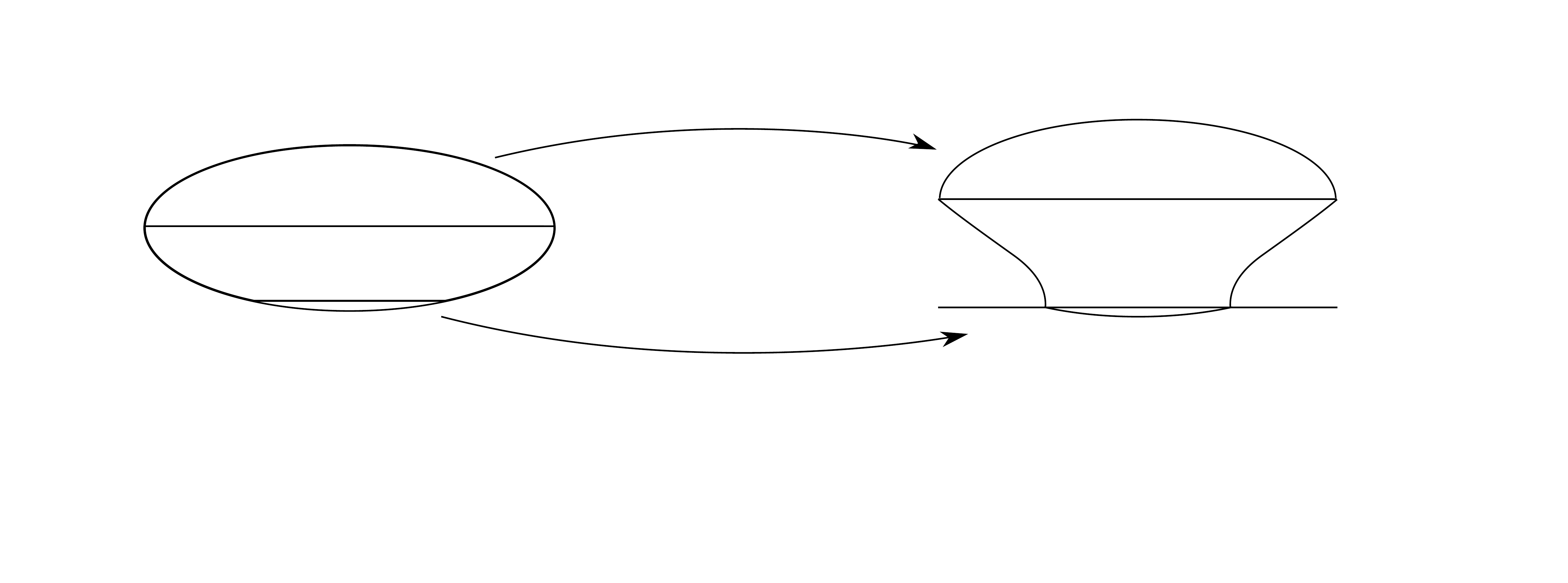}}%
    \put(0.09580281,0.11183368){\color[rgb]{0,0,0}\makebox(0,0)[lb]{\smash{$K_t\cap\{(x',t):t\in(t_1,\tau)\}$}}}%
    \put(0.59316077,0.10964935){\color[rgb]{0,0,0}\makebox(0,0)[lb]{\smash{$E\cap \{(x',t):t\in (t_1,t_2)\}$}}}%
    \put(0.54418005,0.17213367){\color[rgb]{0,0,0}\makebox(0,0)[lb]{\smash{$t_1$}}}%
    \put(0.06256049,0.17919315){\color[rgb]{0,0,0}\makebox(0,0)[lb]{\smash{$t_1$}}}%
    \put(0.06168568,0.23080624){\color[rgb]{0,0,0}\makebox(0,0)[lb]{\smash{$\tau$}}}%
    \put(0.54368365,0.24166397){\color[rgb]{0,0,0}\makebox(0,0)[lb]{\smash{$t_2$}}}%
  \end{picture}%
\endgroup
\caption{The proof of Lemma $\ref{lem_compare_s}$, which states that the surface energy of the set $K_+$ constructed in Lemma $\ref{lem_construction}$ (see Figure $\ref{fig3}$) restricted to the intervals $t\in (t_1,\tau_1)$ is smaller than the surface energy of the original set $E$ restricted to the interval $t\in (t_1,t_2)$.  To prove the lemma, we complete $K_+$ into a rescaling of the Wulff shape $K$ by attaching sets to the top and bottom of $K_+\cap \{(x',t):t\in (t_1,\tau_1)\}$ (depicted on the left) and attach the same sets to $E\cap \{t\in (t_1,t_2)\}$ (depicted on the right).  The construction of $K_+$ and choice of $\tau$ in Lemma $\ref{lem_construction}$ then implies that both sets have equal measures, while the minimality of the Wulff shape ensures the desired decrease in surface energy.  A similar statement holds for $K_-$.\label{fig6}}
\end{figure}
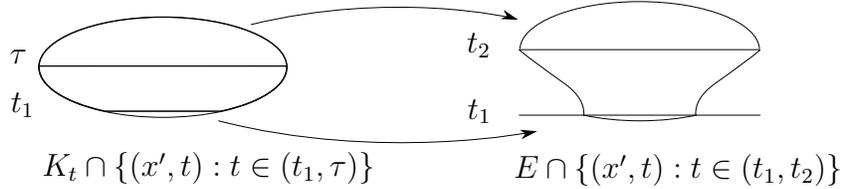

\subsection{Competitor lemma}

The goal of this section is to prove a competitor lemma for minimality candidates for $\mathscr{F}$ to show that minimizers cannot have local points of concavity.  The proof of such a lemma is based in the construction of the previous section, which shows how one can replace a section of the candidate for minimality by a rescaled section of the Wulff shape while preserving the volume constraint.  Lemma $\ref{lem_compare_s}$ of the previous section shows that this procedure reduces the surface energy.  Moreover, in certain cases and when some additional concavity is assumed, the convexity of the Wulff shape implies that the potential energy is reduced as well (this is a consequence of the fact that mass moved downward contributes less to the potential energy; see Figure $\ref{fig-competitor}$).  In particular, we obtain
\begin{lem}
\label{lem_competitor}
Fix $m>0$, $\omega\in (-f(e_N),f(-e_N))$ and let $f$ be given admissible with $f(x)=\phi(h(x'),x_N)$.  Let $E\in \mathcal{F}_m$ be given such that $E=E^*$.  Suppose that there exist $t_2>t_1>0$ such that for a.e. $s\in (0,1)$,
\begin{align}
\label{hyp1}
r_{E}(st_1+(1-s)t_2)<s\cdot r_{E}^-(t_1)+(1-s)r_{E}^+(t_2).
\end{align}
Moreover, if $r_E^-(t_1)>r_E^+(t_2)$, suppose also that 
\begin{align}
\label{hyp2}
t_2-t_1<\frac{|E\cap \{x:x_N\geq t_2\}|}{\lVert v_E(t)\rVert_{L^\infty}}.
\end{align}
Then there exists $E'\in\mathcal{F}_m$ with $\mathscr{F}(E')<\mathscr{F}(E)$.  In particular, $E$ is not a minimizer for $F$.
\end{lem}
The proof of Lemma $\ref{lem_competitor}$ splits into two cases.  When $r_E^-(t_1)\leq r_E^+(t_2)$ it is easy to show that the procedure depicted in Figure $\ref{fig-competitor}$ decreases the potential energy (since all motion of mass occurs in the downward direction).  On the other hand when $r_E^-(t_1)>r_E^+(t_2)$, in order to make use of the concavity hypothesis ($\ref{hyp1}$), we work with the sets $K_-(E,\sigma_2,t_2)$ of Lemma $\ref{lem_construction}$; this variant of the procedure is depicted in Figure $\ref{fig-competitor-2}$).  However, with this construction it is possible that some mass is moved upwards.  In order to show that the functional $\mathscr{F}$ decreases, we therefore make use of the hypothesis ($\ref{hyp2}$), which ensures sufficient decrease in potential energy from the translation of the mass above the height $t=t_2$.

\begin{proof}[Proof of Lemma $\ref{lem_competitor}$]  Our argument proceeds by considering the cases $r_E^-(t_1)\leq r_E^+(t_2)$ and $r_E^-(t_1)>r_E^+(t_2)$ individually.

\begin{figure}
\centering
\def\svgwidth{0.62\columnwidth}

\begingroup
  \makeatletter
  \providecommand\color[2][]{%
    \errmessage{(Inkscape) Color is used for the text in Inkscape, but the package 'color.sty' is not loaded}
    \renewcommand\color[2][]{}%
  }
  \providecommand\transparent[1]{%
    \errmessage{(Inkscape) Transparency is used (non-zero) for the text in Inkscape, but the package 'transparent.sty' is not loaded}
    \renewcommand\transparent[1]{}%
  }
  \providecommand\rotatebox[2]{#2}
  \ifx\svgwidth\undefined
    \setlength{\unitlength}{841.88974609pt}
  \else
    \setlength{\unitlength}{\svgwidth}
  \fi
  \global\let\svgwidth\undefined
  \makeatother
  \begin{picture}(1,0.70707072)%
    \put(0,0){\includegraphics[width=\unitlength]{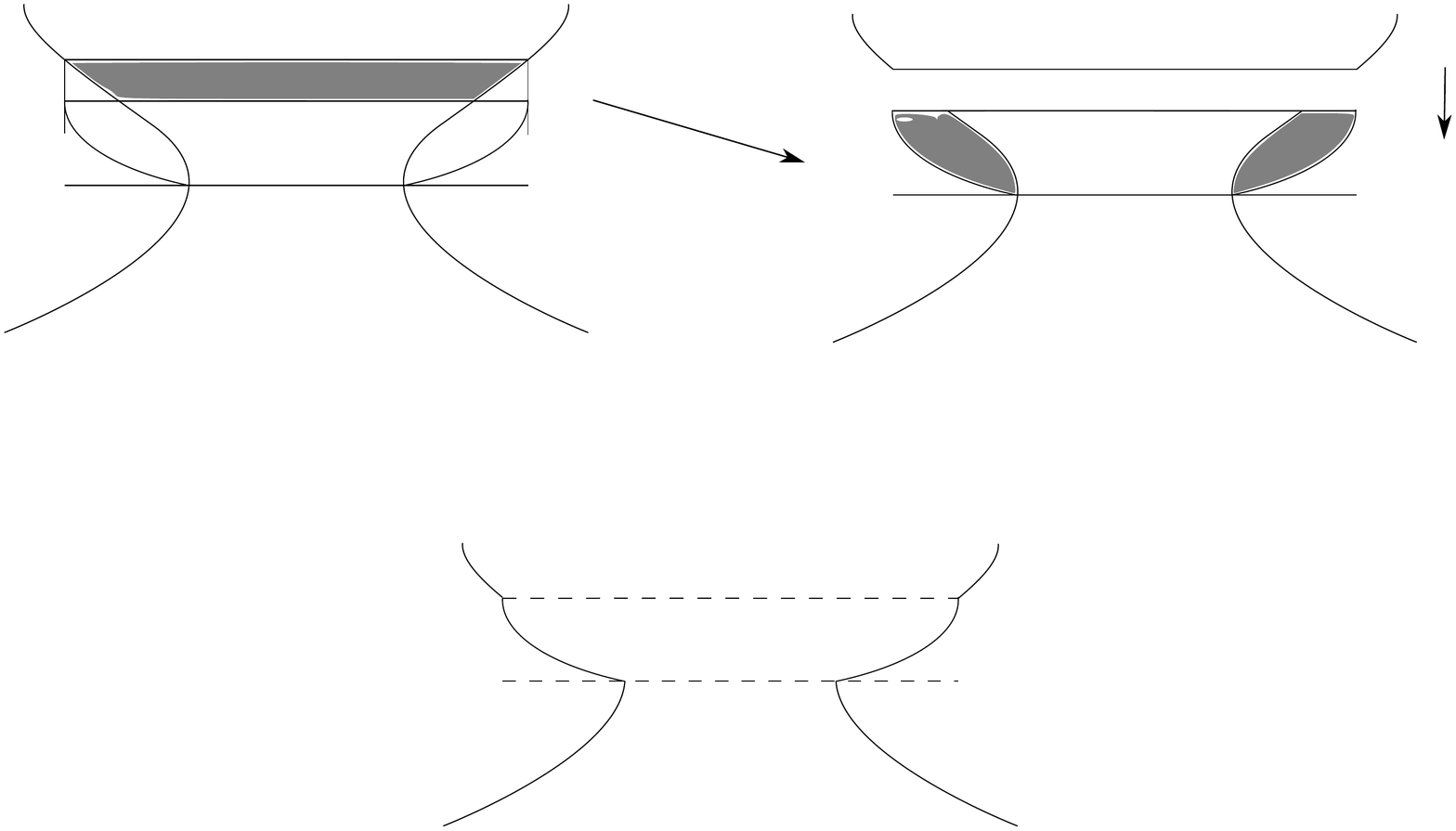}}%
    \put(0.01856811,0.47506506){\color[rgb]{0,0,0}\makebox(0,0)[lb]{\smash{$t_1$}}}%
    \put(0.01178066,0.56422758){\color[rgb]{0,0,0}\makebox(0,0)[lb]{\smash{$t_2$}}}%
    \put(0.01178066,0.52896467){\color[rgb]{0,0,0}\makebox(0,0)[lb]{\smash{$\tau_1$}}}%
    \put(0.10338534,0.38503798){\color[rgb]{0,0,0}\makebox(0,0)[lb]{\smash{$E$}}}%
    \put(0.55415825,0.4686618){\color[rgb]{0,0,0}\makebox(0,0)[lb]{\smash{$t_1$}}}%
    \put(0.54872835,0.56868428){\color[rgb]{0,0,0}\makebox(0,0)[lb]{\smash{$t_2$}}}%
    \put(0.54872828,0.52934887){\color[rgb]{0,0,0}\makebox(0,0)[lb]{\smash{$\tau_1$}}}%
    \put(0.29418725,0.15257093){\color[rgb]{0,0,0}\makebox(0,0)[lb]{\smash{$t_1$}}}%
    \put(0.40054051,0.05032687){\color[rgb]{0,0,0}\makebox(0,0)[lb]{\smash{$E'$}}}%
    \put(0.28875728,0.21095475){\color[rgb]{0,0,0}\makebox(0,0)[lb]{\smash{$\tau_1$}}}%
  \end{picture}%
\endgroup
\caption{The construction used to establish Case $1$ in the proof of Lemma $\ref{lem_competitor}$.  The section $E\cap \{(x',t):t_1<t<t_2\}$ of $E$ is replaced with $K_+(E,\sigma_1,t_1)\cap \{(x',t):t_1<t<\tau_1\}$.  The upper portion of the set $E$ is then translated downward to rest on top of the replaced section.  The surface energy decreases as a consequence of Lemma $\ref{lem_compare_s}$, while the potential energy decreases because all motion of mass occurs downward.\label{fig-competitor}}
\end{figure}

\vspace{0.2in}

\underline{Case 1}: $r^-_{E}(t_1)\leq r^+_E(t_2)$.  

\vspace{0.2in}

Let $\sigma=\sigma_1$ and $\tau=\tau_1$ be as constructed in Lemma \ref{lem_construction}.  We first claim that $\tau<t_2$.  Note that ($\ref{hyp1}$) implies 
\begin{align*}
&E\cap \{x:t_1<x_N<t_2\}\\
&\hspace{0.2in}\subset \bigg\{x:t_1<x_N<t_2,\pi_1(x)\in \bigg(\frac{t_2-x_N}{t_2-t_1}r_E^-(t_1)+(1-\frac{t_2-x_N}{t_2-t_1})r_E^+(t_2)\bigg)K_h\bigg\}
\end{align*}
so that we have
\begin{align*}
&|E\cap \{x:t_1<x_N<t_2\}|\\
&\hspace{0.2in}\leq \int_{t_1}^{t_2} \bigg(\frac{t_2-t}{t_2-t_1}r_E^-(t_1)+(1-\frac{t_2-t}{t_2-t_1})r_E^+(t_2)\bigg)^{N-1}|K_h|dt
\end{align*}
On the other hand, the convexity of $K_+(E,\sigma,t_1)$ implies
\begin{align*}
&\bigg\{x:t_1<x_N<\tau,\pi_1(x)\in \bigg(\frac{\tau-x_N}{\tau-t_1}r_E^-(t_1)+(1-\frac{\tau-x_N}{\tau-t_1})r_E^+(t_2)\bigg)K_h\bigg\}\\
&\hspace{0.2in}\subset K_{+}(E,\sigma,t_1)\cap \{x:t_1<x_N<\tau\}
\end{align*}
so that we have
\begin{align*}
&|E\cap \{x:t_1<x_N<t_2\}|=|K_+(E,\sigma,t_1)\cap \{x:t_1<x_N<\tau\}|\\
&\hspace{0.2in}\geq \int_{t_1}^\tau \bigg(\frac{\tau-t}{\tau-t_1}r_E^-(t_1)+(1-\frac{\tau-t}{\tau-t_1})r_E^+(t_2)\bigg)^{N-1}|K_h|dt 
\end{align*}
Combining these bounds and evaluating the integrals we obtain
\begin{align*}
(t_2-t_1)\cdot \frac{r_E^-(t_1)^N-r_E^+(t_2)^N}{N(r_E^-(t_1)-r_E^+(t_2))}\geq (\tau-t_1)\cdot \frac{r_E^-(t_1)^N-r_E^+(t_2)^N}{N(r_E^-(t_1)-r_E^+(t_2))},
\end{align*}
yielding $\tau\leq t_2$.  The strict inequality $\tau<t_2$ then follows by observing that $(\ref{hyp1})$ implies that for a.e. $t\in (t_1,t_2)$, the integrand in
\begin{align*}
\int_{t_1}^{t_2} \bigg(\frac{t_2-t}{t_2-t_1}r_E^-(t_1)+\Big(1-\frac{t_2-t}{t_2-t_1}\Big)r_E^+(t_2)\bigg)^{N-1}|K_h|\bigg)-r_E(t)dt.
\end{align*}
is a non-negative function.  If $\tau=t_2$, the calculations above show that this integral would be zero, so that the integrand would be zero almost everywhere, contradicting the strict inequality ($\ref{hyp1}$).

Returning to the proof of the lemma, we construct $E'\subset\mathbb{R}^N$ as
\begin{align*}
E'&=\bigg(E\cap \{x:x_N\leq t_1\}\bigg)\\
&\hspace{0.2in}\cup \bigg(K_+(E,\sigma,t_1)\cap \{x:t_1<x_N<\tau\}\bigg)\\
&\hspace{0.2in}\cup \bigg(-(t_2-\tau)e_N+(E\cap \{x:x_N\geq t_2\})\bigg).
\end{align*}

Our goal is now to compare the values of the surface and potential energy terms $\mathscr{F}_s$ and $\mathscr{F}_p$ at $E$ and $E'$.  Note that $t_1>0$ implies that the contact energies $\mathscr{F}_c$ are equal: $\mathcal{H}^{N-1}(\partial^*E\cap \{x:x_N=0\})=\mathcal{H}^{N-1}(\partial^*(E')\cap \{x:x_N=0\})$.  

For the surface energy term, we decompose the integrals into the bottom, middle and top sections and apply the first claim of Lemma \ref{lem_compare_s} to obtain
\begin{align*}
\mathscr{F}_s(E)&\geq \mathscr{F}_s(E'),
\end{align*}
where we have observed that the constructions of $K_+(E,\sigma,t_1)$, $\tau$, and $E'$ imply $\mathcal{H}^{N-1}((\partial^*E')\cap \{x:x_N=t_1\})=\mathcal{H}^{N-1}((\partial^*E')\cap \{x:x_N=t_2\})=0$.  

Turning to the potential energy term, we again decompose the integrals to obtain
\begin{align}
\nonumber \mathscr{F}_p(E)&=\int_{E'\cap \{x:x_N\leq t_1\}} x_N dx+\int_{E\cap \{x:t_1<x_N<t_2\}} x_N dx\\
&\hspace{0.2in}+\int_{E'\cap \{x:x_N\geq \tau\}} (x_N+t_2-\tau)dx.\label{eqab}
\end{align}
The choices of $\sigma$ and $\tau$ now imply the equality
\begin{align*}
&|E\cap \{x:\tau<x_N<t_2\}|=|(E'\setminus E)\cap \{x:t_1<x_N<\tau\}|,
\end{align*}
giving
\begin{align*}
&\int_{E\cap \{x:t_1<x_N<t_2\}} x_Ndx\\
&\hspace{0.4in}>\int_{E\cap \{x:t_1<x_N<\tau\}} x_Ndx+\tau|(E'\setminus E)\cap \{x:t_1<x_N<\tau\}|\\
&\hspace{0.4in}\geq \int_{E\cap \{x:t_1<x_N<\tau\}} x_Ndx+\int_{(E'\setminus E)\cap \{x:t_1<x_N<\tau\}} x_Ndx\\
&\hspace{0.4in}=\int_{E'\cap \{x:t_1<x_N<\tau\}} x_Ndx.
\end{align*}

To conclude, we remark that the strict inequality follows from the strict positivity of $|E\cap \{x:\tau\leq x_N<t_2\}|$.  Noting that $t_2-\tau\geq 0$, we find that the right hand side of $(\ref{eqab})$ is strictly greater than $\mathscr{F}_p(E')$.

Assembling these comparison estimates, we have 
\begin{align*}
\mathscr{F}(E)&>\mathscr{F}(E'),
\end{align*}
which resolves the first case.

\begin{figure}
\centering
\def\svgwidth{0.68\columnwidth}

\begingroup
  \makeatletter
  \providecommand\color[2][]{%
    \errmessage{(Inkscape) Color is used for the text in Inkscape, but the package 'color.sty' is not loaded}
    \renewcommand\color[2][]{}%
  }
  \providecommand\transparent[1]{%
    \errmessage{(Inkscape) Transparency is used (non-zero) for the text in Inkscape, but the package 'transparent.sty' is not loaded}
    \renewcommand\transparent[1]{}%
  }
  \providecommand\rotatebox[2]{#2}
  \ifx\svgwidth\undefined
    \setlength{\unitlength}{841.88974609pt}
  \else
    \setlength{\unitlength}{\svgwidth}
  \fi
  \global\let\svgwidth\undefined
  \makeatother
  \begin{picture}(1,0.70707072)%
    \put(0,0){\includegraphics[width=\unitlength]{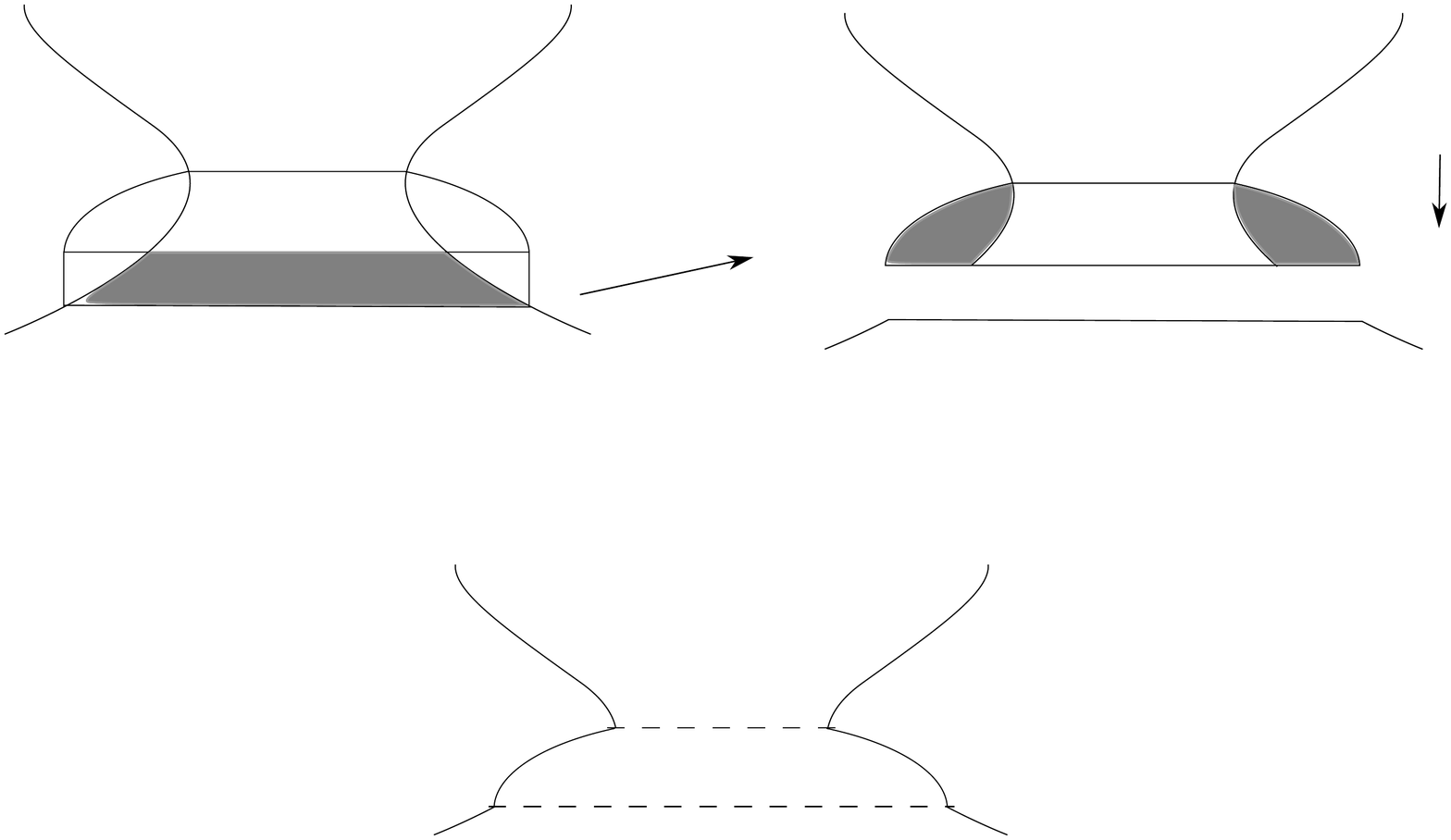}}%
    \put(0.10154881,0.36546997){\color[rgb]{0,0,0}\makebox(0,0)[lb]{\smash{$E$}}}%
    \put(0.01316622,0.4064853){\color[rgb]{0,0,0}\makebox(0,0)[lb]{\smash{$t_1$}}}%
    \put(0.01723869,0.48768805){\color[rgb]{0,0,0}\makebox(0,0)[lb]{\smash{$t_2$}}}%
    \put(0.01452371,0.44262147){\color[rgb]{0,0,0}\makebox(0,0)[lb]{\smash{$\tau_2$}}}%
    \put(0.54013883,0.39969181){\color[rgb]{0,0,0}\makebox(0,0)[lb]{\smash{$t_1$}}}%
    \put(0.54401494,0.49143606){\color[rgb]{0,0,0}\makebox(0,0)[lb]{\smash{$t_2$}}}%
    \put(0.54036823,0.43824921){\color[rgb]{0,0,0}\makebox(0,0)[lb]{\smash{$\tau_2$}}}%
    \put(0.38398158,0.04194238){\color[rgb]{0,0,0}\makebox(0,0)[lb]{\smash{$E''$}}}%
    \put(0.09982595,0.08315287){\color[rgb]{0,0,0}\makebox(0,0)[lb]{\smash{$t_1$}}}%
    \put(0.09833272,0.13337993){\color[rgb]{0,0,0}\makebox(0,0)[lb]{\smash{$t_2-(\tau_2-t_1)$}}}%
  \end{picture}%
\endgroup
\caption{The construction used to establish Case $2$ in the proof of Lemma $\ref{lem_competitor}$.  The section $E\cap \{(x',t):t_1<t<t_2\}$ of $E$ is replaced with $K_-(E,\sigma_2,t_2)\cap \{(x',t):\tau_2<t<t_2\}$.  The upper portion of this new set is then translated downward to rest on top of the bottom portion of the original set $E$.  In contrast to Case $1$, not all mass is moved downward; this effect is compensated by the hypothesis $(\ref{hyp2})$, which ensures that the decrease in potential energy from the translation in the last step of the process outweighs any possible increase in potential from mass initially moved upward.\label{fig-competitor-2}}
\end{figure}

\vspace{0.2in}

\underline{Case 2}: $r_E^-(t_1)>r_E^+(t_2)$.

\vspace{0.2in}

The argument for this case resembles that of Case 1, with a slightly altered competitor.  In order to show that the functional $\mathscr{F}$ decreases, we therefore make use of the hypothesis ($\ref{hyp2}$).  We let $\sigma=\sigma_2$ and $\tau=\tau_2$ be as constructed in Lemma \ref{lem_construction}, and begin by noting that $\tau>t_1$.  Indeed, this follows from a similar argument as before: by (\ref{hyp1}) and the convexity of $K_-(E,\sigma,t_2)$, we have
\begin{align*}
(t_2-t_1)\cdot \frac{r_E^-(t_1)^N-r_E^+(t_2)^N}{N(r_E^-(t_1)-r_E^+(t_2))}\geq (t_2-\tau)\cdot \frac{r_E^-(t_1)^N-r_E^+(t_2)^N}{N(r_E^-(t_1)-r_E^+(t_2))},
\end{align*}
which gives $t_1\leq \tau$.  The strict inequality $t_1<\tau$ then follows from the strictness of $(\ref{hyp1})$ as in the previous case.

We now define $E''\subset\mathbb{R}^N$ by
\begin{align*}
E''&=\bigg(E\cap \{x:x_N\leq t_1\}\bigg)\\
&\hspace{0.2in}\cup \bigg( -(\tau-t_1)e_N+K_-(E,\sigma,t_2)\cap \{x:\tau<x_N<t_2\}\bigg)\\
&\hspace{0.2in}\cup \bigg( -(\tau-t_1)e_N+E\cap \{x:x_N\geq t_2\} \bigg).
\end{align*}

As before, we note that $t_1>0$ gives $\mathcal{H}^{N-1}(\partial^*E\cap \{x:x_N=0\})=\mathcal{H}^{N-1}(\partial^*(E'')\cap \{x:x_N=0\})$.  Similarly, the argument given in Case 1 with the first claim of Lemma $\ref{lem_compare_s}$ replaced by the second claim of the same lemma gives $\mathscr{F}_s(E)\geq \mathscr{F}_s(E'')$.  

To conclude the proof, it remains to compare $\mathscr{F}_p(E'')$ with $\mathscr{F}_p(E)$.  For this, we again decompose the integrals:
\begin{align}
\nonumber \mathscr{F}_p(E)& =\int_{E''\cap \{x:x_N\leq t_1\}} x_Ndx+\int_{E\cap \{x:t_1<x_N<t_2\}}x_Ndx\\
\nonumber &\hspace{0.2in}+\int_{E''\cap \{x:x_N\geq t_2-(\tau-t_1)\}} (x_N+\tau-t_1)dx\\
\nonumber &=\int_{E''\cap \{x:x_N\leq t_1\}} x_Ndx+\int_{E\cap \{x:t_1<x_N<t_2\}}x_Ndx\\
&\hspace{0.2in}+\int_{E''\cap \{x:x_N\geq t_2-(\tau-t_1)\}} x_Ndx+(\tau-t_1)|E\cap \{x:x_N\geq t_2\}|\label{eqbca}
\end{align}
where we have noted that $|E''\cap \{x:x_N\geq t_2-(\tau-t_1)\}|=|E\cap \{x:x_N\geq t_2\}|$.  

To estimate the second term, we decompose the integral and use the change of variables $x\mapsto x+(\tau-t_1)e_N$ to write
\begin{align}
\nonumber \int_{E\cap \{x:t_1<x_N<t_2\}}x_Ndx&> \int_{(E-(\tau-t_1)e_N)\cap I} x_Ndx\\
\nonumber &\hspace{0.4in}+(\tau-t_1)|(E-(\tau-t_1)e_N)\cap I\}|\\
&\hspace{0.4in}+t_1|E\cap \{x:t_1<x_N\leq \tau\}|,\label{eqabc}
\end{align}
where $I=\{x:t_1<x_N<t_2-(\tau-t_1)\}$, and where we have used the strict positivity of $|E\cap \{x:t_1<x_N\leq \tau\}|$ to obtain the strict inequality.  To handle the first term on the right-hand side of ($\ref{eqabc}$), note that the construction of $E''$ along with ($\ref{hyp1}$) and the convexity of $K_-(E,\sigma,t_2)$ give
\begin{align}
(E-(\tau-t_1)e_N)\cap I\subset E''\cap I,\label{inclusion}
\end{align}
and thus
\begin{align}
\nonumber &\int_{(E-(\tau-t_1)e_N)\cap I} x_Ndx\\
\nonumber &\hspace{0.4in}=\int_{E''\cap I} x_Ndx-\int_{(E''\setminus (E-(\tau-t_1)e_N))\cap I} x_Ndx\\
&\hspace{0.4in}\geq \int_{E''\cap I}x_Ndx-(t_2-(\tau-t_1))|(E''\setminus (E-(\tau-t_1)e_N))\cap I\}|.\label{eqbc}
\end{align}

Combining ($\ref{eqabc}$)--($\ref{eqbc}$) with ($\ref{inclusion}$) then yields the bound
\begin{align}
\nonumber &\int_{E\cap \{x:t_1<x_N<t_2\}} x_Ndx\\
\nonumber &\hspace{0.4in}> \int_{E''\cap I} x_Ndx-t_2|(E''\setminus (E-(\tau+t_1)e_N))\cap I|\\
&\hspace{0.7in}+(\tau-t_1)|E''\cap I|+t_1|E\cap \{x:t_1<x_N\leq \tau\}|.\label{equality0}
\end{align}

On the other hand, by ($\ref{inclusion}$) and the construction of $E''$ we have
\begin{align*}
|E\cap \{x:t_1<x_N\leq \tau\}|
&=|(E''\setminus (E-(\tau-t_1)e_N))\cap I|,
\end{align*}
which, substituted into ($\ref{equality0}$), gives
\begin{align}
\nonumber \int_{E\cap \{x:t_1<x_N<t_2)\}}&>\int_{E''\cap I}x_Ndx+(\tau-t_1)|E''\cap I|\\
&\hspace{0.2in}-(t_2-t_1)|E\cap \{x:t_1<x_N\leq \tau\}|.\label{eqbcd}
\end{align}

Assembling these estimates, we have 
\begin{align*}
\mathscr{F}_p(E)&>\mathscr{F}_p(E'')+(\tau-t_1)|E''\cap I|-(t_2-t_1)|E\cap \{x:t_1<x_N\leq \tau\}|\\
&\hspace{1.2in} +(\tau-t_1)|E\cap \{x:x_N\geq t_2\}|.
\end{align*}
The positivity of the second term and the hypothesis ($\ref{hyp2}$) along with the bound 
\begin{align*}
|E\cap \{x:t_1<x_N\leq \tau\}|\leq (\tau-t_1)\sup_{t>0} v_E(t)
\end{align*}
then allow us to obtain $\mathscr{F}_p(E)>\mathscr{F}_p(E'')$, which in turn yields $\mathscr{F}(E)>\mathscr{F}(E'')$ as desired.
\end{proof}

\subsection{Regularity results}

We now use Lemma $\ref{lem_competitor}$ to establish regularity properties of minimizers for $\mathscr{F}$.  With the goal of showing the result on convexity of symmetric minimizers, Theorem $\ref{thm_concave}$, we begin by showing that if $E$ is a symmetric minimizer, the function $t\mapsto v_E(t)=|E_t|\in BV(\mathbb{R})$ is continuous (and thus the function $r_E$ is continuous as well).
\begin{proposition}
\label{prop_cty}
Fix $m>0$, $\omega\in (-f(e_N),f(-e_N))$ and let $f$ be given admissible with $f(x)=\phi(h(x'),x_N)$.  Let $E\in \mathcal{F}_m$ be given such that $E=E^*$.  If $E$ is a minimizer for $\mathscr{F}$, then $v_E^+(t)=v_E^-(t)$ for every $t\geq 0$.
\end{proposition}

\begin{proof}
Let $t_0\geq 0$ be given.  We begin by showing that $v_E^+(t_0)\leq v_E^-(t_0)$.  Suppose for contradiction that this fails; we would like to apply Lemma $\ref{lem_competitor}$ to contradict the minimality of $E$.  We must therefore obtain $t_1<t_0$ such that ($\ref{hyp1}$) holds for a.e. $s\in (0,1)$.  Note that the left continuity of $v_E^-$ implies we may choose $\delta>0$ such that
\begin{align*}
v_E^-(t)<\frac{v_E^-(t_0)+v_E^+(t_0)}{2}
\end{align*}
for every $t\in (t_0-\delta,t_0)$.  Set $\ell(t)=\frac{2}{\delta}(r_E^+(t_0)-r_E^-(t_0-\frac{\delta}{2}))(t-t_0)+r_E^+(t_0)$ so that
\begin{align*}
\bigg\{(t,\ell(t)):t\in [t_0-\frac{\delta}{2},t_0]\bigg\}
\end{align*} 
is the line connecting $(t_0-\frac{\delta}{2},r_E^-(t_0-\frac{\delta}{2}))$ and $(t_0,r_E^+(t_0))$.  We then set
\begin{align*}
t'=\sup \bigg\{t\in [t_0-\frac{\delta}{2},t_0]:r_E^-(t)\geq \ell(t)\bigg\}.
\end{align*}
The left continuity of $v_E^-$ (and hence $r_E^-$) then implies that $r_E^-(t')=\ell(t')$.  Noting that $\ell(t)>\bigg(\frac{v_E^-(t_0)+v_E^+(t_0)}{2}\bigg)^\frac{1}{N-1}$ for $t$ sufficiently close to $t_0$, we obtain $t'<t_0$.  Moreover, the choice of $t'$ implies that $r_E^-(t)<\ell(t)$ for all $t\in (t',t_0)$.  Thus, the hypothesis ($\ref{hyp1}$) of Lemma $\ref{lem_construction}$ holds with $t_1=t'$ and $t_2=t_0$.  Applying Lemma $\ref{lem_construction}$ then shows that $E$ is not a minimizer for $\mathscr{F}$, a contradiction.  Thus, $v_E^+(t_0)\leq v_E^-(t_0)$.

Suppose now that $v_E^+(t_0)<v_E^-(t_0)$.  We would like to again apply Lemma $\ref{lem_competitor}$, this time obtaining $t_2>t_0$ such that the hypotheses $(\ref{hyp1})$ and $(\ref{hyp2})$ hold.  We proceed as before, adjusting the argument to account for the additional hypothesis $(\ref{hyp2})$.  More precisely, we use the right continuity of $v_E^+$ to choose $\delta>0$ such that
\begin{align*}
v_E^+(t)<\frac{v_E^-(t_0)+v_E^+(t_0)}{2}
\end{align*}
for every $t\in (t_0,t_0+\delta)$.  Set $\delta_0=\min\{\frac{\delta}{2},\frac{|E\cap \{x:x_N\geq t_0+\delta\}|}{\lVert v_E(t)\rVert_{L^\infty}}\}$.  Let $\ell(t)=\frac{1}{\delta_0}(r_E^+(t_0+\delta_0)-r_E^-(t_0))(t-t_0)+r_E^-(t_0)$, so that $\{ (t,\ell(t)):t\in [t_0,t_0+\frac{\delta}{2}]\}$ is the line connecting $(t_0,r_E^-(t_0))$ and $(t_0+\frac{\delta}{2},r_E^+(t_0+\frac{\delta}{2}))$.  We then set
\begin{align*}
t'=\inf \bigg\{t\in [t_0,t_0+\delta_0]:r_E^+(t)\geq \ell(t)\bigg\},
\end{align*}
and obtain $r_E^+(t')=\ell(t')$ and $t'>t_0$ as before.  Moreover, the choice of $t'$ implies $r_E^+(t)<\ell(t)$ for all $t\in (t_0,t')$, while the choice of $\delta_0$ ensures 
\begin{align*}
t'-t_0<\delta_0<\frac{|E\cap \{x:x_N\geq t_0+\delta\}|}{\lVert v_E(t)\rVert_{L^\infty}}\leq \frac{|E\cap \{x:x_N\geq t'\}|}{\lVert v_E(t)\rVert_{L^\infty}}.
\end{align*}
Thus, the hypotheses of Lemma $\ref{lem_competitor}$ hold with $t_1=t_0$ and $t_2=t'$, and we obtain a contradiction with the assumption that $E$ is a minimizer for $\mathscr{F}$.  Thus, $v_E^+(t_0)=v_E^-(t_0)$.
\end{proof}

With this continuity in hand, we now show Theorem $\ref{thm_concave}$, which states that $r_E$ is concave on its support and that minimizers consist of a single connected component (or, equivalently, that $E$ is convex; recall that concavity of $r_E$ corresponds to convexity of $E$).  The concavity result is a consequence of a simple fact concerning one-dimensional continuous functions, which we give in the appendix as Lemma $\ref{lem_concave}$.
\begin{proof}[Proof of Theorem $\ref{thm_concave}$.]
Suppose that $r_E$ is not concave on $\{t:v_E(t)>0\}$.  Then we can find $0<t_1<t_2$ and $s_0\in (0,1)$ such that
\begin{align}
r_E(s_0t_1+(1-s_0)t_2)<s_0r_E(t_1)+(1-s_0)r_E(t_2).\label{eqs0}
\end{align}
Setting $\epsilon=\frac{|E\cap \{x:x_N\geq t_2\}|}{\lVert v_E(t)\rVert_{L^\infty}}$, an application of Lemma $\ref{lem_concave}$ gives $t'_1,t'_2$ for which the hypotheses ($\ref{hyp1})$ and $(\ref{hyp2})$ of Lemma $\ref{lem_competitor}$ are satisfied.  Invoking Lemma $\ref{lem_competitor}$, we obtain a contradiction with the assumption that $E$ is a minimizer for $\mathscr{F}$.

To obtain the existence of $T_{\max}$, note that the positivity of the potential energy term $\mathscr{F}_p$ of $\mathscr{F}$ implies that there exists $T_1>0$ such that $v_E(t)>0$ for every $t\in (0,T_1)$.  It therefore suffices to show that $E$ has exactly one connected component.  To see this, note that if $E$ has more than one connected component, the concavity of $r_E$ on its support implies that there exist $0<T_2<T_3$ such that $r_E=0$ on $(T_2,T_3)$ but $E\cap \{x:x_N>T_3\}\neq\emptyset$.  A simple comparison argument then shows that $E$ cannot be a minimizer (one constructs a competitor which moves $E\cap \{x:x_N>T_3\}$ downward, reducing $\mathscr{F}_p(E)$).
\end{proof}

\section{General minimizers: centering and characterization.}

We now turn our attention to the properties of minimizers of $\mathscr{F}$.  In contrast to the previous section, we make no a priori assumption of symmetry.  Indeed, our first goal is to show that minimizers of $\mathscr{F}$ are symmetric.  In this direction, our approach is inspired by the study of the barycenter introduced in \cite{BCF}, and our proofs are closely related to the techniques developed there.  

\subsection{Symmetry of minimizers: centering.}
The first step towards obtaining the symmetry of minimizers is the following remark, which states that minimizers for $\mathscr{F}$ have essentially no vertical normals at interior heights.  The argument is based on combining the minimality of $E$ with the symmetrization results of Section $4$ and the fact that the result holds for symmetric minimizers as a consequence of the convexity result Theorem $\ref{thm_concave}$.
\begin{rem}
\label{rem61}
Any minimizer $E$ for $\mathscr{F}$ satisfies
\begin{align}
\mathcal{H}^{N-1}(\{x\in \partial^*E:\nu_E(x)=\pm e_N\}\cap J)=0,\label{clmob1}
\end{align}
where $J=\{x:0<x_N<T_{\max}\}$ and $T_{\max}$ is as in Theorem $\ref{thm_concave}$.

To obtain this, let $E$ be a given minimizer.  Theorem $\ref{thm_concave}$ then implies that $(\ref{clmob1})$ holds with $E$ replaced by $E^*$.  An application of Lemmas $\ref{lem_coarea}$ and $\ref{lem_symm}$ therefore yields
\begin{align}
\nonumber &\int_{(\partial^*E^*)\cap J} f(\nu_{E^*}(x))d\mathcal{H}^{N-1}(x)\\
\nonumber &\hspace{0.2in}\leq \int_{(\partial^*E^*)\cap \{x:\nu_{E^*}(x)\neq \pm e_N\}\cap J} f(\nu_{E^*}(x))d\mathcal{H}^{N-1}(x)\\
&\hspace{0.2in} \leq \int_{(\partial^*E)\cap \{x:\nu_E(x)\neq \pm e_N\}\cap J} f(\nu_E(x))d\mathcal{H}^{N-1}(x).\label{symm_eq}
\end{align}
Moreover, by the definition of the symmetrization map $E\mapsto E^*$ we have $\mathscr{F}_p(E^*)=\mathscr{F}_p(E)$, as well as
\begin{align*}
\mathcal{H}^{N-1}(\partial^*E^*\cap \{x:x_N=0\})&=\mathcal{H}^{N-1}(\partial^*E\cap \{x:x_N=0\}),\\
\mathcal{H}^{N-1}(\partial^*E^*\cap \{x:x_N=T_{\max}\})&=\mathcal{H}^{N-1}(\partial^*E\cap \{x:x_N=T_{\max}\}).
\end{align*}
Combining these equalities with the minimality of $E$ and Theorem $\ref{thm_symm_f}$, we obtain
\begin{align*}
\mathscr{F}(E)&=\mathscr{F}(E^*)=\omega \mathcal{H}^{N-1}(\partial^*E^*\cap \{x:x_N=0\})+\int_{(\partial^*E^*)\cap J} f(\nu_E(x))d\mathcal{H}^{N-1}(x)\\
&\hspace{0.2in}+f(e_N)\mathcal{H}^{N-1}(\partial^*E^*\cap \{x:x_N=T_{\max}\})+\mathscr{F}_p(E^*)\\
&\leq \omega\mathcal{H}^{N-1}(\partial^*E^*\cap \{x:x_N=0\})+\int_{(\partial^*E)\cap \{x:x_N\neq \pm e_N\}\cap J} f(\nu_E(x))d\mathcal{H}^{N-1}(x)\\
&\hspace{0.2in}+f(e_N)\mathcal{H}^{N-1}(\partial^*E\cap \{x:x_N=T_{\max}\})+\mathscr{F}_p(E)\\
&\leq \mathscr{F}(E).
\end{align*}
This in turn implies
\begin{align*}
\int_{(\partial^*E)\cap \{x:\nu_E(x)=\pm e_N\}\cap J}f(\nu_E(x))d\mathcal{H}^{N-1}(x)&=0,
\end{align*}
which gives ($\ref{clmob1}$).
\end{rem}

\vspace{0.2in}

For any minimizer $E$, we now define the barycenter function $\beta:\mathbb{R}\rightarrow\mathbb{R}^{N-1}$ for the slices $E_t$ by
\begin{align*}
\beta(t)&:=\frac{1}{v_E(t)}\int_{E_t} x'dx'-\frac{1}{v_E(t)}\int_{r_E(t)K_h} x'dx'.
\end{align*}
Observe that $\beta$ is well defined as a consequence of ($\ref{rbound}$) in the proof of Theorem $\ref{thm_exist}$.

As a first step in our analysis of $\beta$, we will establish the symmetry of individual slices of minimizers, up to translation.  For this purpose, we will need a density estimate for almost-minimizers of the surface energy 
\begin{align*}
F(E):=\int_{\partial^*E} f(\nu_E(x))d\mathcal{H}^{N-1}(x).
\end{align*}
For details, see ($3.3$) in \cite{FM}.
\begin{lem}[Density estimate for almost minimizers, \cite{FM}]
\label{lem_density}
Let $E\in \mathcal{F}$ be given, and fix $\epsilon,R>0$, $t\in \mathbb{R}\cap \{-\infty\}$.  Suppose that $E$ is an {\it $(\epsilon,R,t)$-minimizer} for $F$ in the sense that for every $E'\in\mathcal{F}$ satisfying $|E'|=|E|$, $E\Delta E'\subset \{x:x_N>t\}$, and $E'\subset \{x:\inf_{y\in E} f_*(x-y)<R\}$, one has
\begin{align*}
F(E)&\leq F(E')+\epsilon|K|^{1/N}|E|^{-1/N}|E\Delta E'|
\end{align*}
where
\begin{align*}
f_*(x)=\sup\{x\cdot y:f(y)=1\}.
\end{align*}

Then there exists $C>0$ and $r_0>0$ such that for every $x\in \partial E\cap \{x:x_N>t\}$ and every $0<r<r_0$, one has
\begin{align*}
|B(x,r)\cap E|\geq Cr^N.
\end{align*}
\end{lem}
We remark that although the proof of ($3.3$) in \cite{FM} concerns $(\epsilon,R)$-minimizers ($(\epsilon,R,-\infty)$-minimizers in our terminology), the same argument gives our statement.  With this tool in hand, we obtain
\begin{lem}[Symmetry of slices of minimizers up to translation]
\label{lem_symm1}
Fix $m>0$, $\omega\in (-f(e_N),f(-e_N))$ and let $f$ be given admissible with $f(x)=\phi(h(x'),x_N)$.  Let $E$ be a minimizer for $\mathscr{F}$.  Then for almost every $t>0$ we have
\begin{align}
E_t&=r_E(t)K_h+\beta(t).\label{clmob2}
\end{align}
Moreover, $\beta$ is locally bounded.
\end{lem}
\begin{proof}
Let $E$ be a given minimizer of $\mathscr{F}$.  Observe that by the argument given in Remark $\ref{rem61}$, we have $\mathscr{F}_s(E)=\mathscr{F}_s(E^*)$.  The claim ($\ref{clmob2}$) then follows immediately from Remark $\ref{rem61}$, Lemma $\ref{lem_coarea}$ and Lemma $\ref{lem_symm}$.

It remains to show that $\beta$ is locally bounded.  Let $I\subset [0,\infty)$ be a given compact interval.  We claim that there exists $r=r(E,I)>0$ such that
\begin{align*}
E\cap \{x:x_N\in I\}\subset \{x:x_N\in I,|\pi_1(x)|<r\}.
\end{align*}

We proceed in two steps.

\vspace{0.2in}

\underline{Step $1$}: There exist $\epsilon>0$ and $R>0$ such that $E$ is an $(\epsilon,R,t)$-minimizer for every $t>0$.

\vspace{0.2in}

The argument we give is inspired by and closely related to the proof of Corollary $4.4$ in \cite{FM}.  Recall that by ($\ref{rbound}$) in the proof of Theorem $\ref{thm_exist}$, we may find $R>0$ such that
\begin{align*}
\pi_1(E^*)\subset \{x:|x'|\leq R\}.
\end{align*}
The convexity of $E^*$ and the volume constraint $|E|=|E^*|=m$ then imply that there exists $T>0$ such that
\begin{align*}
\pi_N(E)=\pi_N(E^*)\subset [0,T].
\end{align*}
Indeed, since there exists $\epsilon>0$ such that $B(0,\epsilon)\subset\mathbb{R}^{N-1}$ is contained in $\partial^*E^*\cap \{x_N=0\}=v_{E}^+(0)K_h$, supposing $v_E(\tau)>0$ for some $\tau>0$, the convexity of $E^*$ implies that $E^*$ contains the cone $P(\epsilon,\tau)=\cup_{0\leq s\leq \tau} B(0,\epsilon-s(\epsilon/\tau))\times \{s\}$.  We then obtain
$|E^*|\geq |P(\epsilon,\tau)|=(\tau/N)\epsilon^{N-1}\rightarrow \infty$ as $\tau\rightarrow\infty$, which contradicts $|E^*|=m$ for $\tau$ sufficiently large.

Fix $R=1$, and let $t>0$ be given.  Let $E'\in\mathcal{F}$ be given such that $|E'|=|E|$, $E\Delta E'\subset \{x:x_N>t\}$ and $E'\subset \{x:\inf_{y\in E}f_*(x-y)<R\}$.  Then, using the minimality of $E$ and the condition $E\Delta E'\subset \{x:x_N>t\}$, we obtain
\begin{align*}
F(E)&=\mathscr{F}_s(E)+f(-e_N)\mathcal{H}^{N-1}(\partial^*E\cap \{x_N=0\})\\
&\leq F(E')+\mathscr{F}_p(E')-\mathscr{F}_p(E)\\
&\hspace{0.2in}+(f(-e_N)-\omega)(\mathcal{H}^{N-1}(\partial^*E\cap \{x_N=0\})-\mathcal{H}^{N-1}(\partial^*E'\cap \{x_N=0\}))\\
&=F(E')+\mathscr{F}_p(E')-\mathscr{F}_p(E).
\end{align*}
Now, observing that $|E'|=|E|$ implies $|E'\setminus E|=\frac{|E'\Delta E|}{2}$, we obtain
\begin{align*}
\mathscr{F}_p(E')-\mathscr{F}_p(E)&\leq \int_{E'\setminus E} x_Ndx\\
&\leq T|E'\setminus E|=\frac{T}{2}|E'\Delta E|=\epsilon|K|^{1/N}|E|^{-1/N}|E'\Delta E|
\end{align*}
with $\epsilon:=\left(\frac{Tm^{1/N}}{2|K|^{1/N}}\right)$.  Since $E'$ was arbitrary, we conclude that $E$ is an $(\epsilon,R,t)$ minimizer as desired.

\vspace{0.2in}

\underline{Step $2$}: Local bound for $\beta$.

\vspace{0.2in}

Let $0<T_1<T_2<\infty$ be given.  We show that $\sup_{t\in [T_1,T_2]} |\beta(t)|<\infty$.  Indeed, from the definition of $\beta$, it suffices to find $R'>0$ such that $\pi_1(E\cap \{x:T_1\leq x_N\leq T_2\})\subset \{x:|x'|\leq R'\}$.  We claim that this follows from the density estimate of Lemma $\ref{lem_density}$ (the set $E$ satisfies the hypotheses of this lemma by Step $1$ above).  To obtain this, suppose that there existed a sequence $(x_n)\subset E\cap \{x:T_1\leq x_N\leq T_2\}$ such that $|\pi_1(x_n)|\rightarrow\infty$.  We may then choose a subsequence $(x_{n_k})$ such that $(B(x_{n_k},\frac{r_0}{2}):n\in\mathbb{N})$ is a disjoint sequence of balls in $\mathbb{R}^N$.  We then have
\begin{align*}
m=|E|\geq \sum_{k=1}^\infty |B(x_{n_k},\frac{r_0}{2})\cap E|\geq \sum_{k=1}^\infty Cr_0^N=\infty,
\end{align*}
giving a contradiction.  This completes the proof of Lemma $\ref{lem_symm1}$.
\end{proof}

The local boundedness property of $\beta$ then gives the following lemma.
\begin{lem}
\label{lem_w11}
Fix $m>0$, $\omega\in (-f(e_N),f(-e_N))$ and let $f$ be given admissible with $f(x)=\phi(h(x'),x_N)$.  Suppose $E$ is a minimizer for $\mathscr{F}$.  Then $\beta \in W_{loc}^{1,1}((0,T_{\max});\mathbb{R}^{N-1})$ with
\begin{align}
\nonumber \beta'(t)&=\frac{-v'_E(t)}{v_E(t)^2}\left(\int_{E_t} x'dx'-\int_{r_E(t)K_h} x'dx'\right)\\
\nonumber &\hspace{0.2in}+\frac{1}{v_E(t)}\left(\int_{\partial^*E_t} x'\frac{\pi_N(\nu_E(x',t))}{|\pi_1(\nu_E(x',t))|}d\mathcal{H}^{N-2}(x')\right.\\
\label{eq-beta}&\hspace{1.8in}\left.-\int_{\partial^*E^*_t} x'\frac{\pi_N(\nu_{E^*}(x',t))}{|\pi_1(\nu_{E^*}(x',t))|}d\mathcal{H}^{N-2}(x')\right).
\end{align}
\end{lem}

This result follows from a simple distributional calculation involving Fubini's theorem, integration by parts and the coarea formula, together with standard product and chain rules in Sobolev spaces.  Since the proof is essentially identical to the proof of Theorem $4.3$ of \cite{BCF} after accounting for the additional term in $\beta$ (which in any case has the same form as the first term), we omit the details.

\vspace{0.2in}

Having shown these lemmas, we are now ready to prove the main symmetry result for minimizers, Theorem $\ref{thm_min}$.

\begin{proof}[Proof of Theorem $\ref{thm_min}$]
We set $\alpha=r_E$, and observe that Theorem $\ref{thm_symm_f}$ implies that $E^*$ is a symmetric minimizer for $\mathscr{F}$, so that Theorem $\ref{thm_concave}$ gives the concavity of $r_{E^*}$ on its support.  However, by the definition of the symmetrization, we have $r_E=r_{E^*}$, so that this property also holds for $\alpha=r_E$.

Turning to the remaining claim, it suffices to show that there exists $\beta_0\in\mathbb{R}^{N-1}$ such that $\beta(t)=\beta_0$ for a.e. $t\in [0,T_{\max}]$.  Note that since $E$ is a minimizer, equality holds for the application of Jensen's inequality in Lemma $3.2$ for a.e. $t\in [0,T_{\max}]$, and therefore there  
exist $c_E(t),c_{E^*}(t):[0,T_{\max}]\rightarrow \mathbb{R}$ such that
\begin{align*}
\pi_N(\nu_E(x',t))&=c_E(t)h(\pi_1(\nu_E(x',t))),\\
\pi_N(\nu_{E^*}(x',t))&=c_{E^*}(t)h(\pi_1(\nu_{E^*}(x',t)))
\end{align*}
for a.e. $t\in [0,T_{\max}]$ and  $\mathcal{H}^{N-2}$-a.e. $x'\in \partial^*E_t$
 ($\partial^*E^*_t$ for the second line).

But by Lemma $2.2$, we have
\begin{align*}
v'_E(t)&=\int_{\partial^* E_t} \frac{\pi_N(\nu_E(x',t))}{|\pi_1(\nu_E(x',t))|}d\mathcal{H}^{N-2}(x')\\
&=\int_{\partial^* E_t} \frac{c_E(t)h(\pi_1(\nu_E(x',t)))}{|\pi_1(\nu_E(x',t))|} d\mathcal{H}^{N-2}(x')\\
&=\int_{\partial^* E_t} c_E(t)h\left(\frac{\pi_1(\nu_E(x',t))}{|\pi_1(\nu_E(x',t))|}\right)d\mathcal{H}^{N-2}(x')\\
&=c_E(t)P_h(E^*_t)
\end{align*}
for a.e. $t\in [0,T_{\max}]$, where we have used that $E_t$ is a translate of $E
^*_t$.

Thus, since $v_E(t)=v_{E^*}(t)$ for all $t$, we have
\begin{align*}
c_E(t)=\frac{v'_E(t)}{P_h(E^*_t)}\quad\textrm{and}\quad c_{E^*}(t)=\frac{v'_E(t)}{P_h(E^*_t)}
\end{align*}

Substituting these expressions into the formula from Lemma $\ref{lem_w11}$ and using the change of variables $x'\mapsto x'-\beta(t)$, we get (since $\partial^*E_t=\beta(t)+\partial^*E^*_t
$)
\begin{align*}
\beta'(t)&=\frac{-v_E'(t)}{v_E(t)}\beta(t)+\frac{1}{v_E(t)}\left(\frac{v_E'(t)}{ P_h(E^*_t)}\right)\\
&\hspace{0.4in}\cdot \left(\int_{\partial^*E_t} x'\frac{h(\pi_1(\nu_E(x',t)))}{| \pi_1(\nu_E(x',t))|}d\mathcal{H}^{N-2}(x')\right.\\
&\hspace{1.6in}\left.-\int_{\partial^*E_t} (x'-\beta(t))\frac{h_1(\pi_1(\nu_E(x' ,t)))}{|\pi_1(\nu_E(x',t))|}d\mathcal{H}^{N-2}(x')\right)\\
&=\frac{-v_E'(t)}{v_E(t)}\beta(t)+\frac{1}{v_E(t)}\left(\frac{v_E'(t)}{P_h(E^*_t)}\right)\\
&\hspace{0.4in}\cdot \left(\int_{\partial^*E_t} \beta(t)\frac{h(\pi_1(\nu_E(x',t)))}{|\pi_1(\nu_E(x',t))|}d\mathcal{H}^{N-2}(x')\right)\\
&=\frac{-v_E'(t)}{v_E(t)}\beta(t)+\frac{1}{v_E(t)}\left(\frac{v_E'(t)}{P_h(E^*_t)}\right)\left(\beta(t)P_h(E^*_t)\right)\\
&=0.
\end{align*}
for a.e. $t\in [0,T_{\max}]$, giving the existence of $\beta_0$ as desired.
\end{proof}

\subsection{ODE characterization of minimizers.}

In this section, we address the issue of obtaining a characterization of minimizing profiles by ODE methods.  In particular, we establish Theorem $\ref{thm_young}$, giving the Euler-Lagrange equation for $r_E(t)$ and the analogue of the boundary (contact angle) condition of Young's law, by a first variation argument:
\begin{proof}[Proof of Theorem $\ref{thm_young}$]
Let $E$ be a given minimizer, and note that since $E$ is symmetric by Theorem $\ref{thm_min}$ we may write
\begin{align*}
\frac{\mathscr{F}(E)}{|K_h|}&=\frac{1}{|K_h|}\left[\omega v_E(0)+\int_0^\infty \phi\left(P_h(K_h)r_E(t)^{N-2},-v_E'(t)\right)dt+\int_0^{\infty} tv_E(t)dt\right]\\
&=\omega r_E(0)^{N-1}+\int_0^\infty r_E(t)^{N-2}\phi(\Lambda,-(N-1)r'_E(t))dt\\
&\hspace{0.2in}+\int_0^{\infty} tr_E(t)^{N-1}dt.
\end{align*} 
Recalling the volume constraint $|E|=m$ and invoking the minimality of $E$, we obtain
\begin{align*}
0&=\frac{d}{d\epsilon}\bigg[\omega (r_E(0)+\epsilon \psi(0))^{N-1}\\
&\hspace{0.2in}+\int_0^\infty (r_E(t)+\epsilon \psi(t))^{N-2}\phi\left(\Lambda,-(N-1)(r'_E(t)+\epsilon \psi'(t))\right)dt\\
&\hspace{0.4in}+\int_0^{\infty} t(r_E(t)+\epsilon \psi(t))^{N-1}dt+\lambda\left(\int_0^\infty (r_E(t)+\epsilon\psi(t))^{N-1}dt-m\right)\bigg]\bigg|_{\epsilon=0}\\
&=\omega (N-1)r_E(0)^{N-2}\psi(0)\\
&\hspace{0.2in}+\int_0^\infty (N-2)r_E(t)^{N-3}\psi(t)\phi\left(\Lambda,-(N-1)r'_E(t)\right)dt\\
&\hspace{0.2in}-\int_0^\infty (N-1)r_E(t)^{N-2}\partial_2\phi\left(\Lambda,-(N-1)r'_E(t)\right)\psi'(t)dt\\
&\hspace{0.4in}+\int_0^{\infty} t(N-1)r_E(t)^{N-2}\psi(t)dt+\lambda\int_0^\infty (N-1)r_E(t)^{N-2}\psi(t)dt
\end{align*}
for every $\psi\in C^1([0,\infty))$ having compact support.  Integrating by parts in the third term, one obtains
\begin{align*}
0&=\omega (N-1)r_E(0)^{N-2}\psi(0)\\
&\hspace{0.2in}+\int_0^\infty (N-2)r_E(t)^{N-3}\psi(t)\phi\left(\Lambda,-(N-1)r'_E(t)\right)dt\\
&\hspace{0.2in}+(N-1)r_E(0)^{N-2}\partial_2\phi\left(\Lambda,-(N-1)r'_E(0)\right)\psi(0)\\
&\hspace{0.2in}+\int_0^\infty \frac{d}{dt}\left[(N-1)r_E(t)^{N-2}\partial_2\phi\left(\Lambda,-(N-1)r'_E(t)\right)\right]\psi(t)dt\\
&\hspace{0.4in}+\int_0^{\infty} t(N-1)r_E(t)^{N-2}\psi(t)dt+\lambda\int_0^\infty (N-1)r_E(t)^{N-2}\psi(t)dt
\end{align*}
Since this must hold for every $\psi$, we obtain the Euler-Lagrange equation ($\ref{el1}$) along with the boundary condition ($\ref{youngs}$) as desired.
\end{proof}

\subsection{Uniqueness of minimizers.}

With the ODE characterization in hand, we now investigate uniqueness of minimizers, with the goal of proving Theorem \ref{thm_unique}.  We model our arguments on the isotropic case studied by Finn \cite{Finn1,FinnBook}, where the analysis is based on a one-to-one correspondence between solutions of the Euler-Lagrange equation and solutions of the capillary surface equation.  

We partition our analysis into three steps.  In the first step, we make a suitable change of variables which serves to characterize minimizers as elements a family of solutions to a related ODE indexed by a parameter $v_0\in\mathbb{R}$.  The uniqueness will then follow by showing that the volume constraint and boundary condition identified in Theorem $\ref{thm_young}$ can only be satisfied simultaneously for a single value of $v_0$, which is accomplished in the second and third steps.

\begin{proof}[Proof of Theorem $\ref{thm_unique}$]
Note that when $\omega$ belongs to $(-f(e_N),0)$, the profile of an arbitrary minimizer is a graph; this follows from the boundary condition (\ref{youngs}) of Theorem \ref{thm_young} and the convexity of minimizers.  For the sake of simplicity, we shall therefore restrict ourselves to the case $\omega\in (-f(e_N),0)$.  We will return to the general case at the conclusion of the proof.  

As mentioned above, we proceed in three steps.  The first step consists of rewriting the equation under a suitable change of variables.

\vspace{0.2in}

\underline{Step $1$}: (change of variables)

\vspace{0.2in}

\begin{figure}
\centering
\def\svgwidth{0.82\columnwidth}

\begingroup
  \makeatletter
  \providecommand\color[2][]{%
    \errmessage{(Inkscape) Color is used for the text in Inkscape, but the package 'color.sty' is not loaded}
    \renewcommand\color[2][]{}%
  }
  \providecommand\transparent[1]{%
    \errmessage{(Inkscape) Transparency is used (non-zero) for the text in Inkscape, but the package 'transparent.sty' is not loaded}
    \renewcommand\transparent[1]{}%
  }
  \providecommand\rotatebox[2]{#2}
  \ifx\svgwidth\undefined
    \setlength{\unitlength}{972.28798828pt}
  \else
    \setlength{\unitlength}{\svgwidth}
  \fi
  \global\let\svgwidth\undefined
  \makeatother
  \begin{picture}(1,0.29052297)%
    \put(0,0){\includegraphics[width=\unitlength]{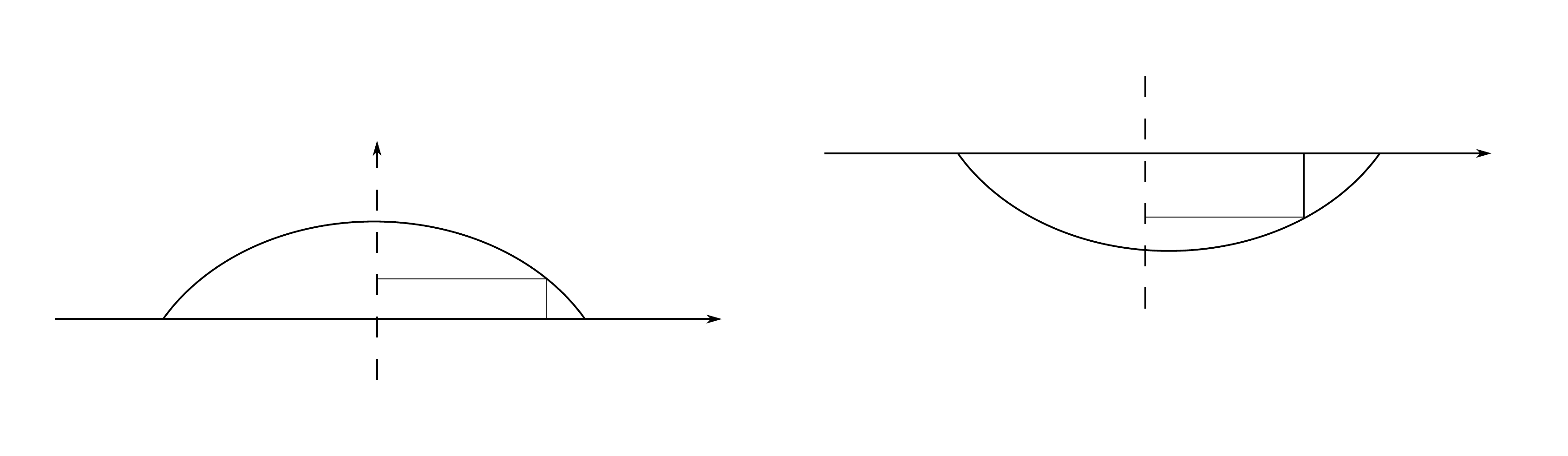}}%
    \put(0.34463071,0.04914262){\color[rgb]{0,0,0}\makebox(0,0)[lb]{\smash{$r_E(t)$}}}%
    \put(0.25203822,0.20082124){\color[rgb]{0,0,0}\makebox(0,0)[lb]{\smash{$t$}}}%
    \put(0.44339775,0.04797195){\color[rgb]{0,0,0}\makebox(0,0)[lb]{\smash{$r$}}}%
    \put(0.16845143,0.10796188){\color[rgb]{0,0,0}\makebox(0,0)[lb]{\smash{$u(r)$}}}%
    \put(0.91925156,0.16150063){\color[rgb]{0,0,0}\makebox(0,0)[lb]{\smash{$r$}}}%
    \put(0.84983672,0.1109469){\color[rgb]{0,0,0}\makebox(0,0)[lb]{\smash{$v(r)$}}}%
  \end{picture}%
\endgroup
\caption{The change of variables used in Step $1$ of the proof of Theorem $\ref{thm_unique}$.  The choice of coordinates is made to ensure that $v'$ is positive and to simplify the equation by removing the scaling factor $\Lambda$ and eliminating the Lagrange multiplier term.\label{fig-ode-cvar}}
\end{figure}
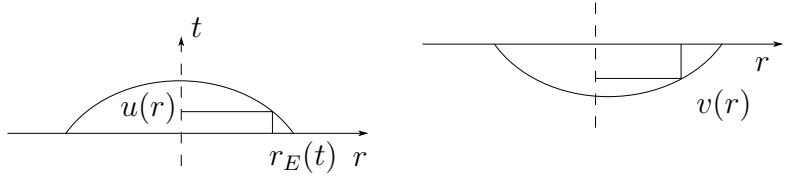

For each $t\in [0,T_{\max}]$, set $r=r_E(t)$ and $u(r)=r_E^{-1}(r)=t$.  This gives
\begin{align}
\nonumber &-\frac{d}{dr}\left[(N-1)r^{N-2}\partial_2\phi\left(\Lambda,-\frac{(N-1)}{u'(r)}\right)\right]\frac{1}{u'(r)}\\
&\hspace{0.2in}=(N-2)r^{N-3}\phi\left(\Lambda,-\frac{(N-1)}{u'(r)}\right)+(N-1)(u+\lambda)r^{N-2}\label{eq_chv1}
\end{align}
from which we obtain the equation 
\begin{align}
-\frac{d}{dr}\left[\Lambda r^{N-2}\partial_1\phi(-\Lambda u'(r),N-1)\right]&=(N-1)(u(r)+\lambda)r^{N-2},
\label{eq_chv2}
\end{align}
with boundary conditions $u'(0)=0$ and $-\partial_2\phi(-\Lambda u'(R_{\max}),N-1)=\omega$ (see the calculations in Appendix $B$).
Defining $v$ by
\begin{align*}
v(r)=-\Lambda^{-1}\left(u\left(\Lambda^2 r\right)+\lambda\right),\quad r\geq 0
\end{align*}
we find that $v$ solves
\begin{align}
\left\lbrace\begin{array}{rl}
\frac{d}{dr}\left[r^{N-2}\partial_1\phi(v',N-1)\right]&=(N-1)r^{N-2}v,\\
v'(0)&=0.
\end{array}\right.
\label{eq_cap}
\end{align}
with the additional boundary condition $-\partial_2\phi(v'(R_{\max}/\Lambda^2),N-1)=\omega$.

Observing that solutions to ($\ref{eq_cap}$) are uniquely characterized by the value of the initial height $v(0)$ (and noting that invoking another change of variables allows us to map solutions of ($\ref{eq_cap}$) back to sets), it will suffice to show that the volume constraint along with the boundary condition determine $v(0)$ uniquely.  

Toward this end, we begin with some basic observations about solutions of ($\ref{eq_cap}$).  Note that integrating ($\ref{eq_cap}$) with respect to $r$ gives
\begin{align}
r^{N-2}\partial_1\phi(v',N-1)=\int_0^r (N-1)\rho^{N-2}v(\rho)d\rho.\label{eq123}
\end{align}
Note also that since $\phi$ is strictly convex and positively $1$-homogeneous with $\partial_1\phi(0,1)=0$, we have
\begin{align}
\partial_1\phi(v',N-1)\quad\left\lbrace\begin{array}{l}<0\quad\textrm{if}\,v'<0,\\=0\quad\textrm{if}\,v'=0,\\>0\quad\textrm{if}\,v'>0.\end{array}\right.\label{eq_phi_ineq}
\end{align}

We now observe that simple continuity arguments allow us to reduce to the case $v(0)>0$.  In particular, any solution $v$ to ($\ref{eq_cap}$) with $v(0)<0$ is decreasing on its domain, and therefore cannot correspond to a minimizing shape.  To see this, let $v_0<0$ be given and let $v$ be the solution to ($\ref{eq_cap}$) with $v(0)=v_0$.  The equality $(\ref{eq123}$) then implies $\partial_1\phi(v'(r),N-1)<0$ for $r>0$ sufficiently small.  Recalling ($\ref{eq_phi_ineq}$), we therefore have $v'(r)<0$ for $r$ sufficiently small.  Suppose for contradiction that $v$ is not decreasing, and set $r_0:=\sup\{r:v'(r)<0\}<\infty$.  The continuity of $v'$ now implies that $v'(r_0)=0$, while ($\ref{eq123}$) gives
\begin{align*}
0=r_0^{N-2}\partial_1\phi(v'(r_0),N-1)<v(0)\int_0^{r_0} (N-1)\rho^{N-2}d\rho=v_0r_0^{N-1},
\end{align*}
which contradicts $v_0<0$.

An identical argument shows that if $v$ solves ($\ref{eq_cap}$) with $v(0)>0$, then $v$ is increasing.  Moreover, $v$ is strictly convex: expanding the derivative on the left side of ($\ref{eq_cap})$ gives the equality
\begin{align}
r^{N-2}\partial^2_{11}\phi(v',N-1)v''=(N-1)r^{N-3}\Delta,\label{eqa123}
\end{align}
where we have set
\begin{align}
\Delta=rv-\frac{N-2}{N-1}\partial_1\phi(v',N-1).\label{Delta}
\end{align}
Since $v$ is increasing, ($\ref{eq123}$) gives $r^{N-2}\partial_1\phi(v',N-1)\leq v(r)r^{N-1}$, which may be rewritten as
\begin{align}
\Delta\geq \frac{\partial_1\phi(v',N-1)}{N-1}>0\quad \forall r>0,\label{Delta_ineq}
\end{align}  
where to conclude the strict positivity we have used ($\ref{eq_phi_ineq}$) along with the fact that $v$ is increasing.  Recalling ($\ref{eqa123}$) and the convexity of $\phi$, we obtain that $v''(r)>0$ for $r>0$, which establishes the desired strict convexity for $v$.

The next step in our argument is to calculate the volume and its derivatives.

\vspace{0.2in}

\underline{Step $2$}: (volume calculations)

\vspace{0.2in}

Fix $v_0>0$ and let $v=v(\cdot;v_0)$ be the solution of ($\ref{eq_cap}$) corresponding to the initial condition $v(0)=v_0$.  To simplify notation, we will often write $v'(r;v_0)=\partial_r v(r;v_0)$.  For each $s\geq 0$, define 
\begin{align*}
V_{v_0}(s)=|\{(x',t):v(|x'|)\leq t\leq v(r(s))\}|,
\end{align*}
where $r(s)\geq 0$ is chosen so that $v'(r(s))=s$ is satisfied (the strict convexity of $v$ established in Step $1$ implies that a unique such value exists).

To compute $V_{v_0}(s)$, we write
\begin{align}
\nonumber V_{v_0}(s)&=\omega_{N-1}r(s)^{N-1}v(r(s))-|S^{N-2}|\int_0^{r(s)} \rho^{N-2}v(\rho)d\rho\\
\nonumber &=\omega_{N-1}\left(r(s)^{N-1}v(r(s))-\int_0^{s} \frac{d}{ds}\left[r(s)^{N-2}\partial_1\phi(s,N-1)\right]ds\right)\\
\label{eqvcalc}&=\omega_{N-1}\left(r(s)^{N-1}v(r(s))-r(s)^{N-2}\partial_1\phi(s,N-1)\right).
\end{align}
where to obtain the second equality we have noted that $|S^{N-2}|=(N-1)\omega_{N-1}$.

The rest of our argument will consist of establishing a suitable monotonicity property for $V_{v_0}(s)$.  Toward this end, we will use the notation $V=V(s)$, $r=r(s)$ and $v=v(r(s))$, as well as $\dot{V}(s)=\frac{d}{dv_0}V_{v_0}[s]$, $\dot{r}(s)=\frac{d}{dv_0}[r(s)]$ and $\dot{v}(s)=\frac{d}{dv_0}[v(r(s);v_0)]$.  In particular, we have
\begin{align}
\dot{V}&=\omega_{N-1}r^{N-3}((N-1)\dot{r}\Delta+r^2\dot{v})\label{dotV}
\intertext{and}
\partial_s \dot{V}&=\frac{\omega_{N-1}r^{N-1}s\partial^2_{11}\phi(s,N-1)}{(N-1)\Delta^2}((N\Delta-rv)\dot{r}-r^2\dot{v}),\label{dotVs}
\end{align}
for every $v_0>0$ and $s\geq 0$, where $\Delta$ is as in ($\ref{Delta}$).

To obtain ($\ref{dotVs}$), one first differentiates $V_{v_0}(s)$ with respect to $s$, obtaining 
\begin{align}
\nonumber \partial_s V_{v_0}(s)&=\omega_{N-1}\bigg((N-1)r^{N-2}(\partial_s r)v+r^{N-1}(\partial_s v)\\
\nonumber &\hspace{1.2in}-(N-2)r^{N-3}(\partial_s r)\partial_1\phi(s,N-1)\\
\nonumber &\hspace{1.2in}-r^{N-2}\partial^2_{11}\phi(s,N-1)\bigg)\\
\nonumber &=\omega_{N-1}r^{N-1}(\partial_s v)
\end{align}
where to obtain the second equality, we have observed that after setting $r=r(s)$ and recalling that by the definition of $r(s)$, $\partial_1\phi(v'(r(s)),N-1)=\partial_1\phi(s,N-1)$, differentiating ($\ref{eq123}$) with respect to $s$ gives the identity
\begin{align}
\partial_s r&=\frac{r\partial^2_{11}\phi(s,N-1)}{(N-1)\Delta}.\label{eq_rs}
\end{align}
Now, noting that the definition of $r(s)$ implies $\partial_s[v(r(s))]=s\partial_s r$ and using the identity $(\ref{eq_rs})$ once again, we obtain
\begin{align}
\partial_s V_{v_0}(s)&=\frac{\omega_{N-1}r^{N}s\partial^2_{11}\phi(s,N-1)}{(N-1)\Delta},\label{eqabcd1}
\end{align}
Differentiating ($\ref{eqabcd1}$) with respect to $v_0$ now gives ($\ref{dotVs}$) as desired.

\vspace{0.2in}

\underline{Step $3$}: (conclusion of the argument)

\vspace{0.2in}

Choosing $s_*>0$ as the unique value for which $-\partial_2\phi(s_*,N-1)=\omega$, and noting that
\begin{align*}
|E|=P(K_h)\int_0^{R_{\max}} r^{N-2}u(r)dr=\frac{P(K_h)\Lambda^{2N-1}}{|S^{N-2}|}V_{v_0}(s_*),
\end{align*}
it will be enough to show that there exists a unique value $v_0>0$ such that $V_{v_0}(s_*)=|S^{N-2}|\,|E|/(P(K_h)\Lambda^{2N-1})$.  To accomplish this, we will show that for all $v_0>0$,
\begin{align*}
\frac{d}{dv_0}[V_{v_0}(s_*)]<0.
\end{align*}

Noting that $\dot{V}(0)=0$ for any $v_0>0$, it will be enough to show 
\begin{align}
\partial_s\dot{V}<0\label{goal1}
\end{align}
for all $s>0$ and $v_0>0$.  In view of ($\ref{dotVs}$), let $v_0>0$ be given and observe that the convexity of $\phi$ gives
\begin{align*}
\frac{\omega_{N-1}r^{N-1}s\partial^2_{11}\phi(s,N-1)}{(N-1)\Delta^2}>0,
\end{align*}
for any $s>0$.  On the other hand, the definition of $\Delta$ and the inequality ($\ref{Delta_ineq}$) imply $\Delta\geq rv-(N-2)\Delta$, i.e.
\begin{align}
N\Delta-rv\geq \Delta>0.\label{eqDelta}
\end{align}
The inequality ($\ref{goal1}$) will therefore follow if we show $\dot{r}(s)<0$ and $\dot{v}(s)>0$ for $s>0$.

We begin with $\dot{r}$, showing that $\dot{r}<0$ on some interval $(0,\delta_0)$.  Let $s\geq 0$ be given, and note that differentiating ($\ref{eq123}$) with respect to $v_0$ gives
\begin{align}
\nonumber &(N-2)r^{N-3}\dot{r}\partial_1\phi(s,N-1)\\
&\hspace{0.2in}=(N-1)r^{N-2}\dot{r}v+\int_0^r (N-1)\rho^{N-2}\partial_{v_0}v(\rho;v_0)d\rho.\label{eq_diff123}
\end{align}

We would now like to express the values $\partial_{v_0}[v(\rho;v_0)]$ in terms of the function $\dot{v}(\cdot)=\frac{d}{dv_0}[v(r(\cdot);v_0)]$.  To accomplish this, recall that $\dot{r}=\frac{d}{dv_0}r$ and note that for any $s\geq 0$, we have
\begin{align*}
\dot{v}(s)=\frac{d}{dv_0}[v(r(s);v_0)]=\partial_r v(r(s))\dot{r}(s)+\partial_{v_0}v(r(s);v_0)
\end{align*}
Now, fixing $\rho>0$ and setting $s_\rho=v'(\rho;v_0)$, we obtain
\begin{align}
\dot{v}(s_\rho)=s_\rho\dot{r}(s_\rho)+\partial_{v_0}[v(\rho;v_0)],\label{eq_vsrho}
\end{align}
where we have recalled that $(\partial_r v)(r(s))=s$ by the definition of $r(s)$, and that the choice of $s_\rho$ implies $r(s_\rho)=\rho$.  Combining ($\ref{eq_vsrho}$) with ($\ref{eq_diff123}$), we obtain
\begin{align}
-(N-1)\dot{r}r^{N-3}\Delta&=\int_0^r (N-1)\rho^{N-2}(\dot{v}(s_\rho)-s_\rho\dot{r}(s_\rho))d\rho.\label{eq34}
\end{align}
The desired local negativity of $\dot{r}$ now follows by observing that $\dot{v}(0)=1$, $\dot{r}(0)=0$ and $r(0)=0$ so that the right hand side of ($\ref{eq34}$) is positive for $s$ is sufficiently small (recall that the convexity of $v$ shows that $r(s)$ is increasing in $s$).

Let $I\subset\mathbb{R}$ be the largest interval containing $0$ for which $\dot{r}<0$ on $I$.  We claim that $\dot{v}>0$ on $I$.  To obtain this, note that $\dot{v}(0)=1$ and the continuity of $\dot{v}$ imply that $\dot{v}>0$ on some interval $(0,\delta_1)$ and let $I_1$ denote the largest such interval.  Suppose for contradiction that $\sup I_1<\sup I$.  By continuity we have $\dot{v}(\sup I_1)=0$.  Let $s\in I_1$ be given and use ($\ref{eq_rs}$) and $\partial_s[v(r(s))]=s\partial_s r$ to write
\begin{align*}
\partial_s \dot{v}&=\partial_{v_0}\left[\frac{sr\partial^2_{11}\phi(s,N-1)}{(N-1)\Delta}\right]\\
&=\frac{s\partial^2_{11}\phi(s,N-1)}{(N-1)\Delta^2}((\Delta-rv)\dot{r}-r^2\dot{v})\\
&>-\frac{sr^2\dot{v}\partial^2_{11}\phi(s,N-1)}{(N-1)\Delta^2}\\
&>-\frac{(N-1)^2sr^2\dot{v}\partial^2_{11}\phi(s,N-1)}{[\partial_1\phi(s,N-1)]^2},
\end{align*}
where to obtain the first inequality we have noted that the definition of $\Delta$ along with ($\ref{eq_phi_ineq}$) gives $\Delta-rv=-\frac{N-2}{N-1}\partial_1 \phi(s,N-1)<0$, while to obtain the second inequality we have used ($\ref{Delta_ineq}$) and the positivity of $\dot{v}$ (that is, $s\in I_1$).  Integrating this differential inequality, we obtain
\begin{align}
\dot{v}(s)>\exp\left(-\int_0^{s} \frac{(N-1)^2sr^2\partial^2_{11}\phi(s,N-1)}{[\partial_1\phi(s,N-1)]^2}ds\right)>0\label{eq11a}
\end{align}
with $s=\sup I_1$, which contradicts $\dot{v}(\sup I_1)=0$.  Thus $\dot{v}>0$ on $I$.

We now conclude by showing that $I=(0,\infty)$, and thus $(\ref{goal1})$ holds as desired.  Suppose for contradiction that $\sup I<\infty$.  The continuity of $\dot{r}$ then implies $\dot{r}(\sup I)=0$, so that
\begin{align*}
\dot{V}(\sup I)=\omega_{N-1}r(\sup I)^{N-1}\dot{v}(\sup I)
\end{align*}
Observing that ($\ref{eq11a}$) holds for $s=\sup I$, we obtain $\dot{V}(\sup I)>0$, contradicting 
\begin{align*}
\dot{V}(\sup I)\leq 0,
\end{align*}
which follows from the expression ($\ref{dotV}$) for $\dot{V}$ combined with the bound ($\ref{eqDelta}$) and the inequalities $\dot{r}<0$ and $\dot{v}>0$ on $I$.   Since $u_0>0$ was arbitrary, this completes the proof of Theorem $\ref{thm_unique}$.

We finish the proof by observing that the restriction to $\omega\in (-f(e_N),0)$ imposes no loss of generality.  In particular, although the change of variables used above becomes singular at the point where $r'_E(t)=0$ (see Figures $\ref{fig-mina}$ and $\ref{fig-minbc}$) when $\omega\geq 0$ and, moreover, the profile of minimizers is no longer a graph for $\omega>0$, one can recover the argument by noting that the convexity of minimizers implies that the minimizers divides into ``top'' and ``bottom'' sets where $r_E(t)$ is monotone.  The extension to general $\omega\in (-f(e_N),f(-e_N))$ then carries through as in the isotropic case \cite{Finn1,FinnBook} (where the same singularity is present) by repeating the analysis on each of these sets.
\end{proof}

\appendix

\section{A one-dimensional lemma concerning local convexity}

We now recall an elementary lemma showing that if a continuous function on $\mathbb{R}$ fails to be concave, then it has a region of local strict convexity.  This result may be seen as a contrapositive formulation of the statement that for continuous functions, local concavity implies concavity.  We remark that this lemma is used to establish the convexity of symmetric minimizers in the proof of Theorem \ref{thm_concave}.
\begin{lem}
\label{lem_concave}
Let $I\subset\mathbb{R}$ be an open interval and suppose that $r:I\rightarrow\mathbb{R}$ is a continuous function.  Let $t_1,t_2\in I$ be given along with $s\in (0,1)$ such that
\begin{align}
r(st_1+(1-s)t_2)<sr(t_1)+(1-s)r(t_2).\label{Aeqs0}
\end{align}
Then for every $\epsilon>0$ there exists $t'_1,t'_2\in (t_1,t_2)$ with $t'_1<t'_2$ and $|t'_1-t'_2|<\epsilon$ such that for a.e. $s'\in (0,1)$,
\begin{align}
\label{Ahyp1}r(s't'_1+(1-s')t'_2)<s'r(t'_1)+(1-s')r(t'_2).
\end{align}
\end{lem}

\begin{proof}
We begin by defining the set of slopes of any piecewise linear components of the graph of $r$,
\begin{align*}
M=\{m\in\mathbb{R}:\exists\, t_1<a_m<b_m\leq t_2\,\textrm{s.t.}\,\frac{r(t)-r(a_m)}{t-a_m}=m\,\textrm{for}\,a_m<t<b_m\},
\end{align*}
Note that $m,n\in M$, $m\neq n$ implies that $(a_m,b_m)\cap (a_n,b_n)=\emptyset$.  We therefore have the bound $\sum_{m\in M} b_m-a_m\leq t_2-t_1$, so that $M$ is at most countable.

By choosing $t_2$ slightly smaller, we can ensure 
\begin{align}
\label{Aeqasdf1}
\frac{r(t_2)-r(t_1)}{t_2-t_1}\not\in M.
\end{align}
Note that this perturbation of $t_2$ may cause $s$ to no longer obey ($\ref{Aeqs0}$); however, for small perturbations, the continuity of $r$ implies that this property can be restored by a suitable perturbation of $s$.  When necessary, we therefore let $s$ refer to this new value.  For $\mu\in (0,1)$, define
\begin{align*}
f(\mu)=\mu r(t_1)+(1-\mu)r(t_2)-r(\mu t_1+(1-\mu)t_2),
\end{align*}
so that $f(\mu)$ is the vertical distance from the point $(\mu t_1+(1-\mu)t_2,r(\mu t_1+(1-\mu)t_2))$ to the line connecting the points $(t_1,r(t_1))$ and $(t_2,r(t_2))$.  Then $f$ is a continuous function with $f(0)=f(1)=0$ and $f(s)>0$.  By continuity, we may find $s_*\in (0,1)$ such that
\begin{align}
f(s_*)=\sup_{\mu \in (0,1)} f(\mu).\label{Aeqasdf}
\end{align}
Fix $\delta>0$ to be determined later in the argument.  Then $\{s:f(s)>(1-\delta) f(s_*)\}$ is a nonempty open set, which can be written as a union of disjoint open intervals.  Let $I$ be the interval containing $s_*$ and set $s_1=\inf I$, $s_2=\sup I$.  Then $f$ continuous implies $f(s_1)=f(s_2)=(1-\delta) f(s_*)$.  Define $t'_1=s_1t_1+(1-s_1)t_2$ and $t'_2=s_2t_1+(1-s_2)t_2$.  The construction of $I$ then shows that ($\ref{Ahyp1}$) holds.

Letting $\epsilon>0$ be given, it remains to check that $\delta$ can be chosen to ensure $t'_2-t'_1<\epsilon$.  For this purpose, note that ($\ref{Aeqasdf1}$) and ($\ref{Aeqasdf}$) imply that $s_1$ and $s_2$ tend to $s_*$ as $\delta\rightarrow 0$.  Indeed, if (up to subsequences) $s_i\rightarrow s''_i$, $t'_i\rightarrow t''_i$, $i=1,2$ as $\delta\rightarrow 0$ and $s''_1<s_*$,  then we have $f(s''_1)=f(s''_2)=f(s_*)$.   A simple calculation then shows that $r$ is linear on the interval $[t''_1,t''_2]$ with slope 
\begin{align*}
\frac{r_E(t_2)-r_E(t_1)}{t_2-t_1}
\end{align*}
contradicting ($\ref{Aeqasdf1}$).  We may therefore choose $\delta$ sufficiently small so that $t'_2-t'_1<\epsilon$, which completes the proof of the lemma.
\end{proof}

\section{Change of variables for the ODE characterization}

In this appendix, for completeness we give the derivation of ($\ref{eq_chv2}$) in the case $u'<0$, starting from the Euler-Lagrange equation ($\ref{eq_chv1}$).  We begin by recalling that ($\ref{eq_chv2}$) was
\begin{align*}
-\frac{d}{dr}\left[\Lambda r^{N-2}\partial_1\phi(-\Lambda u'(r),N-1)\right]&=(N-1)(u(r)+\lambda)r^{N-2}
\end{align*}
Now, note that $\phi$ positively $1$-homogeneous implies
\begin{align*}
\partial_2\phi(\lambda a,\lambda b)=\partial_2\phi(a,b)
\end{align*}
and
\begin{align*}
\phi(a,b)=a\partial_1\phi(a,b)+b\partial_2\phi(a,b)
\end{align*}
for every $\lambda>0$.  We therefore obtain
\begin{align*}
&-\frac{d}{dr}\left[(N-1)r^{N-2}\partial_2\phi(\Lambda, -\frac{N-1}{u'})\right]\\
&\hspace{0.2in}=-\frac{d}{dr}\left[(N-1)r^{N-2}\partial_2\phi(-\Lambda u',N-1)\right]\\
&\hspace{0.2in}=-\frac{d}{dr}\left[r^{N-2}\phi(-\Lambda u',N-1)+\Lambda u' r^{N-2}\partial_1\phi(-\Lambda u', N-1)\right]\\
&\hspace{0.2in}=-(N-2)r^{N-3}\phi(-\Lambda u',N-1)+\Lambda r^{N-2}u''\partial_1\phi(-\Lambda u', N-1)\\
&\hspace{0.4in}-\frac{d}{dr}\left[\Lambda r^{N-2}\partial_1\phi(-\Lambda u',N-1)\right]u'-\Lambda r^{N-2}\partial_1\phi(-\Lambda u',N-1)u''\\
&\hspace{0.2in}=-(N-2)r^{N-3}\phi(-\Lambda u',N-1)-\frac{d}{dr}\left[\Lambda r^{N-2}\partial_1\phi(-\Lambda u',N-1)\right]u'.
\end{align*}
Substituting this into ($\ref{eq_chv1}$) and again invoking the positive $1$-homogeneity of $\phi$ gives
\begin{align*}
&-\frac{(N-2)r^{N-3}}{u'}\phi(-\Lambda u',N-1)-\frac{d}{dr}\left[\Lambda r^{N-2}\partial_1\phi(-\Lambda u',N-1)\right]\\
&\hspace{0.2in}=-\frac{(N-2)r^{N-3}}{u'}\phi(-\Lambda u',N-1)+(N-1)(u+\lambda)r^{N-2}
\end{align*}
which gives ($\ref{eq_chv2}$) as desired.

\bigskip
\footnotesize
\noindent\textit{Acknowledgments.}

This work is part of the author's Ph.D. thesis at the University of Texas at Austin (2012).  The author would like to thank A. Figalli for fruitful discussions throughout the preparation of the work.  The author would also like to thank F. Maggi, E. Indrei, and the anonymous referee for helpful conversations and comments.  This material is based upon work supported by the National Science Foundation under Award No. DMS-1204557.

\end{document}